\documentclass[10pt,reqno]{amsart}

\usepackage{amssymb}
\usepackage{amsmath}
\usepackage{amsfonts, dsfont}

\usepackage{amsthm} 
\usepackage[utf8]{inputenc} 
\usepackage{enumitem} 

\usepackage{graphicx}

\usepackage{mathrsfs}

\usepackage[utf8]{inputenc}
\usepackage{anyfontsize}

\usepackage{tikz}
\usetikzlibrary{calc,intersections,arrows.meta,patterns}

\DeclareGraphicsExtensions{.pdf,.png,.jpg}

\usepackage{color}

\theoremstyle{theorem} 
\newtheorem{thm}{Theorem}[section]

\newtheorem{lemma}[thm]{Lemma}
\newtheorem{prop}[thm]{Proposition}
\newtheorem{assumption}[thm]{Assumption}

\theoremstyle{definition}   
\newtheorem{defn}[thm]{Definition}

\theoremstyle{remark}  
\newtheorem{remark}[thm]{Remark}
\newtheorem{example}[thm]{Example}

\newcommand\bK{\mathbb{K}}
\newcommand\bL{\mathbb{L}}

\newcommand\bR{\mathbb{R}}

\newcommand\bH{\mathbb{H}}
\newcommand\bZ{\mathbb{Z}}

\newcommand\bD{\mathbb{D}}
\newcommand\bS{\mathbb{S}}

\newcommand\bE{\mathbb{E}}
\newcommand\bN{\mathbb{N}}

\newcommand\bP{\mathbb{P}}
\newcommand\bU{\mathbb{U}}

\newcommand\tbf{\mathbf{f}}

\newcommand\cA{\mathcal{A}}

\newcommand\cC{\mathcal{C}}
\newcommand\cD{\mathcal{D}}
\newcommand\cF{\mathcal{F}}

\newcommand\cH{\mathcal{H}}

\newcommand\cK{\mathcal{K}}
\newcommand\cL{\mathcal{L}}
\newcommand\cP{\mathcal{P}}
\newcommand\cR{\mathcal{R}}

\newcommand\cM{\mathcal{M}}
\newcommand\cT{\mathcal{T}}
\newcommand\cO{\mathcal{O}}

\newcommand\rF{\mathscr{F}}

\newcommand{\domain}{\mathcal{O}}

\newcommand{\mysection}[1]{\section{#1}
\setcounter{equation}{0}}

\begin{document}

\title[Parabolic equations on conic domains]{Sobolev space theory  and H\"older estimates for the stochastic partial  differential equations  on conic and polygonal domains}
 
\thanks{The first and  third authors were  supported by the National Research Foundation of Korea(NRF) grant funded by the Korea government(MSIT) (No. NRF-2020R1A2C1A01003354)} 
\thanks{The second author was  supported by the National Research Foundation of Korea(NRF) grant funded by the Korea government(MSIT) (No. NRF-2019R1F1A1058988)}

\author{Kyeong-Hun Kim}
\address{Kyeong-Hun Kim, Department of Mathematics, Korea University, Anam-ro 145, Sungbuk-gu, Seoul, 02841, Republic of Korea}
\email{kyeonghun@korea.ac.kr}

\author{Kijung Lee}
\address{Kijung Lee, Department of Mathematics, Ajou University, Worldcup-ro 206, Yeongtong-gu, Suwon, 16499, Republic of Korea}
\email{kijung@ajou.ac.kr}

\author{jinsol Seo}
\address{Jinsol Seo, Department of Mathematics, Korea University, Anam-ro 145, Sungbuk-gu, Seoul, 02841, Republic of Korea}
\email{seo9401@korea.ac.kr}



\begin{abstract}
We establish existence, uniqueness, and  Sobolev and H\"older regularity results for the stochastic partial differential equation 
\begin{equation*}
\begin{aligned}
du=\Big(\sum_{i,j=1}^d a^{ij}u_{x^ix^j}&+f^0+\sum_{i=1}^d f^i_{x^i}\Big)dt\\
&+\sum_{k=1}^{\infty}g^kdw^k_t, \quad t>0, \,x\in \cD
\end{aligned}
\end{equation*}
given with non-zero initial data. Here $\{w^k_t: k=1,2,\cdots\}$ is a family of independent Wiener processes defined on a probability space $(\Omega, \bP)$, $a^{ij}=a^{ij}(\omega,t)$ are merely  measurable functions on $\Omega\times (0,\infty)$, and  $\cD$ is either a polygonal domain in $\bR^2$ or an arbitrary dimensional conic domain of the type
\begin{equation}
\label{conic}
\cD(\cM):=\left\{x\in \bR^d :\,\frac{x}{\lvert x\rvert}\in \cM\right\}, \quad \quad \cM\subsetneq S^{d-1},  \quad (d\geq 2)
\end{equation}
where $\cM$ is an open subset of $S^{d-1}$ with $C^2$ boundary. 
We measure the Sobolev and H\"older  regularities of arbitrary  order derivatives of the solution  using a system of mixed weights consisting of appropriate powers of the distance to the vertices and of the distance to the boundary. 
The ranges of admissible powers of the distance to the vertices and  to the boundary are sharp.
\end{abstract}

\maketitle

\mysection{Introduction}\label{sec:Introduction}
The goal of this article is to present a  Sobolev space theory and  H\"older regularity results for  the stochastic partial differential equation (SPDE)
\begin{align}\label{main equation in introduction}
d u =\left(\sum_{i,j}^d a^{ij}u_{x^ix^j}+f^0+\sum_{i=1}^d f^i_{x^i}\right)dt +\sum^{\infty}_{k=1} g^kdw_t^k, \,\,\, t>0\,; \,\,\, u(0,\cdot)=u_0
\end{align}
defined on either  multi-dimensional conic domains $\cD(\cM)$ (see \eqref{conic}) or  two dimensional polygonal domains.   Here, $\cM$ is an open subset of $S^{d-1}$ with $\cC^2$ boundary, $\{w^k_t: k=1,2,\cdots\}$ is an infinite sequence of independent one dimensional Wiener processes, and the coefficients $a^{ij}$ are  merely measurable  functions of $(\omega,t)$ with the uniform parabolicity condition; see Assumption \ref{ass coeff} below. 

To give the reader a flavor of our results in this article we state a particular one, an estimate, below:  Let $\cD=\cD(\cM)$ be a conic domain in $\bR^d$, $\rho(x):=dist(x,\partial \cD)$, and $\rho_{\circ}(x):=\lvert x\rvert$. Then for the solution $u$ of \eqref{main equation in introduction} with zero boundary and zero initial conditions, the following holds for any $p\geq 2$:
\begin{align}\label{main estimate simple}
&\bE \int^T_0 \int_{\cD} \left(\lvert\rho^{-1}u\rvert^p+ \lvert u_x\rvert^p\right) \rho_{\circ}^{\theta-\Theta}\rho^{\Theta-d}\, dx\,dt \nonumber\\
\leq\quad & C\,\bE \int^T_0 \int_{\cD} \Big( \lvert \rho f^0\rvert^p+\sum_{i=1}^d\lvert f^i\rvert^p +\lvert g\rvert_{l_2}^p\Big) \rho_{\circ}^{\theta-\Theta}\rho^{\Theta-d}\, dx\,dt
\end{align}
with $d-1<\Theta<d-1+p$ accompanied with the sharp admissible  range of $\theta$; see \eqref{theta con intro} below.  Also see \eqref{main estimate intro} for  higher order derivative estimates. Unlike the range of $\Theta$, the range of $\theta$ is affected by the shape of domain $\cD$, which is determined by $\cM$. 
 Estimate \eqref{main estimate simple}, if $\rho_{\circ}$ is replaced by the distance to the set of vertices, also holds when  $\cD$ is a (bounded) polygonal domain in $\bR^2$.  Regarding H\"older regularity, we have for instance, if $1-\frac{d}{p}=\delta>0$, 
$$
\lvert \rho^{-1+\frac{\Theta}{p}} \rho^{(\theta-\Theta)/p}_{\circ}u(\omega,t,\cdot)\rvert_{ \cC(\cD)}+
 [\rho^{-1+\delta+\frac{\Theta}{p}} \rho^{(\theta-\Theta)/p}_{\circ} u( \omega,t,\cdot)]_{\cC^{\delta}(\cD)}<\infty, 
$$
for a.e. $(\omega,t)$. In particular, 
\begin{align}
 \lvert u(\omega,t,x)\rvert\leq C(\omega,t) \rho^{1-\frac{\Theta}{p}}(x) \rho^{(-\theta+\Theta)/p}_{\circ}(x)\quad \text{for all }x\in\cD \label{Holder simple}.
\end{align}
Estimate \eqref{Holder simple} shows how $\theta$ and $\Theta$ are involved in measuring the boundary behavior of the solution with respect to $\rho$ and $\rho_{\circ}$. See Theorem \ref{cor 8.10} and Theorem \ref{cor 8.23} for the full H\"older regularity results with respect to both space and time variables.

To position our results in the context of regularity theory of stochastic parabolic equations, let us provide a stream of historical remarks. 

The $L_p$-theory ($p\geq 2$) of   equation \eqref{main equation in introduction}  defined on the entire space $\bR^d$ was first introduced by N.V. Krylov \cite{Krylov 1999-4, Krylov 1996}.  In these articles the author used  an analytic approach and proved the maximal regularity estimate
\begin{align}
 \label{krylov lp}
\|u_x\|_{\bL_p(T)}\leq C\Big(\|f^0\|_{\bL_p(T)}+\sum_{i=1}^d \|f^i\|_{\bL_p(T)}
+\||g|_{\ell_2}\|_{\bL_p(T)}\Big), \qquad p\geq 2,
\end{align}
provided that $u(0,\cdot)\equiv 0$, where $\bL_p(T):=L_p(\Omega\times (0,T); L_p(\bR^d))$.

As for other approaches on Sobolev regularity theory,  the method based  on $H^{\infty}$-calculus is also available in the literature. This approach  was introduced in \cite{Veraar},  in which  the maximal regularity of $\sqrt{-A}u$ is obtained  for the stochastic convolution 
$$
u(t):=\int^t_0 e^{(t-s)A} g(s) dW_H (s).
$$
 Here,  $W_H(t)$ is a cylindrical Brownian motion on a Hilbert space $H$, and  the operator $-A$ is assumed to admit a bounded $H^{\infty}$-calculus of angle less than $\pi/2$ on $L^q(\cO)$, where $q\geq 2$ and $\cO$ is a domain in $\bR^d$.  The result of \cite{Veraar}  generalizes \eqref{krylov lp} with $f^i=0$, $i=1,\ldots,d$ as one can take $A=\Delta$ and $\cO=\bR^d$. 
 
 One advantage of the approach based on $H^{\infty}$-calculus is that it provides a unified way of handling a class of differential operators satisfying the above mentioned condition. However this approach is not applicable  for SPDEs with operators depending on $(\omega,t)$, and even  the simplest  case $A=\Delta$, it is needed that $\partial \cO$ is regular enough, that is $\partial \cO \in \cC^2$. Compared to the approach based on $H^{\infty}$-calculus, Krylov's analytic approach works well for SPDEs with operators depending also on $(\omega,t)$, and it also provides the arbitrary order regularity of solutions without much extra efforts even  under  weaker smoothness condition on  domains.

Since the work of \cite{Krylov 1999-4, Krylov 1996} on $\bR^d$,  the analytic approach has been further used for the regularity theory of SPDEs  on half space \cite{Krylov 1999-2, Krylov 1999-22, KK2004-2} and on $\cC^1$-domains \cite{KK2004, Kim2004, Kim2004-2}. 
The major obstacle of studying SPDEs on domains is that, unless certain compatibility conditions (cf. \cite{Flandoli}) are fulfilled, the second and higher order derivatives of solutions to SPDEs blow up near the boundary, and such blow-ups are inevitable even on $\cC^{\infty}$-domains. Hence, one needs appropriate weight system to understand the behavior of  solutions  near the boundary. 

 It is shown  in \cite{Krylov 1999-2, KK2004, Kim2004} that if domains satisfy $\cC^1$ boundary condition, then blow-ups of derivatives of solutions can be described very accurately by a  weight system introduced in \cite{Krylov 1999-1, KK2004, Lo1}.
This weight system is  based solely on the distance to the boundary. Surprisingly enough, under this weight system it is  irrelevant  whether domains have $\cC^{\infty}$-boundary or $\cC^1$-boundary, that is, the regularity of solutions is not affected by the smoothness of  the boundary provided that the boundary is at least of class $\cC^1$. To be more specific, let $\cO$ be a $\cC^1$-domain, $\rho(x)=dist (x,\partial \cO)$, then  it   holds that (see \cite{Kim2004, KK2004})  for  any $d-1<\Theta<d-1+p$,
  \begin{align}
  \label{eqn 9.3.1}
&\bE \int^T_0 \int_{\cO}(\vert \rho^{-1}u\vert + \vert u_x\vert )^p \rho^{\Theta-d}dt \nonumber\\
\leq\,& C\bE \int^T_0 \int_{\cO}\big( \vert \rho f^0\vert ^p+\sum_{i=1}^d \vert f^i\vert ^p+\vert g\vert ^p_{\ell_2} \big)^p \rho^{\Theta-d}dt.
 \end{align}  
  The condition $\Theta \in (d-1, d-1+p)$  is sharp and is not affected by further smoothness of  $\partial \cO$  as long as  $\partial \cO\in \cC^1$. Note that estimate  \eqref{eqn 9.3.1} with smaller $\Theta$ gives better decay of solutions near the boundary than that with larger $\Theta$. In particular, we  have $u(\omega,t,\cdot)\in W^{1,p}_{0}(\cO)$ from \eqref{eqn 9.3.1} if $\Theta\leq d$.
   
  As for results on non-smooth domains, that is  $\partial \cO \not\in \cC^1$,   very few fragmentary results are known.  It turns out  that \eqref{eqn 9.3.1} holds true  on general Lipschitz domains   if $\Theta \approx d-2+p$ (see \cite{Kim2014}), and hence  the case $\Theta=d$ is not included in general if $p>2$.    An example in \cite{Kim2014} also shows that if  $\Theta<p/2$, then estimate \eqref{eqn 9.3.1} fails  to hold even on  simple wedge domains of the type 
\begin{equation}\label{angular domain1}
\cD^{(\kappa)}=\big\{(r\cos \eta, r\sin \eta)\in \bR^2: r>0, \eta\in (-\kappa/2, \kappa/2)\big\}, \quad \kappa < 2\pi.
\end{equation}
The vertex $0$ makes the boundary non-smooth and changes the game.

Our interest on conic and polygonal domains arises from such question which, in particular, ask if estimates similar to \eqref{eqn 9.3.1} hold on such simple Lipschitz domains.  We got the clue of the problem from 
  a PDE result on conic domains \cite{Kozlov Nazarov 2014} (also see \cite{Na, Sol2001})  which is similar to \eqref{eqn 9.3.1}, without the term $g=(g^1, g^2,\cdots)$ of course. 
  It uses the weight based only on the distance to the vertex. 
 A work  on SPDE using a weight system based only on the distance to the vertex is introduced in \cite{CKLL 2018} 
(also see \cite{CKL 2019+}), in which we studied the model case of  $d=2$ and $a^{ij}=\delta_{ij}$ for a starter of the program.

Even for the model case considered in \cite{ CKL 2019+,CKLL 2018}  we struggled to have higher order derivative estimate and left the problem  as the future work. The main issue is to include the distance to the boundary  in our weight system to have a satisfactory regularity relation between solutions and the inputs. In fact, there was an omen of aforementioned difficulty that is implied in the Green's function estimate used in  \cite{CKLL 2018}  and \cite{CKL 2019+}. The estimate dominating Green's function does not vanish at the boundary although it does at the vertex. We need more refined Green's function estimate for the starter of a satisfactory regularity result.

We then set a program of three steps: (i) preparing a refined $d$-dimensional Green's function estimate for  operators with measurable coefficients  (ii) preparing PDE result (iii) establishing SPDE result addressing the higher order derivative estimates. First two steps are done in \cite{Green} and \cite{ConicPDE},  and this article fulfills the last step.  In \cite{Green} the refined Green's function estimate involves both the distance to the vertex and the distance to the boundary and  it now vanishes at all the points on the boundary with informative decay rate near the boundary. The work \cite{ConicPDE} fully makes use of what we prepared in \cite{Green} and it is designed to serve this article well.

Now let us explain our $L_p$-regularity result  in more detail. Recall
$
\rho_{\circ}(x):=\vert x\vert  \quad  \text{and} \quad  \rho(x):=d(x,\partial \cD),
$
which denote the distance from $x$ to vertex and to the boundary of the conic domain $\cD=\cD(\cM)$, respectively.  We prove that for any $p\ge 2$ and $n=0,1,2,\cdots$, the estimate   
\begin{align}
&\bE \int^T_0 \int_{\cD} \left(\vert \rho^{-1}u\vert ^p+\vert u_x\vert ^p+\cdots+ \vert \rho^{n}D^{n+1}u\vert ^p\right) \rho_{\circ}^{\theta-\Theta}\rho^{\Theta-d}\, dx\,dt \nonumber
 \\
\leq& C \bE
  \int^T_0 \int_{\cD} \Big( \vert \rho f^0\vert ^p+\cdots+\vert \rho^{n+1}D^nf^0\vert ^p \nonumber 
  \\
  &\quad \quad \quad \quad \quad\,\,\,\, +\sum^d_{i=1}\vert f^i\vert ^p+\cdots+\sum_{i=1}^d\vert \rho^{n}D^nf^i\vert ^p \nonumber \\
  &\quad \quad \quad \quad \quad\,\,\,\, +\vert g\vert _{\ell_2}^p+\cdots+\vert \rho^{n}D^{n}g\vert _{\ell_2}^p\Big)
 \rho_{\circ}^{\theta-\Theta}\rho^{\Theta-d}\, dx\,dt  \label{main estimate intro}
\end{align}
holds for the solution $u=u(\omega,t,x)$ to equation \eqref{main equation in introduction} with zero initial condition, provided  that
  \begin{equation}
     \label{theta con intro}         
d-1<\Theta<d-1+p, \quad\,\,  p(1-\lambda^+_c)<\theta<p(d-1+\lambda^-_c).
 \end{equation}
  Here, $\lambda^+_c$ and $\lambda^-_c$ are positive constants  which depend on  $\cM$   and are defined in   Definition \ref{lambda} below (also see Proposition \ref{critical exponents} and Remark \ref{example proposition}). The same estimate holds for polygonal domains in $\bR^2$.  Estimate \eqref{main estimate intro} with  condition \eqref{theta con intro} is indeed an (seamless) extension of \cite{ConicPDE} to SPDEs, and what is satisfactory is that the ranges of  $\Theta$ and $\theta$ in \eqref{theta con intro} are not shrunken smaller than the ranges for the deterministic parabolic equation. For this however very delicate computation is required and providing the work done successfully is one of main purposes of this article.

Finally, we want to summarize the improvement in this article over the results in \cite{CKLL 2018}  and \cite{CKL 2019+}. Our domains $\cD(\cM)$ in $\bR^d$, $d\ge 2$, generalize two dimensional angular domains \eqref{angular domain1}; the choice of $\cM$ is much richer when $d>2$. Our operator $\sum_{i,j}a^{ij}(\omega,t)D_{ij}$ far generalizes Laplacian operator $\Delta$ in \cite{CKLL 2018}  and \cite{CKL 2019+}. These generalizations make computation much more involved, especially, for the stochastic part of the solution. Also, thanks to the mixed weight system, we can now study the higher order derivatives in an appropriate manner and implementing it requires quite a work.  Moreover, in this article we do not pose zero initial condition and hence we propose right function spaces for the initial condition in terms of regularity relations between inputs and output, where the initial condition is one of inputs. This result is new  even  for deterministic PDEs on conic domains. H\"older regularity results based on aforementioned improvements are also new even for PDEs on conic domains.

This article is organized as follows. In Section 2 we introduce some properties of weighted Sobolev spaces and  present our main results on conic domains, including H\"older regularity results.  In Section 3 we estimate weighed $L_p$ norm of the zero-th order derivative of the solution  on conic domains based on the solution representation via Green's function and elementary but highly involved computations. The estimates of the derivatives of the solution on conic domains are obtained in Section 4 and the proof of the main results on conic domains are posed there, too. In section 5 we establish a regularity theory on polygonal domains in $\bR^2$.
\vspace{2mm}

\noindent\textbf{Notations.}
\begin{itemize}

\item We use $:=$ to denote a definition.

\item  For a measure space $(A, \cA, \mu)$, a Banach space $B$ and $p\in[1,\infty)$, we write $L_p(A,\cA, \mu;B)$ for the collection of all $B$-valued $\bar{\cA}$-measurable functions $f$ such that
$$
\|f\|^p_{L_p(A,\cA,\mu;B)}:=\int_{A} \lVert f\rVert^p_{B} \,d\mu<\infty.
$$
Here, $\bar{\cA}$ is the completion of $\cA$ with respect to $\mu$. We will drop $\cA$ or $\mu$ or even $B$ in $L_p(A,\cA, \mu;B)$   when they   are obvious from the context. 

\item $\bR^d$ stands for the $d$-dimensional Euclidean space of points $x=(x^1,\cdots,x^d)$, $B_r(x):=\{y\in \bR^d: \vert x-y\vert <r\}$, 
 $\bR^d_+:=\{x=(x^1,\ldots,x^d): x^1>0\}$, and $S^{d-1}:=\{x\in \bR^d: \vert x\vert =1\}$.
 
 \item For  a domain $\domain \subset \bR^d$, $B^{\domain}_R(x):=B_R(x)\cap \domain$ and  $Q^{\domain}_R(t,x):=(t-R^2,t]\times B^{\domain}_R(x)$.
 
 \item  $\bN$ denotes the natural number system, $\bN_0=\{0\}\cup \bN$,  and   $\bZ$ denotes the set of integers.

\item For $x$, $y$ in $\bR^d$,  $x\cdot y :=\sum^d_{i=1}x^iy^i$ denotes the standard inner product.

\item For a domain $\domain$ in $\bR^d$, $\partial \domain$ denotes the boundary of $\domain$.

\item  For  any multi-index $\alpha=(\alpha_1,\ldots,\alpha_d)$, $\alpha_i\in \{0\}\cup \bN$,   
$$
f_t=\frac{\partial f}{\partial t}, \quad f_{x^i}=D_if:=\frac{\partial f}{\partial x^i}, \quad D^{\alpha}f(x):=D^{\alpha_d}_d\cdots D^{\alpha_1}_1f(x).
$$
 We denote $\vert \alpha\vert :=\sum_{i=1}^d \alpha_i$.  For the second order derivatives we denote $D_jD_if$ by $D_{ij}f$. We often use the notation 
$\vert gf_x\vert ^p$ for $\vert g\vert ^p\sum_i\vert D_if\vert ^p$ and $\vert gf_{xx}\vert ^p$ for $\vert g\vert ^p\sum_{i,j}\vert D_{ij}f\vert ^p$.  We also use $D^m f$ to denote arbitrary partial derivatives  of order $m$ with respect to the space variable.

\item $\Delta_x f:=\sum_i D_{ii}f$, the Laplacian for $f$.

\item For $n\in \{0\}\cup \bN$,  $W^n_p(\domain):=\{f: \sum_{\vert \alpha\vert \le n}\int_{\domain}\vert D^{\alpha}f\vert ^p dx<\infty\}$, the Sobolev space.

\item  For a domain $\domain\subseteq\bR^d$ and a Banach space $X$ with the norm $\vert \cdot\vert _X$, $\cC(\domain;X)$ denotes the set of $X$-valued continuous functions $f$ in $\domain$ such that $\vert f\vert _{\cC(\domain;X)}:=\sup_{x\in\domain}\vert f(x)\vert _X<\infty$. Also, for $\alpha\in (0,1]$, we define the H\"older space
 $\cC^{\alpha}(\domain;X)$ as the set of all $X$-valued functions $f$ such that 
$$
\vert f\vert _{\cC^{\alpha}(\domain;X)}:=\vert f\vert _{\cC(\domain;X)}+[f]_{\cC^{\alpha}(\domain;X)}<\infty
$$
with the semi-norm $[f]_{\cC^{\alpha}(\domain;X)}$ defined by
$$
[f]_{\cC^{\alpha}(\domain;X)}=\sup_{x\neq y\in \domain} \frac{\vert f(x)-f(y)\vert _X}{\vert x-y\vert ^{\alpha}}.
$$
In particular, $\domain$ can be an interval in $\bR$.

\item  For a domain $\domain\subseteq\bR^d$, $\cC^{\infty}_c(\domain)$ is the the space of infinitely differentiable functions with compact support in $\domain$. $supp(f)$ denotes the support of the function $f$. Also, $\cC^{\infty}(\domain)$ denotes the the space of infinitely differentiable functions in $\domain$.

\item  For a distribution $f$ on  $\domain$ and  $\varphi\in \cC^{\infty}_c(\domain)$,  the expression $(f,\varphi)$ denote the evaluation of $f$ with the test function $\varphi$.


\item For functions $f=f(\omega,t,x)$ depending on $\omega\in\Omega$, $t\geq 0$ and $x\in\bR^d$, we usually drop the argument $\omega$  and  just write $f(t,x)$ when there is no confusion.

\item  Throughout the article, the letter $C$ denotes a finite positive constant which may have different values along the argument  while the dependence  will be informed;  $C=C(a,b,\cdots)$, meaning that  $C$ depends only on the parameters inside the parentheses.

\item  $A\sim B$ means that there exist constants $C_1, C_2>0$ independent of $A$ and $B$ such that  $A\leq C_1B \leq C_2A$. 

\item $d(x,\domain)$ stands for the distance between a point $x$ and a set $\domain\in\bR^d$.

\item $a \vee b =\max\{a,b\}$, $a \wedge b =\min\{a,b\}$. 

\item $1_U$ the indicator function on $U$.


\item We will use the following sets of functions (see \cite{Kozlov Nazarov 2014}).
\begin{itemize}
\item[-]
$\mathcal{V}(Q^{\mathcal{\domain}}_R(t_0,x_0))$ : the set of functions $u$ defined at least on $Q^{\mathcal{\domain}}_R(t_0,x_0)$ and satisfying
\begin{equation*}
\sup_{t\in(t_0-R^2,t_0]}\|u(t,\cdot)\|_{L_2(B^{\domain}_{R}(x_0))} +\|\nabla u\|_{L_2(Q^{\domain}_{R}(t_0,x_0))}<\infty.\nonumber
\end{equation*}
\item[-] 
$\mathcal{V}_{loc}(Q^{\domain}_R(t_0,x_0))$ : the set of functions $u$ defined at least on $Q^{\domain}_R(t_0,x_0)$ and satisfying
\begin{equation*}
u\in \mathcal{V}(Q^{\domain}_r(t_0,x_0)),  \quad \forall r\in (0,R).\nonumber
\end{equation*}
\end{itemize}

\end{itemize}

\mysection{SPDE on $d$-dimensional conic domains}\label{sec:Cone}

Throughout this article we assume $d\ge 2$. Let $\cM$ be a nonempty open set in $S^{d-1}:=\left\{x\in \bR^d\,:\,\vert x\vert =1\right\}$ and  $\overline{\cM}$ denotes the closure of $\cM$.  We assume $\overline{\cM}\neq S^{d-1}$,
and  define the $d$-dimensional conic domain $\mathcal{D}$ by 
$$
\mathcal{D}=\cD(\cM):=\Big\{x\in\mathds{R}^d\setminus\{0\} \ \Big\vert  \ \ \frac{x}{\vert x\vert}\in \mathcal{M} \Big\}.
$$
 When $d=2$, the shapes of conic domains are quite simple.  For instance, with a fixed angle $\kappa$ in the range of $\left(0,2\pi\right)$ we can consider 
\begin{equation}\label{wedge in 2d}
\mathcal{D}=\mathcal{D}^{(\kappa)}:=\left\{(r\cos\eta,\ r\sin\eta)\in\mathds{R}^2 \mid r\in(0,\ \infty),\ -\frac{\kappa}{2}<\eta<\frac{\kappa}{2}\right\}.
\end{equation}

\begin{figure}[ht]
\begin{tikzpicture}[> = Latex]
\begin{scope}
\clip (-2.4,-2.4)--(2.4,-2.4)--(2.4,2.4)--(-2.4,2.4);

\draw[->] (-2.4,0) -- (2.4,0);
\draw[->] (0,-2.4) -- (0,2.4);
\draw (0,0) -- (-2.3,1.725);
\draw (0,0) -- (-2.3,-1.725);
\draw[->] (0.5,0) arc(0:{180-atan(3/4)}:0.5);
\draw[<-] ({-180+atan(3/4)}:0.5) arc({-180+atan(3/4)}:0:0.5);
\begin{scope}
\clip (0,0)--(-2.3,-1.725)--(-2.2,-2.2)--(2.2,-2.2)--(2.2,2.2)--(-2.2,2.2)--(-2.3,1.725)--(0,0);
\foreach \i in {-4,-3.67,...,3}
{\draw (\i,-2.8)--(\i+1,2.8);}
\path[fill=white] (0.60,0.3) rectangle (1.05,1.1);
\path (0.85,0.35) node[above] {{\Large $\frac{\kappa}{2}$}};
\path[fill=white] (0.45,-0.3) rectangle (1.05,-1.1);
\path (0.75,-0.35) node[below] {{\small $-$}{\Large $\frac{\kappa}{2}$}};
\draw[fill=black] (0.5,0) circle (0.5mm);
\end{scope}
\end{scope}
\path (0,-3) node {$d=2$};
\end{tikzpicture}
\begin{tikzpicture}[> = Latex]

\begin{scope}

\begin{scope}[scale=0.8]
\begin{scope}[scale=0.9]
\clip (0,0) circle (3.6) ;
\draw (0,0) circle (1.7);
\draw (1.7,0) arc(0:-180:1.7 and 0.6);
\draw[dashed] (1.7,0) arc(0:180:1.7 and 0.6);
\clip (-3,-3) -- (3,-3) -- (3,3) -- (-3,3);
\draw (1.1,0.55) -- (6,3);
\draw (0.3,1.2) -- (1.5,6);
\draw[dashed] (0,0) -- (1.1,0.55);
\draw[dashed] (0,0) -- (1.1,1.1);
\draw[dashed] (0,0) -- (0.7,1.3);
\draw[dashed] (0,0) -- (0.3,1.2);
\draw (1.1,1.1) -- (3.3,3.3);
\draw (0.7,1.3) -- (2.1,3.9);
\end{scope}

\begin{scope}[scale=0.45]
\draw[fill=white][dashed] (1.1,0.55) .. controls (1.4,0.7) and (1.2,1) .. (1.1,1.1) .. controls (1.0,1.2) and (0.8,1.25) .. (0.7,1.3) .. controls (0.5,1.4) and (0.33,1.32) .. (0.3,1.2) .. controls (0.23,0.92) and (0.5,0.85) .. (0.6,0.7) .. controls (0.7,0.55) and (1,0.5) .. (1.1,0.55);
\end{scope}

\draw[dashed] (0,0) -- (0.6,0.7);

\begin{scope}[scale=0.9]
\draw[fill=white] (1.1,0.55) .. controls (1.4,0.7) and (1.2,1) .. (1.1,1.1) .. controls (1.0,1.2) and (0.8,1.25) .. (0.7,1.3) .. controls (0.5,1.4) and (0.33,1.32) .. (0.3,1.2) .. controls (0.23,0.92) and (0.5,0.85) .. (0.6,0.7) .. controls (0.7,0.55) and (1,0.5) .. (1.1,0.55);
\end{scope}

\begin{scope}[scale=1.8]
\draw[fill=white] (1.1,0.55) .. controls (1.4,0.7) and (1.2,1) .. (1.1,1.1) .. controls (1.0,1.2) and (0.8,1.25) .. (0.7,1.3) .. controls (0.5,1.4) and (0.33,1.32) .. (0.3,1.2) .. controls (0.23,0.92) and (0.5,0.85) .. (0.6,0.7) .. controls (0.7,0.55) and (1,0.5) .. (1.1,0.55);
\end{scope}

\draw (1.5,1.75) -- (0.6,0.7);

\begin{scope}[scale=0.9]

\clip (1.1,0.55) .. controls (1.4,0.7) and (1.2,1) .. (1.1,1.1) .. controls (1.0,1.2) and (0.8,1.25) .. (0.7,1.3) .. controls (0.5,1.4) and (0.33,1.32) .. (0.3,1.2) .. controls (0.23,0.92) and (0.5,0.85) .. (0.6,0.7) .. controls (0.7,0.55) and (1,0.5) .. (1.1,0.55);

\draw (0,1.7) arc(90:0:0.54 and 1.95);
\draw (0,1.7) arc(90:0:0.795 and 1.92);
\draw (0,1.7) arc(90:0:1.03 and 1.85);
\draw (0,1.7) arc(90:0:1.23 and 1.787);
\draw (0,1.7) arc(90:0:1.4 and 1.71);

\draw (-0.15,1.14) arc(-90:0:1.7*0.565 and 0.6*0.565);
\draw (-0.15,0.99) arc(-90:0:1.7*0.66 and 0.6*0.66);
\draw (-0.15,0.82) arc(-90:0:1.7*0.745 and 0.6*0.745);
\draw (-0.15,0.645) arc(-90:0:1.7*0.82 and 0.6*0.82);
\draw (-0.15,0.46) arc(-90:0:1.7*0.89 and 0.6*0.89);
\end{scope}
\end{scope}
\end{scope}
\path (0,-3) node {$d=3$};
\end{tikzpicture}
\caption{Cases of $d=2$ and $d=3$}
\end{figure}
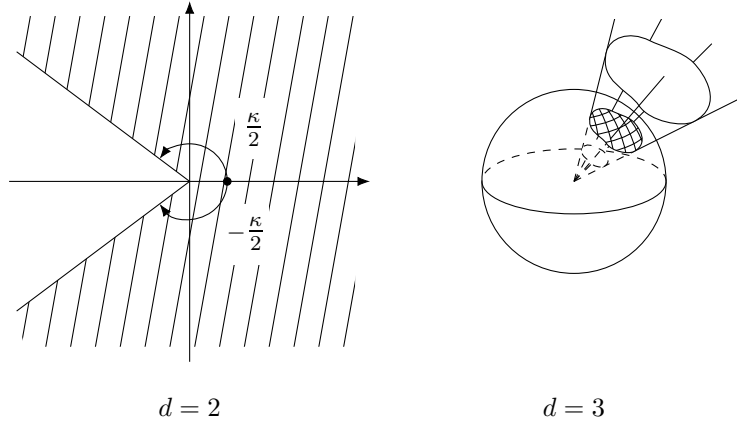


%

Let  $\{w_{t}^{k}\}_{k\in\bN}$ be a family of independent one-dimensional
Wiener processes  defined on a complete probability space $(\Omega,\mathscr{F},\bP)$ equipped with an  increasing filtration of
$\sigma$-fields $\mathscr{F}_{t}\subset\mathscr{F}$, each of which
contains all $(\mathscr{F},\bP)$-null sets.  By $\cP$  we  denote the predictable $\sigma$-field  on $\Omega \times (0,\infty)$ generated by $\mathscr{F}_{t}$.

In this article we study the  regularity theory of  the stochastic  partial differential equation 
\begin{equation}\label{stochastic parabolic equation}
d u =\Big( \cL u+f^0+\sum_{i=1}^d f^i_{x^i}\Big)dt +\sum^{\infty}_{k=1} g^kdw_t^k,\quad   t>0, \;x\in \cD(\cM)
\end{equation}
under the zero Dirichlet boundary condition.  Here
\begin{equation*}
\cL := \sum_{i,j=1}^d a^{ij}(\omega,t) D_{ij}.
\end{equation*}

\begin{itemize}

\item[-] Each of the stochastic integrals  in \eqref{stochastic parabolic equation} is understood as an It\^o stochastic integral against the given Wiener process.  

\item[-] The infinite sum of stochastic integrals is understood as the limit in probability (uniformly in $t$) of the finite sums of stochastic integrals. See Remark \ref{sto series}.

\end{itemize}




 Here are our assumptions on  $\cM$ and the diffusion coefficients. 

\begin{assumption}
\label{ass M}
The boundary $\partial \cM $ of $\cM$  in  $S^{d-1}$ is of class $\cC^2$.
\end{assumption}

\begin{assumption}
  \label{ass coeff}
The diffusion coefficients $a^{ij}$, $i,j=1,\cdots,d$,  are  real-valued $\cP$-measurable  functions of $(\omega,t)$, symmetric; $a^{ij}=a^{ji}$,  and satisfy the uniform parabolicity condition, i.e. there exist  constants $\nu_1, \nu_2>0$ such that for any $t\in\mathds{R}$, $\omega\in \Omega$ and $\xi=(\xi^1,\ldots,\xi^d)\in\mathds{R}^d$,
\begin{equation}
\nu_1 \vert \xi\vert ^2\le \sum_{i,j}a^{ij}(\omega, t)\xi_i\xi_j\le \nu_2 \vert \xi\vert ^2.  \label{uniform parabolicity}
\end{equation}
\end{assumption}

To explain our main result in the frame of weighted Sobolev regularity, we  introduce some  function spaces (c.f. \cite{CKL 2019+, ConicPDE}). These spaces collect the functions  whose weak derivatives can be measured by the help of appropriate weights consisting of powers of the distance to the vertex  and of the distance to the boundary.  Let  us define
$$
\rho_{\circ}(x)=\rho_{\circ,\cD}:=\vert x\vert ,\quad \quad \rho(x)=\rho_{\cD}(x):=d(x,\partial\cD).
$$
For $p\in(1,\infty)$, $\theta\in\bR$ and $\Theta\in \bR$,  we define 
$$
 L_{p,\theta,\Theta}(\cD):=L_p(\cD,\rho_{\circ}^{\theta-\Theta}\rho^{\Theta-d}dx),
$$
and for $m\in \bN_0$ define
$$
K^m_{p,\theta,\Theta}(\cD):=\{f\, :  \rho^{\vert \alpha\vert } D^{\alpha}f\in L_{p,\theta,\Theta}(\cD),  \, \,\vert \alpha\vert \leq m \}.
$$
The norm in  $K^m_{p,\theta,\Theta}(\cD)$ is defined by
\begin{equation} 
\|f\|_{K^m_{p,\theta,\Theta}(\cD)}
=\sum_{\vert \alpha\vert \leq m} \left(\int_{\cD} \vert \rho^{\vert \alpha\vert }D^{\alpha}f\vert ^p  \rho_{\circ}^{\theta-\Theta}\rho^{\Theta-d}\,dx\right)^{1/p}. \label{space K norm}
\end{equation}
The space $K^m_{p,\theta,\Theta}(\cD)$ is related to the weighted Sobolev space  $H^m_{p,\Theta}(\cD)$ introduced in \cite{KK2004, Krylov 1999-1,Lo1}  as follows:
$$
H^m_{p,\Theta}(\cD)=K^m_{p,\Theta,\Theta}(\cD),
$$
whose norm is given by
\begin{equation}
   \label{eqn 8.9.5}
   \|f\|_{H^m_{p,\Theta}(\cD)}:=\sum_{\vert \alpha\vert \leq m} \left(\int_{\cD} \vert \rho^{\vert \alpha\vert }D^{\alpha}f\vert ^p \rho^{\Theta-d}\,dx\right)^{1/p}, \quad m\in \bN_0.
      \end{equation}
      Note that the weight of $H^m_{p,\Theta}(\cD)$ is based only on the distance to the boundary. 
Using the fact that for any $\mu\in \bR$ and multi-index $\alpha$
\begin{equation}
\label{eqn 8.28.4}
\sup_{x\in \cD} \rho^{\vert \alpha\vert -\mu}_{\circ} \vert D^{\alpha} \rho^{\mu}_{\circ}(x)\vert \leq C(\mu,\alpha)<\infty,
\end{equation}
one can easily check 
$$f\in K^m_{p,\theta,\Theta}(\cD) \quad \text{ if and only if} \quad  \rho^{(\theta-\Theta)/p}_{\circ}f\in H^m_{p,\Theta}(\cD),
$$ and the norms in their corresponding spaces are equivalents, that is, 
\begin{equation}
    \label{eqn 8.9.7}
    \|f\|_{K^m_{p,\theta,\Theta}(\cD)}\sim \|\rho^{(\theta-\Theta)/p}_{\circ}f\|_{H^m_{p,\Theta}(\cD)}, \quad n\in \bN_0.
\end{equation}

Below we use 
relation \eqref{eqn 8.9.7} to define $K^{\gamma}_{p,\theta,\Theta}(\cD)$ for all $\gamma \in \bR$. 
Let   $\psi=\psi_{\cD}$  be a smooth function in $\cD$ (see e.g. \cite[Lemma 4.13]{Ku})
such that for any $m\in \bN_0$,
\begin{equation}
\label{eqn 8.9.1}
\psi_{\cD}(x)\sim  \rho_{\cD}(x),\quad \rho^{m}_{\cD}\vert D^{m+1}\psi_{\cD}\vert \leq
N(m)<\infty.
\end{equation}
Actually, such $\psi$ exists on any domains. Indeed, let $\cO$ be an arbitrary domain, and put $\rho_{\cO}(x)=d(x,\partial \cO)$, and 
\begin{equation}
\label{eqn 8.25.1}
\cO_{n,k}:=\{x\in \cO: e^{-n-k}<\rho_{\cO} (x)<e^{-n+k}\}. 
\end{equation} 
Then mollifying $1_{\cO_{n,2}}$ one can easily construct $\xi_n$ such that  
$$
\xi_n \in \cC^{\infty}_c(\cO_{n,3}), \quad \vert D^m \xi_n\vert \leq C(m)e^{mn}, \quad \sum_{n\in \bZ} \xi_n(x) \sim 1,
$$
and then one  can take 
\begin{equation}\label{eqn 8.25.2}
\psi=\psi_{\cO}=\sum_{n\in \bZ} e^{-n}\xi_n(x).
\end{equation}
 It is easy to check that $\psi=\psi_{\cO}$ satisfies \eqref{eqn 8.9.1} with $\rho_{\cO}$  in place of 
 $\rho_{\cD}$.
 
Next we choose a nonnegative function $\zeta\in \cC^{\infty}_{c}(\bR_{+})$ such that  $\zeta>0$ on $[e^{-1},e]$. Then, by the periodicity,
\begin{equation}
                                                       \label{11.4.1}
\sum_{n=-\infty}^{\infty}\zeta(e^{n+t})>c>0,\quad\forall\; t\in\bR.
\end{equation}
 For  $p\in(1,\infty)$ and $\gamma\in \bR$,  by $H^{\gamma}_p=H^{\gamma}_p(\bR^d)$ we  denote the space of Bessel potential with the norm  
$$
\|u\|_{H^{\gamma}_p}:=\|(1-\Delta)^{\gamma/2}u\|_{L_p(\bR^d)}:=\|\cF^{-1}[(1+\vert \xi\vert ^2)^{\gamma/2} \cF(u)(\xi)]\|_{L_p(\bR^d)}.
$$
In case $\gamma\in \bN_0$, $H^{\gamma}_p(\bR^d)$ coincides with $W^{\gamma}_p(\bR^d)$. The spaces of Bessel potentials enjoy the property
$$
\|u\|_{H^{\gamma_1}_p}\le \|u\|_{H^{\gamma_2}_p},\quad  \gamma_1\le \gamma_2.
$$
Especially, we have $\|u\|_{L_p}\le \|u\|_{H^{\gamma}_p}$ for any $\gamma\ge 0$.
For  $\ell_2$-valued functions $g$  we also define
$$
\|g\|_{H^{\gamma}_p(\ell_2)}:=\|\vert (1-\Delta)^{\gamma/2}g\vert _{\ell_2}\|_{L_p(\bR^d)}.
$$
Moreover, for $\bR^d$-valued functions $\tbf=(f^1,\ldots,f^d)$  we define
$$
\|\tbf\|_{H^{\gamma}_p(d)}:=\|\,\vert (1-\Delta)^{\gamma/2}\tbf\vert \,\|_{L_p(\bR^d)}.
$$

 From now on, if a function  defined on  a domain $\cO$ vanishes near the boundary of $\cO$, then by a trivial extension we consider it as a function defined on $\bR^d$. In particular, for any $k\in \bZ$ and a function $f$ on $\cO$, the function $\zeta(e^{-k}\psi_{\cO}(x))f(x)$ has a compact support in $\cO$ and  can be considered as a function on $\bR^d$.
 
\begin{defn}
\label{defn 8.28}
Let $p\in(1,\infty), \Theta, \gamma\in \bR$, and $\cO$ be a domain in $\bR^d$.  By $H^{\gamma}_{p,\theta}(\cO)$ we denote the class of all distributions $f$ on $\cO$ such that 
\begin{equation}
   \label{eqn 8.10.14}
\|f\|^p_{H^{\gamma}_{p,\Theta}(\cO)}:= \sum_{n\in \bZ} e^{n\Theta} \|\zeta(e^{-n}\psi(e^n\cdot))f(e^{n}\cdot)\|^p_{H^{\gamma}_p(\bR^d)}<\infty,
\end{equation}
where $\psi=\psi_{\cO}$ is taken from \eqref{eqn 8.25.2}.  Similarly, $H^{\gamma}_{p,\theta}(\cO;\ell_2)$ is the set of  $\ell_2$-valued functions $g$ such that
\begin{equation*}
\|g\|^p_{H^{\gamma}_{p,\Theta}(\cO;\ell_2)}:= \sum_{n\in \bZ} e^{n\Theta} \|\zeta(e^{-n}\psi(e^n\cdot))g(e^{n}\cdot)\|^p_{H^{\gamma}_p(\bR^d;\ell_2)}<\infty.
\end{equation*}
\end{defn}

It turns out  (see \cite[Proposition 2.2]{Lo1} or \cite[Lemma 4.3]{ConicPDE}) that the new norm in \eqref{eqn 8.10.14} is equivalent to the norm  in  \eqref{eqn 8.9.5} if $\gamma\in \bN_0$. In other words, 
 for $\gamma \in \bN_0$,
\begin{equation}
\label{eqn 8.9.8}
 \sum_{n\in \bZ} e^{n\Theta} \|\zeta(e^{-n}\psi_{\cO} (e^n\cdot))f(e^{n}\cdot)\|^p_{H^{\gamma}_p} \quad \sim \quad  \sum_{\vert \alpha\vert \leq \gamma} \int_{\cO} \vert \rho^{\vert \alpha\vert }D^{\alpha}f\vert ^p \rho^{\Theta-d}\,dx,
\end{equation}
and the equivalence relation depends only on $p,\gamma, \Theta,d, n,\zeta, \psi$ and $\cO$.  

Now we use equivalence relations \eqref{eqn 8.9.7} and \eqref{eqn 8.9.8},  and  define $K^{\gamma}_{p,\theta,\Theta}(\cD)$ for any chosen $\gamma\in \bR$.

\begin{defn}
\label{defn 8.19}
Let $p\in (1,\infty), \theta, \Theta, \gamma \in \bR$, and $\cD$ be a conic domain in $\bR^d$. 
 We write $f\in K^{\gamma}_{p,\theta,\Theta}(\cD)$ if and only if $\rho^{(\theta-\Theta)/p}_{\circ} f\in H^{\gamma}_{p,\Theta}(\cD)$, and define
\begin{equation}
             \label{eqn 8.10.1}
\|f\|_{ K^{\gamma}_{p,\theta,\Theta}(\cD)} := \|\rho^{(\theta-\Theta)/p}_{\circ} f\|_{H^{\gamma}_{p,\Theta}(\cD)}.
\end{equation}
The space $K^{\gamma}_{p,\theta,\Theta}(\cD;\ell_2)$ and its norm are defined similarly.  Also we write 
$\tbf=(f^1,f^2,\cdots,f^d)\in K^{\gamma}_{p,\theta,\Theta}(\cD; \bR^d)$ if 
$$
\|\tbf\|_{K^{\gamma}_{p,\theta,\Theta}(\cD; \bR^d)}:=\sum_{i=1}^d \|f^i\|_{K^{\gamma}_{p,\theta,\Theta}(\cD)}<\infty.
$$
\end{defn}

Note that the new norm of the space $K^{\gamma}_{p,\theta,\Theta}(\cD)$ is equivalent to the previous one if $\gamma\in \bN_0$.
Below we collect some  basic properties of the space $K^{\gamma}_{p,\theta,\Theta}(\cD)$. 

\begin{lemma}\label{property1}
Let $p\in(1,\infty)$ and $\theta, \Theta, \gamma \in\bR$.

(i) For a domain $\cO$ and  $\eta\in \cC^{\infty}_c(\bR_+)$,
\begin{equation}
 \label{eqn 4.24.5}
  \sum_{n \in\bZ}
e^{n\Theta} \|\eta(e^{-n}\psi_{\cO} (e^n\cdot))f(e^{n}\cdot)\|^p_{H^{\gamma}_p} \leq C(p,\Theta,d,\gamma, \eta, \cO) \|f\|_{H^{\gamma}_{p,\Theta}(\cO)}^{p}.
\end{equation}
The reverse inequality  also holds if $\eta$ satisfies \eqref{11.4.1}. 
 Moreover, the same statements hold for $\ell_2$-valued functions.

(ii) $\cC_c^{\infty}(\cD)$ is dense in $K^{\gamma}_{p,\theta,\Theta}(\cD)$.

(iii) For any $\mu \in \bR$,
 \begin{equation}
   \label{eqn 8.19.81}
 \|\psi^{\mu}f\|_{K^{\gamma}_{p,\theta,\Theta}(\cD)}\sim \|f\|_{K^{\gamma}_{p,\theta+\mu p, \Theta+\mu p}(\cD)}, 
 \end{equation}
where $\psi$ satisfies \eqref{eqn 8.9.1}.  The same statement holds for $\ell_2$-valued functions.
  
(iv) (Pointwise multiplier) Let $\gamma\in \bR$, $n\in \bN_0$ with $\vert \gamma\vert \leq n$. If  $\vert a\vert ^{(0)}_n:=\sup_{\cD} \sum_{\vert \alpha\vert \leq \vert n\vert } \rho^{\vert \alpha\vert }\vert D^{\alpha}a\vert <\infty$, then
\begin{equation}
   \label{eqn 8.19.11}
\|af\|_{K^{\gamma}_{p,\theta,\Theta}(\cD)}\leq C(n,p,d)\vert a\vert ^{(0)}_n \|f\|_{K^{\gamma}_{p,\theta,\Theta}(\cD)}.
\end{equation}

(v) The  operator $D_i:K^{\gamma}_{p,\theta,\Theta}(\cD)\to K^{\gamma-1}_{p,\theta+p,\Theta+p}(\cD)$ is  bounded for any $i=1,\ldots,d$.  In genereal,  for any multi-index $\alpha$  we have
\begin{align}
                  \label{eqn 4.16.1}
\|D^{\alpha}f\|_{K^{\gamma-\vert \alpha\vert }_{p,\theta+\vert \alpha\vert p,\Theta+\vert \alpha\vert p}(\cD)}\leq C \|f\|_{K^{\gamma}_{p,\theta,\Theta}(\cD)}.
\end{align}
 The same statement holds for $\ell_2$-valued functions.

 (vi) (Sobolev-H\"older embedding) Let $\gamma-\frac{d}{p}\geq n+\delta$, where $n\in \bN_0$ and $\delta\in (0,1)$. 
 Then for any $f\in K^{\gamma}_{p,\theta-p,\Theta-p}(\cD)$, 
 \begin{eqnarray}
&&\sum_{k\leq n} \vert \rho^{k-1+\frac{\Theta}{p}} \rho^{(\theta-\Theta)/p}_{\circ} D^{k}f\vert _{\cC(\cD)}  \nonumber \\
&&\quad + [\rho^{n-1+\delta+\frac{\Theta}{p}} \rho^{(\theta-\Theta)/p}_{\circ} D^{n} f]_{\cC^{\delta}(\cD)} \leq C \|f\|_{K^{\gamma}_{p,\theta-p,\Theta-p}(\cD)},
\label{eqn 8.21.1}
  \end{eqnarray}
 where $C=C(d,\gamma,p,\theta,\Theta,\cM)$.

\end{lemma}
\begin{proof}
All the results  follow from  Definition \ref{defn 8.19} and properties of the weighted Sobolev space $H^{\gamma}_{p,\Theta}(\cO)$ (cf. \cite{Lo1,Krylov 1999-1, Krylov 2001,KK2004}). See e.g. \cite[Proposition 2.2]{Lo1}  for (i)-(iii) and see \cite[Theorem 3.1]{Lo1} for (iv).   

 To prove (v),  we put $\xi=\rho^{(\theta-\Theta)/p}_{\circ}$. Then, using $\xi Df=D(\xi f)-\xi (\xi^{-1}D \xi) f$ and  \eqref{eqn 8.10.1}, we get
$$
\|Df\|_{K^{\gamma-1}_{p,\theta+p,\Theta+p}(\cD)}\leq \|D(\xi f)\|_{H^{\gamma-1}_{p,\Theta+p}(\cD)}+\|(\xi^{-1}D \xi) f\|_{K^{\gamma-1}_{p,\theta+p,\Theta+p}(\cD)}.
$$
By \cite[Theorem 3.1]{Lo1}, 
$$\|D(\xi f)\|_{H^{\gamma-1}_{p,\Theta+p}(\cD)} \leq C  \|\xi f\|_{H^{\gamma}_{p,\Theta}(\cD)}=C\|f\|_{K^{\gamma}_{p,\theta,\Theta}(\cD)}.
$$
Using \eqref{eqn 8.28.4}, one can check $\vert \psi \xi^{-1}D\xi\vert ^{(0)}_m<\infty$ for any $m\in \bN$. Thus, by \eqref{eqn 8.19.81} and \eqref{eqn 8.19.11},  
\begin{eqnarray*}
\|(\xi^{-1}D \xi) f\|_{K^{\gamma-1}_{p,\theta+p,\Theta+p}(\cD)}
&\leq&C\| (\psi \xi^{-1}D \xi)  f\|_{K^{\gamma-1}_{p,\theta,\Theta}(\cD)}\leq C \|f\|_{K^{\gamma-1}_{p,\theta,\Theta}(\cD)}.
\end{eqnarray*}
Thus  (v) is proved.

  Finally  we prove (vi). Put $g=\xi f$. Then by \cite[Theorem 4.3]{Lo1}, 
\begin{equation}
  \label{eqn 8.28.8}
\sum_{k\leq n} \vert \rho^{k-1+\frac{\Theta}{p}}  D^{k}g\vert _{\cC(\cD)}  
+ [\rho^{n-1+\delta+\frac{\Theta}{p}} D^{n} g]_{\cC^{\delta}(\cD)} \leq C \|g\|_{H^{\gamma}_{p,\Theta-p}(\cD)}.
\end{equation}
Hence, to prove (vi), it is enough to note that  the left hand side of \eqref{eqn 8.21.1} is bounded by a constant times of the left hand side of \eqref{eqn 8.28.8}.  The lemma is proved.
\end{proof}

Using the aforementioned spaces,  we now  introduce the function spaces for the  solutions $u$ to equation \eqref{stochastic parabolic equation} as well as the function spaces for the inputs $f^0,\tbf$, and $g$.
To make equation \eqref{stochastic parabolic equation} well-defined after all, we restrict $p\in [2,\infty)$; see Remark \ref{sto series} ($i$) below. With such $p$ and a fixed time $T\in(0,\infty)$  we first define  
\begin{eqnarray*}\label{entire}
&&\bH^{\gamma}_{p}(T):=L_p(\Omega\times (0,T], \cP ; H^{\gamma}_p),\\ 
&&\bH^{\gamma}_{p}(T,\ell_2):=L_p(\Omega\times (0,T], \cP ; H^{\gamma}_p(\ell_2)).
\end{eqnarray*}
Next, for   $\theta, \Theta, \gamma \in\bR$  we define the function spaces
\begin{eqnarray*}
&&\bK^{\gamma}_{p,\theta,\Theta}(\cD,T)\,:=L_p(\Omega\times (0,T], \cP;K^{\gamma}_{p,\theta,\Theta}(\cD)),\\
&&\bK^{\gamma}_{p,\theta,\Theta}(\cD,T,d)\,:=L_p(\Omega\times (0,T], \cP;K^{\gamma}_{p,\theta,\Theta}(\cD; \bR^d)),\\
&&\bK^{\gamma}_{p,\theta,\Theta}(\cD,T, \ell_2)\,:=L_p(\Omega \times (0,T], \cP;K^{\gamma}_{p,\theta,\Theta}(\cD;\ell_2)),
\end{eqnarray*}
and denote
$$
\bL_{p,\theta,\Theta}(\cD,T):=\bK^0_{p,\theta,\Theta}(\cD,T),\quad \bL_{p,\theta,\Theta}(\cD,T,d):=\bK^0_{p,\theta,\Theta}(\cD,T,d),
$$
$$ \bL_{p,\theta,\Theta}(\cD,T, \ell_2):=\bK^0_{p,\theta,\Theta}(\cD,T, \ell_2).
$$
Also, by  $\bK^{\infty}_c(\cD,T)$  we denote  the space of all functions $f$ of the form
\begin{equation*}
f(\omega,t,x)=\sum^m_{i=1}{\bf{1}}_{(\tau_{i-1}(\omega),\tau_i(\omega)]}(t)f_i(x),
\end{equation*}
where $\tau_0\le \cdots\le \tau_m$ is a finite sequence of bounded  stopping times with respect to the filtration $(\rF_t)_{t\geq 0}$, and $f_i\in \cC^{\infty}_c(\cD)$, $i=1,\ldots,m$. Similarly, we define $\bK^{\infty}_c(\cD,T, \ell_2)$ as the space of $\ell_2$-valued functions $g=(g^1,g^2,
\ldots)$ such that  the first finite number of $g^k$ are in $\bK^{\infty}_c(\cD,T)$ and the rest are all identically zero. We also define $\bK^{\infty}_c(\cD,T, d)$ for  $\bR^d$-valued functions $\tbf=(f^1,\ldots,f^d)$ in the same manner. 
Moreover, by $\bK^{\infty}_c(\cD)$ we denote  the space of all functions $f$ of the form
\begin{equation*}
f(\omega,x)=\sum^m_{i=1}{\bf{1}}_{A_i}(\omega)f_i(x),
\end{equation*}
where $A_i\in\rF_0$ and $f_i\in \cC^{\infty}_c(\cD)$, $i=1,\ldots,m$. 
 
\begin{remark}\label{dense space}
For any $\theta, \Theta, \gamma \in \bR$,   $\bK^{\infty}_c(\cD,T)$ is dense in $\bK^{\gamma}_{p,\theta,\Theta}(\cD,T)$  and so is $\bK^{\infty}_c(\cD,T,\ell_2)$ in $\bK^{\gamma}_{p,\theta,\Theta}(\cD,T,\ell_2)$. Indeed, by the definition of $\cP$, any function $f\in \bK^{\gamma}_{p,\theta,\Theta}(\cD,T)$ can be approximated by functions of the type 
$$
\sum_{i=1}^m 1_{(\tau_i(\omega), \tau_{i+1}(\omega)]}(t) h_i(x),
$$
where $\tau_m$ are bounded stopping times and $h_i\in K^{\gamma}_{p,\theta,\Theta}(\cD)$, $i=1,\ldots,m$.  Thus the claim follows from Lemma \ref{property1} (ii).  Similarly, $\bK^{\infty}_c(\cD)$ is dense in $L_p(\Omega;K^{\gamma}_{p,\theta,\Theta}(\cD)):=L_p(\Omega, \rF_0,\bP;K^{\gamma}_{p,\theta,\Theta}(\cD))$.
\end{remark}

From now on we will also use the notation 
$$U^{\gamma+2}_{p,\theta,\Theta}(\cD):=K^{\gamma+2-2/p}_{p,\theta+2-p,\Theta+2-p}(\cD).
$$
The following definition frames the spaces for the solutions of our SPDE. 
\begin{defn}\label{first spaces}

  Let $p\in[2,\infty)$ and  $\theta, \Theta,\gamma \in\bR$.   We write $u\in\cK^{\gamma+2}_{p,\theta,\Theta}(\cD,T)$ if
 $u \in \bK^{\gamma+2}_{p,\theta-p,\Theta-p}(\mathcal{D},T)$, $u(0,\cdot)\in \bU^{\gamma+2}_{p,\theta,\Theta}(\cD):=L_p(\Omega,\rF_0,\bP;U^{\gamma+2}_{p,\theta,\Theta}(\cD))$, and there exists $(\tilde{f}, \tilde{g}) \in\bK^{\gamma}_{p,\theta+p,\Theta+p}(\mathcal{D},T)\times \bK^{\gamma+1}_{p,\theta,\Theta}(\cD,T, \ell_2)$ such that  
\begin{align}
du=\tilde{f}\,dt+\sum_k \tilde{g}^kdw^k_t,\quad t\in(0,T] \nonumber
\end{align}
  in the sense of distributions on $\cD$, that is,  for any $\varphi\in \cC_c^{\infty}(\cD)$ the equality
\begin{align}
  \label{eqn sol}
(u(t,\cdot),\varphi)=(u(0,\cdot),\varphi)+\int^{t}_{0}(\tilde{f}(s,\cdot),\varphi)ds+\sum_{k=1}^{\infty}\int^t_0(\tilde{g}^k(s,\cdot),\varphi)dw^k_s
\end{align}
holds for all $t\in (0,T]$ (a.s.).
In this case we  write
\begin{align*}
\bD u:=\tilde{f}\quad\text{and}\quad \bS u:=\tilde{g}.
\end{align*}
The norm in $\cK^{\gamma+2}_{p,\theta,\Theta}(\cD,T)$  is given  by
\begin{align*}
\|u\|_{\cK^{\gamma+2}_{p,\theta,\Theta}(\cD,T)}&=\|u\|_{\bK^{\gamma+2}_{p,\theta-p,\Theta-p}(\cD,T)}+\|\bD u\|_{\bK^{\gamma}_{p,\theta+p,\Theta+p}(\cD,T)}+\|\bS u\|_{\bK^{\gamma+1}_{p,\theta,\Theta}(\cD,T,\ell_2)}\\
&\quad +\|u(0,\cdot)\|_{\bU^{\gamma+2}_{p,\theta,\Theta}(\cD)}.
\end{align*}
\end{defn}

\vspace{0.1cm}

\begin{remark}
   \label{sol}
 Let us go back to our main equation \eqref{stochastic parabolic equation}.  Let  $f^0\in \bK^{\gamma}_{p,\theta+p,\Theta+p}(\cD,T)$, 
   $\tbf=(f^1,\cdots,f^d)\in \bK^{\gamma+1}_{p,\theta,\Theta}(\cD,T, d)$, $g\in \bK^{\gamma+1}_{p,\theta,\Theta}(\cD,T, \ell_2)$, $u(0,\cdot)\in \bU^{\gamma+2}_{p,\theta,\Theta}(\cD)$, and  $u$ belong to $ \bK^{\gamma+2}_{p,\theta-p,\Theta-p}(\cD,T)$ and be a solution to   equation \eqref{stochastic parabolic equation}, that is, $u$ satisfies  
   $$
d u =\left( \cL u+f^0+\sum_{i=1}^d f^i_{x^i}\right)dt +\sum^{\infty}_{k=1} g^kdw_t^k,\quad   t\in(0,T]
$$
 in the sense of distributions on $\cD$. Then by \eqref{eqn 4.16.1}  in Lemma \ref{property1} ($v$), we have
 $$
 \cL u+f^0+\sum_{i=1}^d f^i_{x^i}\in \bK^{\gamma}_{p,\theta+p,\Theta+p}(\cD,T)
 $$
  and consequently $u$ belongs to $\cK^{\gamma+2}_{p,\theta,\Theta}(\cD,T)$  with the accompanied inequality
  \begin{eqnarray}
 && \|u\|_{\cK^{\gamma+2}_{p,\theta,\Theta}(\cD,T)}\nonumber\\ &\leq& C \Big(\|u\|_{\bK^{\gamma+2}_{p,\theta-p,\Theta-p}(\cD,T)}+ \|f^0\|_{\bK^{\gamma}_{p,\theta+p,\Theta+p}(\cD,T)}+ \sum_{i=1}^d \|f^i\|_{\bK^{\gamma+1}_{p,\theta,\Theta}(\cD,T)} \nonumber \\
 &&\quad \quad +\|g\|_{\bK^{\gamma+1}_{p,\theta,\Theta}(\cD,T,\ell_2)}+\|u(0,\cdot)\|_{\bU^{\gamma+2}_{p,\theta,\Theta}(\cD)}\Big).  \label{eqiv norm}
  \end{eqnarray}
   \end{remark}

\begin{remark}
   \label{sto series}
   (i)    Note that for any $m,n\in \bN$ with $m>n$, the quadratic variation of the continuous martingale $\sum_{k=n}^m \int^t_0(\tilde{g}^k, \varphi) dw^k_s$ is $\sum_{k=n}^m \int^t_0 (\tilde{g}^k(s), \varphi)^2ds$.   Following the lines in \cite[Remark 3.2]{Krylov 1999-4} and using the condition $p\geq 2$, one can easily check
$$
   \bE\sum_{k=1}^{\infty} \int^T_0 (\tilde{g}^k(t),\varphi)^2 dt 
    \leq N(\varphi,p,T) \|\tilde{g}\|^p_{\bL_{p,\theta,\Theta}(\cD,T,\ell_2)},
$$
which implies  the infinite series $\sum_{k=1}^{\infty}\int^t_0 (\tilde{g}^k(s),\varphi)dw^k_s$ converges in $L_2\big(\Omega;\cC([0,T])\big)$ and in probability uniformly in $t\in [0,T]$. As a consequence, $(u(t,\cdot),\varphi)$ in \eqref{eqn sol} is a continuous semi-martingale on $[0,T]$.

(ii) In Definition~\ref{first spaces}, $\bD u$ and $\bS u$ are uniquely determined. This can be seen by using the same arguments in \cite[Remark 3.3]{Krylov 1999-4}.
\end{remark}


\begin{thm}
\label{banach}
For any $p\in [2,\infty)$ and  $\theta,\Theta, \gamma\in \bR$,  $\cK^{\gamma+2}_{p,\theta,\Theta}(\cD,T)$ is a   Banach space.
\end{thm}
\begin{proof}
We only need to prove the completeness.  This can be proved by repeating argument in Remark 3.8 of \cite{Krylov 2001}, which treats the case $\theta=\Theta$ and 
$\cD=\bR^d_+$. The argument in this proof is quite universal and,    without any changes,  works on any conic domain $\cD$ with any $\theta,\Theta\in \bR$. 
\end{proof}

The following theorem addresses important temporal properties of the functions in $\cK^{\gamma+2}_{p,\theta,\Theta}(\cD,T)$. See Section \ref{sec:Introduction} for the notations $[\cdot]_{\cC^{\alpha}}$ and $\vert \cdot\vert _{\cC^{\alpha}}$.

\begin{thm}
\label{embedding}  
Let $p\in[2,\infty)$ and $\theta,\Theta, \gamma\in \bR$. 

(i) If $2/p<\alpha<\beta \leq 1$, then for any $u\in  \cK^{\gamma+2}_{p,\theta,\Theta}(\cD,T)$,
\begin{eqnarray}
&&\bE [\psi^{\beta-1}u]^p_{\cC^{\alpha/2-1/p}\left([0,T]; K^{\gamma+2-\beta}_{p,\theta,\Theta}(\cD)\right)}\leq C\,T^{(\beta-\alpha)p/2}\|u\|^p_{\cK^{\gamma+2}_{p,\theta,\Theta}(\cD,T)},   \label{eqn 8.10.10}\label{Holder1}
\end{eqnarray}
and in addition, if $\psi^{\beta-1}u(0,\cdot)\in L_p(\Omega;K^{\gamma+2-\beta}_{p,\theta,\Theta}(\cD))$, 
\begin{eqnarray}
\bE  \vert \psi^{\beta-1}u\vert ^p_{\cC\left([0,T]; K^{\gamma+2-\beta}_{p,\theta,\Theta}(\cD)\right)}&\leq& C\bE\|\psi^{\beta-1}u(0,\cdot)\|^p_{K^{\gamma+2-\beta}_{p,\theta,\Theta}(\cD)}\nonumber\\
&& \quad+C T^{p\beta/2-1}\|u\|^p_{\cK^{\gamma+2}_{p,\theta,\Theta}(\cD,T)},\label{Holder2}
\end{eqnarray}
where $\psi$ satisfies \eqref{eqn 8.9.1} and constants $C$ are independent of $T$ and $u$.

(ii) For any $u\in  \cK^{\gamma+2}_{p,\theta,\Theta}(\cD,T)$ with $u(0,\cdot)= 0$, $u$ belongs to $ L_p(\Omega; \cC([0,T]; K^{\gamma}_{p,\theta,\Theta}(\cD))$ and
$$
\bE \sup_{t\leq T} \|u(t)\|^p_{K^{\gamma+1}_{p,\theta,\Theta}(\cD)}\leq C \|u\|^p_{\cK^{\gamma+2}_{p,\theta,\Theta}(\cD,T)},
$$
where $C=C(d,p,n,\theta,\Theta,\cD, T)$. In particular, for any $t\leq T$,
\begin{equation}
 \label{eqn 8.25.31}
\|u\|^p_{\bK^{\gamma+1}_{p,\theta,\Theta}(\cD,t)}\leq \int^t_0 \bE\sup_{r\leq s} \|u(r)\|^p_{K^{\gamma+1}_{p,\theta,\Theta}(\cD,r)} ds \leq 
C\int^t_0 \|u\|^p_{\cK^{\gamma+2}_{p,\theta,\Theta}(\cD,s)}ds.
\end{equation}
\end{thm}

\begin{proof}
We follow the argument in  \cite[Section 6]{Krylov 2001} (or  the proof of \cite[Theorem 2.8]{Kim2014}), using \cite[Corollary 4.12]{Krylov 2001}.

(i).   As usual, we suppress the argument $\omega$. Put $\xi(x)=|x|^{(\theta-\Theta)/p}$ and set  $v=\xi u$, $\bar{f}=\xi \bD u$, $\bar{g}=\xi \bS u$. Then we have
$$
dv=\bar{f}dt+\sum_{k=1}^{\infty} \bar{g}^k dw^k_t, \quad t\in(0,T]
$$ 
in the sense of distributions on $\cD$ with the initial condition $v(0,\cdot)=\xi u(0,\cdot)$.   By \eqref{eqn 8.19.81} and Definition \ref{defn 8.19}, we have
\begin{eqnarray}
I_1&:=&\bE\left[\psi^{\beta-1}u\right]^p_{\cC^{\alpha/2-1/p}([0,T],K^{\gamma+2-\beta}_{p,\theta, \Theta}(\cD))} \nonumber\\
&\sim&   \bE\left[v\right]^p_{\cC^{\alpha/2-1/p}([0,T],H^{\gamma+2-\beta}_{p,\Theta+p(\beta-1)}(\cD))}   \label{eqn 8.20.11}\\
&\leq& C\sum_n e^{n(\Theta+p(\beta-1))}\bE
\left[v(\cdot,e^n\cdot)\zeta(e^{-n}\psi(e^n\cdot))\right]^p_{\cC^{\alpha/2-1/p}([0,T];H^{\gamma+2-\beta}_{p})}.  \nonumber
\end{eqnarray}
Now, by assumption, the function $v_n(t,x):=v(t,e^nx)\zeta(e^{-n}\psi(e^nx))$ belongs to $\bH^{\gamma+2}_p(T)$ and  satisfies 
\begin{equation}
   \label{eqn 08.31.1}
dv_n=\bar{f}(t,e^nx)\zeta(e^{-n}\psi(e^nx))dt+ \sum_{k=1}^{\infty} \bar{g}^k(t,e^nx) \zeta(e^{-n}\psi(e^nx)) dw^k_t, \quad t>0
\end{equation}
 on the entire space $\bR^d$. Then, by \cite[Corollary 4.12]{Krylov 2001} and \eqref{eqn 08.31.1}, there exists a constant $N>0$,
independent of $T$ and $u$, so that for any constant $a>0$,
\begin{eqnarray*}
&&\bE\left[v(\cdot,e^n\cdot)\zeta(e^{-n}\psi(e^n\cdot))\right]^p_{\cC^{\alpha/2-1/p}([0,T];H^{\gamma+2-\beta}_{p})}\\
&\leq&
C\,T^{(\beta-\alpha)p/2}a^{\beta-1} \Big(a\|v(\cdot,e^n\cdot)\zeta(e^{-n}\psi(e^n\cdot))\|^p_{\bH^{\gamma+2}_{p}(T)}\\
&&\hspace{1cm}+\,a^{-1}\|\bar{f}(\cdot,e^n\cdot)\zeta(e^{-n}\psi(e^n\cdot))\|^p_{\bH^{\gamma}_{p}(T)}+\|\bar{g}^k(\cdot,e^n\cdot) \zeta(e^{-n}\psi(e^n\cdot)) \|^p_{\bH^{\gamma+1}_{p}(T,\ell_2)} \Big)
\end{eqnarray*}
holds. Taking $a=e^{-np}$, we note that
(\ref{eqn 8.20.11}) yields
\begin{eqnarray}
\nonumber
I_1 &\leq& C\,T^{(\beta-\alpha)p/2}\Big(\sum_n    
e^{n(\Theta-p)}\|v(\cdot,e^n\cdot)\zeta(e^{-n}\psi(e^n\cdot))\|^p_{\bH^{\gamma+2}_p(T)}\\  \nonumber
&&\quad\quad+ \sum_n e^{n(\Theta+p)}
\|\bar{f}(\cdot,e^n\cdot)\zeta(e^{-n}\psi(e^n\cdot))\|^p_{\bH^{\gamma}_{p}(T)}\\
&&\quad \quad+\sum_n e^{n\Theta}\|\bar{g}^k(\cdot,e^n\cdot) \zeta(e^{-n}\psi(e^n\cdot)) \|^p_{\bH^{\gamma+1}_p(T,\ell_2)} \Big) \nonumber\\
&=& C\,T^{(\beta-\alpha)p/2} \Big(\|u\|^p_{\bK^{\gamma+2}_{p,\theta-p, \Theta-p}(\cD,T)}
+\|\bD u\|^p_{\bK^{\gamma}_{p,\theta+p,\Theta+p}(\cD,T)} +\|\bS u\|^p_{\bK^{\gamma+1}_{p,\theta, \Theta}(\cD,T,\ell_2)}\Big) \nonumber\\
&\le& C\, T^{(\beta-\alpha)p/2}\|u\|^p_{\cK^{\gamma+2}_{p,\theta, \Theta}(\cD,T)}.  \nonumber
\end{eqnarray}
Thus \eqref{Holder1} is proved. 

 If $\psi^{\beta-1}u(0,\cdot)\in L_p(\Omega;K^{\gamma+2-\beta}_{p,\theta,\Theta}(\cD))$, then we note that  $\psi^{\beta-1}u$ belongs to $\cC\left([0,T]; K^{\gamma+2-\beta}_{p,\theta,\Theta}(\cD)\right)$ now.  For  estimate \eqref{Holder2}, we have
\begin{eqnarray}
I_2&:=&\bE  \vert \psi^{\beta-1}u\vert ^p_{\cC\big([0,T]; K^{\gamma+2-\beta}_{p,\theta,\Theta}(\cD)\big)}\nonumber\\
&\leq& C\,\sum_n e^{n(\Theta+p(\beta-1))}\bE
\vert v(\cdot,e^n\cdot)\zeta(e^{-n}\psi(e^n\cdot))\vert ^p_{\cC([0,T];H^{\gamma+2-\beta}_{p})}    \label{Holder2 proof}
\end{eqnarray}
and by \cite[Corollary 4.12]{Krylov 2001} again, for any constant $a>0$, 
\begin{eqnarray*}
&&\bE\vert v(\cdot,e^n\cdot)\zeta(e^{-n}\psi(e^n\cdot))\vert ^p_{\cC\big([0,T];H^{\gamma+2-\beta}_{p}\big)}\\
&\leq&
C\,\bE\|v(0,e^n\cdot)\zeta(e^{-n}\psi(e^n\cdot))\|^p_{H^{\gamma+2-\beta}_{p}}\\ 
&&+C\,T^{p\beta/2-1}a^{\beta-1} \Big(a\|v(\cdot,e^n\cdot)\zeta(e^{-n}\psi(e^n\cdot))\|^p_{\bH^{\gamma+2}_{p}(T)}\\
&&\hspace{0.7cm}+a^{-1}\|\bar{f}(\cdot,e^n\cdot)\zeta(e^{-n}\psi(e^n\cdot))\|^p_{\bH^{\gamma}_{p}(T)}+\|\bar{g}^k(\cdot,e^n\cdot) \zeta(e^{-n}\psi(e^n\cdot)) \|^p_{\bH^{\gamma+1}_{p}(T,\ell_2)} \Big).
\end{eqnarray*}
This, \eqref{Holder2 proof}, and the same argument above, especially the adjustment $a=e^{-np}$ for each $n$, lead us to \eqref{Holder2}.

(ii).  We use the notations used in (i). Obviously,
$$
\bE\sup_{t\leq
T}\|u(t)\|^p_{K^{\gamma+1}_{p,\theta,\Theta}(\cD)}\leq C\,
\sum_ne^{n\Theta} \bE \sup_{t\leq T}
\|v(t,e^n\cdot)\zeta(e^{-n}\psi(e^n\cdot))\|^p_{H^{\gamma+1}_{p}}.
$$
 By Remark 4.14 in \cite{Krylov 2001} with $\beta=1$ there,  $v_n \in L_p(\Omega; \cC([0,T]; H^{\gamma+1}_p))$ and for any $a>0$,
\begin{align*}
\bE \sup_{t\leq T}
\|v(t,e^n\cdot)\zeta(e^{-n}\psi(e^n\cdot))\|^p_{H^{\gamma+1}_{p}}\leq
C\, \Big(a\|v(\cdot,e^n\cdot)\zeta(e^{-n}\psi(e^n\cdot))\|^p_{\bH^{\gamma+2}_{p}(T)}&\\
+a^{-1}\|\bar{f}(\cdot,e^n\cdot)\zeta(e^{-n}\psi(e^n\cdot))\|^p_{\bH^{\gamma}_{p}(T)}
+\|\bar{g}^k(\cdot,e^n\cdot) \zeta(e^{-n}\psi(e^n\cdot)) \|^p_{\bH^{\gamma+1}_{p}(T,\ell_2)}&\Big).
\end{align*}
Again, taking $a=e^{-np}$ and following the above arguments,  we get
\begin{eqnarray*}
&&\bE\sup_{t\leq T}\|u(t)\|^p_{K^{\gamma+1}_{p,\theta, \Theta}(\cD)} \\
&&\leq
C\, \Big(\|u\|^p_{\bK^{\gamma+2}_{p,\theta-p, \Theta-p}(\cD,T)}
+\|\bD u\|^p_{\bK^{\gamma}_{p,\theta+p, \Theta+p}(\cD,T)}+\|\bS u\|^p_{\bK^{\gamma+1}_{p,\theta,\Theta}(\cD,T,\ell_2)}\Big)\\
&&= C\,\|u\|^p_{\cK^{\gamma+2}_{p,\theta,\Theta}(\cD,T)}.
\end{eqnarray*}
The theorem is proved.
\end{proof}

\begin{remark}\label{additional condition initial}
The additional condition  $\psi^{\beta-1}u(0,\cdot)\in L_p(\Omega;K^{\gamma+2-\beta}_{p,\theta,\Theta}(\cD))$ for \eqref{Holder2} does not follow from   the assumption $u\in\cK^{\gamma+2}_{p,\theta,\Theta}(\cD,T)$.
This condition is unnecessary when we prove the corresponding result on polygonal domains. See Remark \ref{remark 8.29} for detail.
\end{remark}

\begin{remark}
 Theorems \ref{banach} and  \ref{embedding}  hold for any $\theta, \Theta\in \bR$, but certain restrictions will be given later for our main results, Theorems \ref{main result} and \ref{main result-random}.
Actually  the  admissible range of $\theta$ for our Sobolev-regularity theory of equation \eqref{stochastic parabolic equation} is affected by \emph{the shape of} $\cD=\cD(\cM)$, the uniform parabolicity of the leading coefficients, the space dimension $d$, and the summability parameter $p$. On the the hand, the admissible range of  $\Theta$ depends only on $d$ and  $p$,  that is, 
\begin{equation*}
d-1<\Theta<d-1+p.
\end{equation*}
\end{remark}

To explain the admissible range of $\theta$ for equation \eqref{stochastic parabolic equation} we need the following definitions. For some of the notations in them one can refer to   Section \ref{sec:Introduction}.

\begin{defn}[cf. Section 2 of \cite{Kozlov Nazarov 2014}] 
\label{lambda}

Let $L=\sum_{i,j=1}^d \alpha^{ij}(t)D_{ij}$ be a uniformly parabolic ``deterministic"  operator with bounded coefficients $\alpha^{ij}$s.

(i)  By  $\lambda^+_{c,L}=\lambda^+_{c,L,\cD}$  we denote  the supremum of all $\lambda\geq 0$ such that  for some constant $K_0=K_0(\lambda, L,\cM)$ it holds that
\begin{equation} \label{eqn 8.17.10}
\vert v(t,x)\vert\le K_0 \left(\frac{\vert x\vert}{R}\right)^{\lambda}\sup_{Q^{\mathcal{D}}_{\frac{3R}{4}}(t_0,0)}\ \vert v\vert,
\quad
\forall \;(t,x)\in Q^{\mathcal{D}}_{R/2}(t_0,0)
\end{equation}
 for any  $R>0$, $t_0$, and  the deterministic function $v=v(t,x)$ belonging to  $\mathcal{V}_{loc}(Q^{\mathcal{D}}_R(t_0,0))$  and satisfying
\begin{equation}
\label{eqn 8.17.14}
v_t=L v  \quad \text{in}\; Q^{\mathcal{D}}_R(t_0,0)\quad ; \;\quad
v(t,x)=0\quad\text{for}\;\; x\in\partial\mathcal{D}.
\end{equation}

(ii) By   $\lambda^-_{c,L}$  we denote   the supremum of  $\lambda \geq 0$ with above property  for the operator 
\begin{equation*}
\hat{L}:=\sum_{i,j}\alpha^{ij}(-t)D_{ij}.
\end{equation*}

\end{defn}

Note that $K_0$ in \eqref{eqn 8.17.10} may depend on the operator $L$.  Such dependency on $L$ is one of major obstacles when one handles SPDE having random coefficients,  since  it naturally involves infinitely many operators at the same time. To treat such case, which is in fact our case in this article, we design the following definition.

\begin{defn}
   \label{lambda2}
(i) By $\cT_{\nu_1,\nu_2}$ we denote the collection of all ``deterministic" operators in the form $L=\sum_{i,j=1}^d \alpha^{ij}(t)D_{ij}$, where $\alpha^{ij}(t)$ are measurable in $t$ and satisfy Assumption \ref{ass coeff} with the fixed constants $\nu_1,\nu_2$ in the uniform parabolicity condition \eqref{uniform parabolicity}. 

(ii) For a fixed $\cD=\cD(\cM)$, by $\lambda_c(\nu_1,\nu_2)=\lambda_c(\nu_1,\nu_2,\cD)$ we denote the supremum of all $\lambda\geq 0$ such that for some constant $K_0=K_0(\lambda, \nu_1,\nu_2,\cM)$ it holds that for any operator $L \in \cT_{\nu_1,\nu_2}$, $R>0$ and $t_0$, 
\begin{equation}
    \label{eqn 8.17.11}
\vert v(t,x)\vert\le K_0 \left(\frac{\vert x\vert}{R}\right)^{\lambda}\sup_{Q^{\mathcal{D}}_{\frac{3R}{4}}(t_0,0)}\ \vert v\vert , \quad
\forall \;(t,x)\in Q^{\mathcal{D}}_{R/2}(t_0,0),
\end{equation}
provided that $v$ is a deterministic function in  $\mathcal{V}_{loc}(Q^{\mathcal{D}}_R(t_0,0))$ satisfying
\begin{equation*}
v_t=L v  \quad \text{in}\; Q^{\mathcal{D}}_R(t_0,0)\quad ; \;\quad
v(t,x)=0\quad\text{for}\;\; x\in\partial\mathcal{D}.
\end{equation*}

\end{defn}

\begin{remark}
(i) Note that the dependency of $K_0$ in Definition \ref{lambda2} is more explicit compared to that of Definition \ref{lambda}. By definitions, if $L$ is an operator in $\cT_{\nu_1,\nu_2}$,  then 
$$\lambda^{\pm}_{c,L}\geq \lambda_c(\nu_1,\nu_2).
$$

(ii) The values of $\lambda^{\pm}_{c,L}$ and $\lambda_c(\nu_1,\nu_2)$ do not change if one replaces $\frac{3}{4}$ in \eqref{eqn 8.17.10} and \eqref{eqn 8.17.11} by any number in $(1/2,1)$ (see \cite[Lemma 2.2]{Kozlov Nazarov 2014}).  Following the proof of  \cite[Lemma 2.2]{Kozlov Nazarov 2014}, one can also show that for any constant $\beta>0$ 
$$
 \lambda^{\pm}_{c,\beta L}= \lambda^{\pm}_{c,L}, \qquad  \lambda_c(\beta \nu_1,\beta \nu_2)= \lambda_c(\nu_1,\nu_2).
 $$
 \end{remark}

Below are some sharp estimates for $\lambda^{\pm}_{c,L}$ and $\lambda_c(\nu_1,\nu_2)$. See \cite{Kozlov Nazarov 2014} for more informations.

\begin{prop}\label{critical exponents}
(i)  If $L=\Delta_x$,  then
\begin{equation*}
\lambda^{\pm}_{c,L}=-\frac{d-2}{2}+\sqrt{\Lambda+\frac{(d-2)^2}{4}} \, >0,
\end{equation*}
where $\Lambda=\Lambda_{\cD}$ is the first eigenvalue of Laplace-Beltrami operator with the Dirichlet condition on  $\mathcal{M}$.
In particular, if $d=2$ and $\cD=\cD^{(\kappa)}$ (see \eqref{wedge in 2d}), then
\begin{equation*}
\lambda^{\pm}_{c,L}=\frac{\pi}{\kappa}.
\end{equation*}

(ii)  Let $0<\nu_1\leq \nu_2<\infty$.     Then  we have $\lambda_{c}(\nu_1,\nu_2)>0$ and 
\begin{equation}\label{CUB3}
\lambda_{c}(\nu_1,\nu_2) \geq  -\frac{d}{2}+\sqrt{\frac{\nu_1}{\nu_2}}\sqrt{\Lambda+\frac{(d-2)^2}{4}}.
\end{equation}
\end{prop}

\begin{proof}
(i) follows from  \cite[Theorem 2.4.3]{Kozlov Nazarov 2014}. (ii) also follows from 
 the proofs of  \cite[Theorem 2.4.1, Theorem 2.4.7]{Kozlov Nazarov 2014}, which only consider the case $\nu_2=1/{\nu_1}$.  Inspecting the proofs of  \cite[Theorem 2.4.1, Theorem 2.4.7]{Kozlov Nazarov 2014} one can easily check 
 $$
\lambda^{\pm}_{c,L}\geq -\frac{d}{2}+\sqrt{\frac{\nu_1}{\nu_2}}\sqrt{\Lambda+\frac{(d-2)^2}{4}}\\,\,\,\text{and}\,\,\,\lambda^{\pm}_{c,L}>c>0\quad\text{if}\quad L\in \cT_{\nu_1,\nu_2},
$$
where the constant $c$ is  the H\"older exponent of solutions to equation \eqref{eqn 8.17.14}, and it can be chosen so that it depends only on $\nu_1,\nu_2$ and $\cM$.
Moreover, for $\lambda>0$ satisfying
$$
\lambda<c\vee \Big(-\frac{d}{2}+\sqrt{\frac{\nu_1}{\nu_2}}\sqrt{\Lambda+\frac{(d-2)^2}{4}}\Big)
$$
the constant $K_0$ in \eqref{eqn 8.17.11} can  be chosen so that it  depends only on $\nu_1,\nu_2$ and $\cM$.
This proves \eqref{CUB3}.

\end{proof}

\begin{example}[$d=2$]\label{example proposition}

For   $\kappa\in (0,2\pi)$ and $\alpha\in [0,2\pi)$,  we consider
$$
\cD=\mathcal{D}_{\kappa,\alpha}:=\left\{x=(r\cos\theta,\ r\sin\theta)\in\bR^2 \,\vert\, r\in(0,\ \infty),\ -\frac{\kappa}{2}+\alpha<\theta<\frac{\kappa}{2}+\alpha\right\}
$$
and  the constant operator 
$$
L=aD_{x_1x_1}+b(D_{x_1x_2}+D_{x_2x_1})+cD_{x_2x_2},
$$
where $a,b,c$ are constants such that $a+c>0$ and $ac-b^2>0$.   Then, by  \cite[Proposition 4.1]{Green}, we have
\begin{align*}
\lambda^{\pm}_{c,L}=\lambda^{\pm}_{c,L,\cD_{\kappa,\alpha}}=\frac{\,\pi\,}{\widetilde{\kappa}},   
\end{align*}
where
$$
\widetilde{\kappa}=\pi-\arctan\Big(\,\frac{\bar{c}\,\cot(\kappa/2)+\bar{b}}{\sqrt{\det(A)}}\,\Big)-\arctan\Big(\,\frac{\bar{c}\,\cot(\kappa/2)-\bar{b}}{\sqrt{\det(A)}}\,\Big)
$$
with constants $\bar{a}, \bar{b}, \bar{c}$ from the relation
$$
\begin{pmatrix} \bar{a} & \bar{b}\\
\bar{b}& \bar{c} \end{pmatrix}
= \begin{pmatrix} \cos \alpha & \sin \alpha\\
-\sin \alpha & \cos \alpha \end{pmatrix} 
\begin{pmatrix} a & b\\
b& c  \end{pmatrix} 
\begin{pmatrix} \cos \alpha & - \sin \alpha\\
\sin \alpha & \cos \alpha \end{pmatrix}.
$$
In particular, we have $\tilde{\kappa}=\pi$ if $\kappa=\pi$. 

Now, let $\kappa\neq \pi$, $\alpha=0$ for $\cD$.  Also, let $b=0$ in $L$.  In this case we can take   $\nu_1=a\wedge c$ and $\nu_2=a\vee c$ in \eqref{uniform parabolicity}. We note that $\tilde{\kappa}$ is determined by the simple relation
\begin{equation*}
\tan\Big(\frac{\widetilde{\kappa}}{\,2\,}\Big)=\sqrt{\frac{a}{c}}\tan\Big(\frac{\kappa}{\,2\,}\Big).
\end{equation*}
\end{example}

We are ready to pose our Sobolev  regularity results on conic domains.  We formulate them into two  theorems to handle random and non-random coefficients separately.  The proofs of them are located in Section \ref{sec:main proofs}.  
Note that the admissible range of $\theta$ for non-random coefficients is relatively wider than that of random coefficients.

\begin{thm}(SPDE on conic domains with non-random coefficients)
\label{main result}
Let $\cL=\sum_{ij}a^{ij}(t)D_{ij}$ be non-random,  $p\in[2,\infty)$, and   $\gamma \geq -1$. Also assume that    Assumptions \ref{ass M} and  \ref{ass coeff} hold, and     $\theta, \Theta\in\bR$  satisfy
\begin{equation}
    \label{theta11}
p(1-\lambda^+_{c,\cL})<\theta<p(d-1+\lambda^-_{c,\cL}), \qquad  d-1<\Theta<d-1+p.
\end{equation}
Then for any $f^0\in\bK^{\gamma \vee 0}_{p,\theta+p,\Theta+p}(\cD,T)$, $\tbf=(f^1,\cdots,f^d)\in \bK^{\gamma+1}_{p,\theta,\Theta}(\cD,T,d)$, $g\in\bK^{\gamma+1}_{p,\theta,\Theta}(\cD,T,l_2)$, and $u_0\in\bU^{\gamma+2}_{p,\theta,\Theta}(\cD)$,  equation \eqref{stochastic parabolic equation} has a unique solution $u$ in  the class $\cK^{\gamma+2}_{p,\theta,\Theta}(\cD,T)$ and moreover we have
\begin{eqnarray}\label{main estimate}
\|u\|_{\cK^{\gamma+2}_{p,\theta,\Theta}(\cD,T)}
&\leq& C\big(\|f^0\|_{\bK^{\gamma \vee 0}_{p,\theta+p,\Theta+p}(\cD,T)}
+ \|\tbf\|_{\bK^{\gamma+1}_{p,\theta,\Theta}(\cD,T,d)}+\|g\|_{\bK^{\gamma+1}_{p,\theta,\Theta}(\cD,T,l_2)}\nonumber\\
&&\;\;\quad+\|u_0\|_{\bU^{\gamma+2}_{p,\theta,\Theta}(\cD)}\big),
\end{eqnarray}
where the constant $C$ depends only on $\cM,d,p,\theta,\Theta,\cL, \gamma$.  In particular, it is independent of $T$.
\end{thm}

\begin{remark}
(i) A particular  result of the above theorem  is introduced in  \cite{CKL 2019+} (cf. \cite{CKLL 2018}).  More precisely,   the combination of 
Theorem 2.8 and Corollary 2.11 in \cite{CKL 2019+} covers the case 
$$
 \cL=\Delta, \quad   \Theta=d=2, \quad \cD=\cD^{(\kappa)}\text{ of}\;\; \eqref{wedge in 2d}.
$$

(ii) If $\gamma \geq 0$, the separation of two terms $f^0$ and $\tbf=(f^1,\cdots,f^d)$ in our equation is redundant and we simply pose $f\in\bK^{\gamma}_{p,\theta+p,\Theta+p}(\cD,T)$ instead. This is because, by \eqref{eqn 4.16.1}, we have
$
h^0+\sum_{i=1}^d h^i_{x^i}\in K^{\gamma}_{p,\theta+p,\Theta+p}(\cD)
$ for
$
h^0\in K^{\gamma}_{p,\theta+p,\Theta+p}(\cD),\;\; h^i \in K^{\gamma+1}_{p,\theta,\Theta}(\cD),\;i=1,\ldots,d.
$
The corresponding change in the estimate \eqref{main estimate} is clear.
\end{remark}

\begin{thm}(SPDE on conic domains with random coefficients)
    \label{main result-random}
Let $\cL=\sum_{ij}a^{ij}(\omega,t)D_{ij}$ be random,  $p\in[2,\infty)$, and   $\gamma \geq -1$. Also assume that    Assumptions \ref{ass M} and  \ref{ass coeff} hold,  $d-1<\Theta<d-1+p$, and 
\begin{equation}
    \label{theta}
p\big(1-\lambda_{c}(\nu_1,\nu_2)\big)<\theta<p\big(d-1+\lambda_{c}(\nu_1,\nu_2)\big).
\end{equation}
Then all the claims of Theorem \ref{main result} hold with a constant $N=N(\cM,d,p,\gamma,\theta,\Theta,\nu_1,\nu_2)$. 
\end{thm}

\begin{remark}
By Proposition \ref{critical exponents},  \eqref{theta} is fulfilled if 
\begin{equation}
\label{eqn 8.28.10}
p\left(\frac{d+2}{2}-\sqrt{\frac{\nu_1}{\nu_2}}\sqrt{\Lambda_{\cD}+\frac{(d-2)^2}{4}}\right)<\theta< p\left(\frac{d-2}{2}+\sqrt{\frac{\nu_1}{\nu_2}}\sqrt{\Lambda_{\cD}+\frac{(d-2)^2}{4}}\right).
\end{equation}
In the case of $L=\Delta$, by Proposition \ref{critical exponents},  \eqref{theta11} is fulfilled if 
$$
p\left(\frac{d}{2}-\sqrt{\Lambda_{\cD}+\frac{(d-2)^2}{4}}\right)<\theta< p\left(\frac{d}{2}+\sqrt{\Lambda_{\cD}+\frac{(d-2)^2}{4}}\right).
$$
\end{remark}

\begin{remark}
 By \eqref{space K norm},  inequality \eqref{main estimate}  yields \eqref{main estimate intro}.  In particular, if $\gamma=-1$ and $u(0,\cdot)\equiv 0$, then we have
\begin{eqnarray*}
&&\bE \int^T_0 \int_{\cD} \left(\vert\rho^{-1}u\vert ^p+\vert u_x\vert ^p\right) \rho_{\circ}^{\theta-\Theta}\rho^{\Theta-d}\, dx\,dt \nonumber
 \\
&\leq& C\, \bE
  \int^T_0 \int_{\cD} \Big( \vert \rho f^0\vert ^p+\sum_{i=1}^d \vert f^i\vert ^p+\vert g\vert ^p_{\ell_2}\Big)
 \rho_{\circ}^{\theta-\Theta}\rho^{\Theta-d}\, dx\,dt.
\end{eqnarray*}
\end{remark}

\begin{remark}
The solutions $u$ in Theorems \ref{main result}  and \ref{main result-random} satisfy zero Dirichlet boundary condition.  Indeed, under the assumption $d-1<\Theta<d-1+p$, \cite[Theorem 2.8]{doyoon} implies that the trace operator  is well defined for functions in $\bK^1_{p,\theta-p, \Theta-p}(\cD,T)$, and hence by Lemma \ref{property1} (iv)  we have $u\vert _{\partial \cD}=0$.
\end{remark}

Here comes  our  H\"older regularity properties of solutions  on conic domains. 

\begin{thm}[H\"older estimates on conic domains]
\label{cor 8.10}
Let $p\in [2,\infty)$, $\theta, \Theta\in \bR$, and $u\in \cK^{\gamma+2}_{p,\theta,\Theta}(\cD,T)$
  be the solution  taken from Theorem \ref{main result} (or  from Theorem \ref{main result-random}).

(i) If $\gamma+2-\frac{d}{p}\geq n+\delta$, where $n\in \bN_0$ and $\delta\in (0,1]$, then for any $0\le k\leq n$, 
$$
\vert \rho^{k-1+\frac{\Theta}{p}} \rho^{(\theta-\Theta)/p}_{\circ} D^{k}u(\omega,t,\cdot)\vert _{\cC(\cD)}+
 [\rho^{n-1+\delta+\frac{\Theta}{p}} \rho^{(\theta-\Theta)/p}_{\circ} D^{n}( \omega,t,\cdot)]_{\cC^{\delta}(\cD)}<\infty
$$
holds for a.e. $(\omega,t)$, in particular, 
\begin{equation}
\label{eqn 8.10.21}
\vert u(\omega,t,x)\vert \leq C(\omega,t) \rho^{1-\frac{\Theta}{p}}(x) \rho^{(-\theta+\Theta)/p}_{\circ}(x)\quad \text{for all }x\in\cD.
\end{equation}

(ii) Let
$$
2/p<\alpha<\beta\leq 1, \quad \gamma+2-\beta-d/p \geq m+\varepsilon,
$$
where $m\in \bN_0$ and $\varepsilon\in (0,1]$.  Put $\eta=\beta-1+\Theta/p$. Then for any $0\le k\leq m$,
\begin{align} \label{eqn 8.11.11}
&\bE \sup_{t\ne s\leq T} 
\frac
{\big\vert \rho^{\eta+k}  \rho^{(\theta-\Theta)/p}_{\circ} \big(D^ku(t,\cdot)-D^ku(s,\cdot)\big)\big\vert^p_{\cC(\cD)}}
{\vert t-s\vert ^{p\alpha/2-1}}<\infty, \\
& \bE \sup_{t\ne s\leq T} 
\frac
{\left[\rho^{\eta+m+\varepsilon} 
\rho^{(\theta-\Theta)/p}_{\circ} \left(D^mu(t,\cdot)-D^mu(s,\cdot)\right)\right]^p_{\cC^{\varepsilon}(\cD)}}
{\vert t-s\vert ^{p\alpha/2-1}} <\infty.  \label{eqn 8.11.21}
\end{align}

\end{thm}

\begin{proof}
(i) By definition, for almost all $(\omega, t)$, we have $u(\omega,t,\cdot)\in K^{\gamma+2}_{p,\theta-p,\Theta-p}(\cD)$. Thus (i) is a consequence of   \eqref{eqn 8.21.1}.  Similarly, the  claims  of (ii)  follow from  \eqref{eqn 8.21.1}  \eqref{eqn 8.10.10}, and the observation
\begin{eqnarray*}
&&\bE \sup_{t\ne s\le T} \frac{\|\psi^{\beta-1}(u(t)-u(s))\|^p_{K^{\gamma+2-\beta}_{p,\theta,\Theta}(\cD)}}{\vert t-s\vert ^{(\alpha/2-1/p)p}} \\
&\sim&
\bE \sup_{t\ne s\le T} \frac{\hspace{1cm}\|u(t)-u(s)\|^p_{K^{\gamma+2-\beta}_{p,\theta+\beta p-p,\Theta+\beta p-p}(\cD)}}{\vert t-s\vert ^{(\alpha/2-1/p)p}}.
\end{eqnarray*}
\end{proof}

\begin{remark}
(i)  \eqref{eqn 8.10.21} tells how fast the solution  from Theorem \ref{main result} (or  Theorem \ref{main result-random}) vanishes near the boundary. Near boundary points away from the vertex, $u$ is controlled by   $\rho^{1-\Theta/p}$and, if $p>\Theta$, the decay near the vertex is not slower than   $\rho^{1-\theta/p}_{\circ}$. 

(ii) In \eqref{eqn 8.11.11} and \eqref{eqn 8.11.21}, $\alpha/2-1/p$ is the  H\"older exponent in time and $\eta$ is related to the decay rate near the boundary.  
As  $\alpha/2-1/p\to 1/2-1/p$,   $\eta$ must  increase accordingly.

(iii) Suppose $\theta=d$ satisfies \eqref{eqn 8.28.10}, and  let $u\in \cK^{1}_{p,d,d}(\cD,T)$ be the solution from Theorem \ref{main result-random}. Assume
$$
\kappa_0:=1-\frac{(d+2)}{p}>0.
$$
Then for any $\kappa\in (0,\kappa_0)$, we have
\begin{equation}
  \label{eqn 8.11.12}
\bE \sup_{t\leq T} \sup_{x,y\in\cD} \Big\vert \frac{\vert u(t,x)-u(t,y)\vert }{\vert x-y\vert ^{\kappa}}\Big\vert^p + \bE \sup_{t\ne s\leq T}\sup_{x\in \cD}\Big\vert \frac{\vert u(t,x)-u(s,x)\vert }{\vert t-s\vert ^{\kappa/2}}\Big\vert^p <\infty.
\end{equation}
Indeed,  \eqref{eqn 8.11.12} can be  obtained from \eqref{eqn 8.11.11} and \eqref{eqn 8.11.21} with appropriate choices of 
$\alpha,\beta$. For the first part, to apply \eqref{eqn 8.11.21}  we  take $\beta=\kappa_0-\kappa+2/p$ such that $2/p<\beta<1$, and take $\varepsilon=1-\beta-d/p=\kappa=-\eta$. For the second part, we use \eqref{eqn 8.11.11} with $\alpha=\kappa+2/p, \beta=1-d/p$ so that $1-\alpha p/2=-p\kappa/2$.
\end{remark}

\mysection{Key  estimates on conic domains}

In this section we consider the solutions  to SPDEs having a non-random operator. We fix a deterministic operator  
\begin{equation}
\label{8.29.1}
L_0:=\sum_{i,j}\alpha^{ij}(t)D_{ij}\,\, \in \, \cT_{\nu_1,\nu_2}.
\end{equation}
See Definition \ref{lambda2}. 
We will estimate the zeroth order derivative of the solution of the equation
\begin{equation}\label{one event equation}
d u =\left( L_0 u+f^0+\sum_{i=1}^d f^i_{x^i}\right)dt +\sum^{\infty}_{k=1} g^kdw_t^k,\quad   t>0, \;x\in \cD(\cM).
\end{equation}

Let  $G(t,s,x,y)$   denote the Green's function for the operator $\partial_t-L_0$ on $\cD=\cD(\cM)$. By definition (cf. \cite[Lemma 3.7]{Kozlov Nazarov 2014}), $G$ is a nonnegative function such that  for any fixed $s\in \bR$ and $y\in\mathcal{D}$,  the function $v(t,x)=G(t,s,x,y)$ satisfies
\begin{align*}
&\big(\partial_t -L_0\big)v(t,x)=\delta(x-y)\delta(t-s) \quad  \text{in}\quad \bR \times \cD, \nonumber\\
& v(t,x)=0\quad \textrm{on}\quad  \bR \times\mathcal{\partial D} \; ; \quad  v(t,x)=0\quad \textrm{for} \quad t<s.
\end{align*}

Now, for any given
$$
 f^0\in \bL_{p,\theta+p,\Theta+p}(\cD,T),  \quad \textbf{f}=(f^1,\cdots,f^d)\in  \bL_{p,\theta,\Theta}(\cD,T,d), \quad
 $$
 $$
  g\in\bL_{p,\theta,\Theta}(\cD,T, \ell_2), \quad  u_0 \in L_p(\Omega; K^0_{\theta+2-p,\Theta+2-p}(\cD)),
 $$
we define  the function $\cR(u_0,f^0,\textbf{f},g)$ by
\begin{eqnarray}
&&\cR(u_0,f^0,\textbf{f},g)(t,x)\nonumber\\
&:=&\int_{\cD} G(t,0,x,y)u_0(y)dy\nonumber\\
&&+
\int^t_0\int_{\cD}G(t,s,x,y)f(s,y)dyds-\sum_{i=1}^d \int^t_0\int_{\cD}G_{y^i}(t,s,x,y)f^i(s,y)dyds\nonumber\\
&&+\sum_{k=1}^{\infty}\int^t_0\int_{\cD}G(t,s,x,y)g^k(s,y)dy\,dw^k_s.   \label{eqn 8.21.11}
\end{eqnarray}
One immediately notices that the function $\cR(u_0,f^0,\textbf{f},g)$  is a representation of a solution of \eqref{one event equation} with zero boundary condition and initial condition $u(0,\cdot)=u_0(\cdot)$; see Lemma \ref{lemma  rep} in the next section. 
Our main result of this section is about this representation and it is given in the following lemma. 
\begin{lemma}\label{main est}
Let $T<\infty$,  $p\in [2,\infty)$ and let $\theta\in\bR$, $\Theta\in\bR$ satisfy 
$$
p(1-\lambda^+_{c,L_0})<\theta<p(d-1+\lambda^-_{c,L_0})\quad\text{and}\quad d-1<\Theta<d-1+p.
$$
If $f^0\in \bL_{p,\theta+p,\Theta+p}(\cD,T)$, $\tbf\in  \bL^d_{p,\theta,\Theta}(\cD,T,d)$, $g\in\bL_{p,\theta,\Theta}(\cD,T, \ell_2)$, and $u_0\in  L_p(\Omega; K^0_{\theta+2-p,\Theta+2-p}(\cD)):=L_p(\Omega,\rF_0; K^0_{\theta+2-p,\Theta+2-p}(\cD))$, then  $u:=\cR(u_0,f^0,\tbf,g)$ belongs to  $\bL_{p,\theta-p,\Theta-p}(\cD,T)$ and  the estimate
 \begin{eqnarray*}
\|u\|_{\bL_{p,\theta-p,\Theta-p}(\cD,T)}&\leq& C \Big(\|f^0\|_{\bL_{p,\theta+p,\Theta+p}(\cD,T)} + \|\tbf\|_{\bL_{p,\theta,\Theta}(\cD,T,d)}\nonumber \\
&&\quad\quad+
\|g\|_{\bL_{p,\theta,\Theta}(\cD,T,\ell_2)} +\|u_0\|_{ L_p(\Omega; K^0_{\theta+2-p,\Theta+2-p}(\cD))}\Big)\nonumber\\
\end{eqnarray*}
holds, where  $C=C(\cM,d,p,\theta,\Theta,L_0)$.   Moreover, if 
\begin{equation*}
p\left(1-\lambda_c(\nu_1,\nu_2)\right)<\theta< p\left(d-1+ \lambda_c(\nu_1,\nu_2)\right),
\end{equation*}
then the constant $C$ depends only on $\cM,d,p,\theta,\Theta, \nu_1$ and $\nu_2$.
\end{lemma}

To prove Lemma~\ref{main est}, we use the following two results. Lemma \ref{lemma3.1} gathers rather technical but important inequalities we keep using in this section.

\begin{lemma}\label{lemma3.1}
 (i) Let $\alpha+\beta>0,\ \beta > 0$, and $\gamma>0$. Then  for any $a\geq b>0$  
 \begin{align*}
\int^{\infty}_0 \frac{1}{\left(a+\sqrt{t}\right)^{\alpha}\left(b+\sqrt{t}\right)^{\beta+\gamma}t^{1-\frac{\gamma}{2}}}dt \leq \frac{C}{a^\alpha b^\beta},
\end{align*}
where $C= C(\alpha,\beta,\gamma)$.

(ii) Let $\sigma>0,\ \alpha+\gamma>-d$, $\gamma>-1$ and $\beta,\ \nu \in \bR$. Then   for any $x\in\cD$, 
\begin{align*}
\int_{\mathcal{D}} \frac{\vert y\vert ^{\alpha}}{\left(\vert y\vert +1\right)^{\beta}}\frac{\rho(y)^{\gamma}}{\left(\rho(y)+1\right)^{\nu}}\ e^{-\sigma \vert x-y\vert ^2} dy \leq 
C \left(\vert x\vert +1\right)^{\alpha-\beta}\left(\rho(x)+1\right)^{\gamma-\nu},
\end{align*}
where $C=C(\cM,d, \alpha, \beta, \gamma,\nu,\sigma)$.
\end{lemma}
\begin{proof}
See Lemma 3.2 and Lemma 3.7 in  \cite{ConicPDE}.
\end{proof}
For the operator $L_0$,  we  take the constants $K_0, \lambda^+_{c,L_0}, \lambda^-_{c,L_0}$ and the operator $\hat{L}_0$ from Definition \ref{lambda}. 
\begin{lemma}\label{Green estimate}
 Let $\lambda^+\in \big(0,\lambda^+_{c,L_0}\big)$ and  $\lambda^-\in\big(0,\lambda^-_{c,L_0}\big)$. Denote
 $$
 K_0^+=K_0(L_0,\cM,\lambda^+), \quad K_0^-=K_0(\hat{L}_0,\cM,\lambda^-).$$
 Then,  there exist positive constants
  $C=C(\mathcal{M},\nu_1,\nu_2,\,\lambda^{\pm}, K_0^{\pm})$
and $\sigma=\sigma(\nu_1,\nu_2)$ such that for any $t>s$ and $x,y\in \mathcal{D}(\cM)$, the estimates
\begin{align*}
&(i)\quad G(t,s,x,y)\leq \frac{C}{(t-s)^{d/2}}J_{t-s,x}\,J_{t-s,y}\,R^{\lambda^+-1}_{t-s,x}\,R^{\lambda^--1}_{t-s,y}\,e^{-\sigma\frac{\vert x-y\vert ^2}{t-s}}\\
&(ii)\quad \big\vert \nabla_y G(t,s,x,y)\big\vert \leq \frac{C}{(t-s)^{(d+1)/2}}J_{t-s,x}\,R^{\lambda^+-1}_{t-s,x}\,R^{\lambda^--1}_{t-s,y}\,e^{-\sigma\frac{\vert x-y\vert ^2}{t-s}}
\end{align*}
hold, where
  $$
 R_{t,x}:=\frac{\rho_{\circ}(x)}{\rho_{\circ}(x)+\sqrt{t}},\quad J_{t,x}:=\frac{\rho(x)}{\rho(x)+\sqrt{t}}.
  $$
  In particular, if $\lambda^{\pm}\in (0,\lambda_c(\nu_1,\nu_2))$, then $C$ depends only on $\cM, \nu_1,\nu_2, \lambda^{\pm}$.
\end{lemma}
\begin{proof}
(i) See inequality (2.8) in \cite{Green}. 

(ii) Denote $\hat{G}(t,s,x,y)=G(-s,-t,y,x)$. Then $\hat{G}$ is  the Green's function of the operator $\partial_t-\hat{L}_0$, 
where $\hat{L}_0:=\sum_{i,j}\alpha^{ij}(-t)D_{ij}$.  Then by inequality (2.14) of \cite{Green} applied to $\hat{G}$,  for any $\lambda^+\in(0,\lambda_c^+)$ and $\lambda^-\in(0,\lambda_c^-)$, there exist constant $C, \sigma>0$, with the dependencies prescribed in the lemma,  such that
\begin{align*}
\vert \nabla_x \hat{G}(t,s,x,y)\vert \leq \frac{C}{(t-s)^{(d+1)/2}}J_{t-s,y}R^{\lambda^--1}_{t-s,x}R^{\lambda^+-1}_{t-s,y}e^{-\sigma\frac{\vert x-y\vert ^2}{t-s}}
\end{align*}
for any $t>s$ and $x,\,y\in\cD$.  This and  the fact $\nabla_y G(t,s,x,y)=\nabla_x \hat{G}(-s,-t,y,x)$ prove (ii).
\end{proof}

Since $\cR(u_0,f^0,\tbf,g)=\cR(u_0,0,0,0)+\cR(0,f^0,\tbf,0)+\cR(0,0,0,g)$ with $0$ as zero functions in their corresponding function spaces, we will treat these three parts  separately in  following three lemmas and then combine them to obtain the claim of Lemma \ref{main est}.
Especially, the stochastic part $\cR(0,0,0,g)$ is important in this article and  elaborated thoroughly in Lemma \ref{main est2}.

\begin{lemma}\label{main est3}
Let $p\in(1,\infty)$, and let $\theta\in\bR$, $\Theta\in\bR$ satisfy
$$
p(1-\lambda^+_{c,L_0})<\theta<p(d-1+\lambda^-_{c,L_0})\quad\text{and}\quad d-1<\Theta<d-1+p.
$$
If $u_0\in L_p(\Omega; K^0_{\theta+2-p,\Theta+2-p}(\cD))$, then $u=\cR(u_0,0,0,0)$ belongs to $\bL_{p,\theta-p,\Theta-p}(\cD,T)$ and
\begin{align*}
\|u\|_{\bL_{p,\theta-p,\Theta-p}(\cD,T)}\leq C \|u_0\|_{L_p(\Omega; K^0_{\theta+2-p,\Theta+2-p}(\cD))}
\end{align*}
holds, where  $C=C(\cM,d,p,\theta,\Theta,L_0)$. Moreover, if
\begin{equation}\label{theta range restricted}
p\big(1-\lambda_c(\nu_1,\nu_2)\big)<\theta<p\big(d-1+\lambda_{c}(\nu_1,\nu_2)\big),
\end{equation}
then the constant $C$ depends only on $\cM,d,p,\theta,\Theta,\nu_1$ and $\nu_2$.
\end{lemma}

\begin{proof}
Green's function itself is not random. Hence, recalling  the definitions of $\cR(u_0,0,0,0)$ and $\bL=\bK^0$,   for simplicity we may assume that $u_0$ and hence $u$ are non-random and we just prove
\begin{align}
\int^T_0\int_{\cD}\vert \rho^{-1}u\vert ^p\rho_{\circ}^{\theta-\Theta}\rho^{\Theta-d}dxdt\leq N\int_{\cD}\vert \rho^{-1+\frac{2}{p}}\,u_0\vert ^p\rho_{\circ}^{\theta-\Theta}\rho^{\Theta-d}dx.
\label{main inequality3}
\end{align}

{\bf 1}. Let us denote $\mu:=-1+(\theta-d+2)/p$, $\alpha:=-1+(\Theta-d+2)/p$, and
$$
h(x):=\rho_{\circ}(x)^{\mu-\alpha}\rho(x)^{\alpha}u_0(x).
$$
Then the claimed estimate \eqref{main inequality3} turns into a simpler form of
\begin{equation}\label{main inequality 3}
\Big\|\rho_{\circ}^{\mu-\alpha}\rho^{\alpha-\frac{2}{p}}u\Big\|_{L_p\left([0,T]\times \mathcal{D}\right)}\le N \|h\|_{L_p\left(\mathcal{D}\right)}.
\end{equation}

On the other hand, by the range of $\theta$ given in the condition, we can always find $\lambda^+\in\big(0,\lambda^+_{c,L_0}\big)$ and $\lambda^-\in\big(0,\lambda^-_{c,L_0}\big)$ satisfying
\begin{align}
-\frac{d-2}{p}-\lambda^+<\mu<\frac{d-2}{p}+\lambda^-.\label{inequality mu1}
\end{align}
Also, by the given range of $\Theta$  we have
\begin{align}
-1+\frac{1}{p}<\alpha<\frac{1}{p}.\label{inequality alpha1}
\end{align}
Hence, we can choose and fix the constants $\gamma$, $\beta$ satisfying
\begin{align*}
0<\gamma<\lambda^-+\frac{d-2}{p}-\mu\,,\,\quad 0<\beta<\frac{1}{p}-\alpha.
\end{align*}
Noting  $\frac{d-2}{p}<d-\frac{d}{p}$, $\frac1p<2-\frac1p$ which is due to condition $p\in(1,\infty)$,  we then have 
\begin{align}
0<\gamma<\lambda^-+d-\frac{d}{p}-\mu\,,\,\quad 0<\beta<2-\frac{1}{p}-\alpha.\label{gamma beta chosen 1}
\end{align}

Moreover, as $\lambda^+\in (0,\lambda^+_{c,L_0})$ and $\lambda^-\in (0,\lambda^-_{c,L_0})$, by Lemma~\ref{Green estimate} there exist constants $C=C(\cM,L_0,\nu_1,\nu_2, \lambda^{\pm}),\,\sigma=\sigma(\nu_1,\nu_2)>0$  such that
\begin{align}\label{20.09.15}
\nonumber G(t,0,x,y)&\leq C\,t^{-\frac{d}{2}}\,R^{\lambda^+-1}_{t,x}R^{\lambda^- -1}_{t,y}\,J_{t,x} J_{t,y}\,e^{-\sigma\frac{\vert x-y\vert^2}{t}}\\
& = C\,t^{-\frac{d}{2}}\,R^{\lambda^+-1}_{t,x}J_{t,x}\,R^{\gamma}_{t,y}\left(\frac{J_{t,y}}{R_{t,y}}\right)^{\beta}\, R^{\lambda^--\gamma}_{t,y}\left(\frac{J_{t,y}}{R_{t,y}}\right)^{1-\beta}\,e^{-\sigma\frac{\vert x-y\vert^2}{t}}
\end{align}
holds for all $t>s$ and $x,y\in\cD$. Let us prove estimate \eqref{main inequality 3}.

{\bf 2}. Using H\"older inequality and \eqref{20.09.15}, we have
\begin{align*}
\vert u(t,x)\vert&=\Big\vert \int_{\mathcal{D}}G(t,0,x,y)u_0(y)dy\Big\vert \\
&\leq \int_{\mathcal{D}}G(t,0,x,y)\vert y\vert ^{-\mu+\alpha}\rho(y)^{-\alpha}\vert h(y)\vert dy\\
&\leq C\cdot I_1(t,x)\cdot I_2(t,x),
\end{align*}
where $q=p/(p-1)$; $\frac1p+\frac1q=1$,
$$
I_1(t,x)= \left( \int_{\mathcal{D}}t^{-\frac{d}{2}}\,e^{-\sigma\frac{\vert x-y\vert ^2}{t}}\cdot R_{t,x}^{(\lambda^+-1)p}J_{t,x}^{\,p}\cdot  K_{1}(t,y)\cdot \vert h(y)\vert ^pdy\right)^{1/p},
$$
and
$$
I_2(t,x)=\left(\int_{\mathcal{D}}t^{-\frac{d}{2}}\,e^{-\sigma\frac{\vert x-y\vert ^2}{t}}\cdot K_{2}(t,y)\cdot \vert y\vert ^{(-\mu+\alpha)q}\rho^{-\alpha q}(y)dy\right)^{1/q}
$$
with
$$
K_{1}(t,y) =R^{\gamma p}_{t,y}\left(\frac{J_{t,y}}{R_{t,y}}\right)^{\beta p},\quad \, K_{2}(t,y)=R^{(\lambda^--\gamma)q}_{t,y}\left(\frac{J_{t,y}}{R_{t,y}}\right)^{(1-\beta)q}.
$$

 
{\bf 3}. We show that there exists a constant $C$ depending only on $\cM, d,p,\theta,\Theta, \nu_1, \nu_2$ and $\lambda^{-}$ such that 
\begin{align*}
I_2(t,x)\leq  C \left(\vert x\vert +\sqrt{t}\right)^{-\mu+\alpha}\left(\rho(x)+\sqrt{t}\right)^{-\alpha}.
\end{align*}
This is done by Lemma~\ref{lemma3.1} (ii). Indeed, by change of variables $y/\sqrt{t}\to y$ and the fact $\rho(y)/\sqrt{t}=\rho(y/\sqrt{t})$, we have
\begin{align*}
I_2^q(t,x)&=t^{-\frac{d}{2}}\int_{\mathcal{D}}e^{-\sigma\frac{\vert x-y\vert ^2}{t}}K_{2}(t,y)\vert y\vert ^{(-\mu+\alpha)q}\vert \rho(y)\vert ^{-\alpha q}dy\\
&=t^{-\mu q/2}\int_{\mathcal{D}}e^{-\sigma\vert \frac{x}{\sqrt{t}}-y\vert ^2}\frac{\vert y\vert ^{(\lambda^--\mu-\gamma-1+\alpha+\beta)q}}{(\vert y\vert +1)^{(\lambda^--\gamma-1+\beta)q}}\cdot\frac{\rho(y)^{(1-\alpha-\beta)q}}{(\rho(y)+1)^{(1-\beta)q}}dy,
\end{align*}
for which we can apply Lemma~\ref{lemma3.1} since \eqref{gamma beta chosen 1} implies $(\lambda^--\mu-\gamma)q>-d$ and $(1-\alpha-\beta)q>-1$. Thus we get constant $C=C(\cM,d,p,\theta,\Theta, \lambda^-,\sigma)$ such that
$$
I_2^q(t,x) \leq  C \left(\vert x\vert +\sqrt{t}\right)^{(-\mu+\alpha)q}\left(\rho(x)+\sqrt{t}\right)^{-\alpha q}
$$
holds for all $t,x$.

{\bf 4}. To prove estimate \eqref{main inequality 3}, by Step 3 we first note 
\begin{align*}
\vert x\vert ^{\mu-\alpha}\rho(x)^{\alpha-\frac{2}{p}}\cdot \vert u(t,x)\vert &\leq  C\, \vert x\vert ^{\mu-\alpha}\rho(x)^{\alpha-\frac{2}{p}} \cdot I_1(t,x)\cdot I_2(t,x)\\
& \leq C\,\rho(x)^{-2/p}R^{\,\mu-\alpha}_{t,x} J^{\,\alpha}_{t,x}\cdot I_1(t,x)
\end{align*}
for any $t,x$. Using this and Fubini's Theorem, we have
\begin{align*}
\|\rho_{\circ}^{\mu-\alpha}\rho^{\alpha-\frac2p}u\|^p_{L_p([0,T]\times\mathcal{D})}&\leq C \int_0^T\int_{\mathcal{D}}\vert \rho(x)\vert ^{-2}\Big(R^{\,\mu-\alpha}_{t,x} J^{\,\alpha}_{t,x}\, I_1(t,x) \Big)^p\,dxdt\\
&=C\int_{\mathcal{D}}I_3(y)\cdot \vert h(y)\vert ^pdy,
\end{align*}
where
\begin{align*}
I_3(y)&=\int^T_0t^{-\frac{d}{2}}K_{1}(t,y)\left(\int_{\mathcal{D}}e^{-\sigma\frac{\vert x-y\vert ^2}{t}}\, R^{\,(\lambda^++\mu-\alpha-1)p}_{t,x} J^{\,(\alpha+1) p}_{t,x}\rho(x)^{-2}\,dx\right)dt.
\end{align*}

Since \eqref{inequality mu1} and \eqref{inequality alpha1} imply $(\lambda^++\mu) p-2>-d$ and $(\alpha+1) p-2>-1$, by change of variables $x/\sqrt{t}\to x$, the fact $\rho(x)/\sqrt{t}=\rho(x/\sqrt{t})$, and Lemma~\ref{lemma3.1} (ii), we have
\begin{align*}
I_3(y)&=\int^T_0\frac{1}{t}K_{1}(t,y)\, \int_{\mathcal{D}}e^{-\sigma\vert x-\frac{y}{\sqrt{t}}\vert ^2}\frac{\vert x\vert ^{(\lambda^++\mu-\alpha-1)p}}{(\vert x\vert +1)^{(\lambda^++\mu-\alpha-1)p}}\frac{\rho(x)^{(\alpha+1) p-2}}{(\rho(x)+1)^{(\alpha+1) p}}\,dx\,dt\\
&\leq C\int_0^{\infty}K_{1}(t,y)\left(\rho(y)+\sqrt{t}\right)^{-2}dt\\
&= C\int_0^{\infty}\frac{\vert y\vert ^{(\gamma-\beta)p}}{\left(\vert y\vert +\sqrt{t}\right)^{(\gamma-\beta)p}}\cdot\frac{\rho(y)^{\beta p}}{\left(\rho(y)+\sqrt{t}\right)^{\beta p+2}}dt.
\end{align*}
Lastly, owing to $\gamma p>0$, $\beta p> 0$, and the fact $|y|\ge \rho(y)$ in $\cD$, we can apply Lemma~\ref{lemma3.1} (i) and we obtain
\begin{align*}
I_3(y)\leq C(\cM,d,p,\theta,\Theta,\nu_1,\nu_2, \lambda^{\pm}).
\label{Last constant}
\end{align*}
Hence, there exists a constant $C$ having the dependency described in the lemma such that 
\begin{align*}
\left\|\rho_{\circ}^{\mu-\alpha}\rho^{\alpha-\frac2p}u\right\|^p_{L_p([0,T]\times\mathcal{D})}\leq C \|h\|^p_{L_p(\mathcal{D})}.
\end{align*}
Estimate \eqref{main inequality 3} and the lemma are proved.

{\bf 5}. When $\theta$ obeys \eqref{theta range restricted},  we choose $\lambda^{\pm}$ in the interval $(0,\lambda_c(\nu_1,\nu_2))$. Then the constant $C$ of Green's function estimates in Lemma \ref{Green estimate} depends only on $\cM,\nu_1,\nu_2,\lambda^{\pm}$.  Therefore, in particular, constant $C$ in 
\eqref {20.09.15} does not depend on $L_0$.  Tracking the constants down through Steps 1, 2, 3, 4, we note that the constant in \eqref{main inequality 3} does not depend on the particular operator $L_0$. Rather, it depends on $\nu_1,\, \nu_2$ and hence $C=C(\cM,d,p,\theta,\Theta,\nu_1,\nu_2)$.  
\end{proof}
\begin{remark}\label{initial.p ge 2.}
For $\gamma\ge 0$, $\|u\|_{L_p(\bR^d)}\le \|u\|_{H^{\gamma}_p(\bR^d)}$ is a basic property of the space of Bessel potentials. This with Lemma \ref{property1} and Definition \ref{defn 8.19},  in the context of Lemma \ref{main est3}, yields 
\begin{equation*}
\|u_0\|_{L_p(\Omega; K^0_{\theta+2-p,\Theta+2-p}(\cD))}\le \|u_0\|_{L_p(\Omega; K^{1-2/p}_{\theta+2-p,\Theta+2-p}(\cD))}=\|u_0\|_{\bU^{1}_{p,\theta,\Theta}(\cD)}
\end{equation*}
if $p\ge 2$.
\end{remark}

\begin{lemma}\label{main est1}
Let $p\in (1,\infty)$ and let $\theta\in\bR$, $\Theta\in\bR$ satisfy 
$$
p\big(1-\lambda^+_{c,L_0}\big)<\theta<p\big(d-1+\lambda^-_{c,L_0}\big)\quad\text{and}\quad d-1<\Theta<d-1+p.
$$
If $f^0\in \bL_{p,\theta+p,\Theta+p}(\cD,T)$, $\tbf\in  \bL^d_{p,\theta,\Theta}(\cD,T,d)$, then $u:=\cR(0,f^0,\tbf,0)$ belongs to  $\bL_{p,\theta-p,\Theta-p}(\cD,T)$ and the estimate  
 \begin{eqnarray*}
\|u\|_{\bL_{p,\theta-p,\Theta-p}(\cD,T)}\leq C \big(\|f^0\|_{\bL_{p,\theta+p,\Theta+p}(\cD,T)}+ \|\tbf\|_{\bL_{p,\theta,\Theta}(\cD,T,d)}\big)
\end{eqnarray*}
holds, where  $C=C(\cM,d,p,\theta,\Theta,L_0)$.   Moreover, if 
\begin{equation*}
p\left(1-\lambda_c(\nu_1,\nu_2)\right)<\theta< p\left(d-1+ \lambda_c(\nu_1,\nu_2)\right),
\end{equation*}
then the constant $C$ depends only on $\cM,d,p,\theta,\Theta, \nu_1$ and $\nu_2$.
\end{lemma}
\begin{proof}
By the same reason explained in the beginning of the proof of  Lemma \ref{main est3},  we can  assume $f^0, \tbf$, and hence $u$ are non-random and we just prove
\begin{align}
\label{main est 2.1}
 \int_0^T  \int_{\cD}  \vert \rho^{-1}u\vert ^p  \rho_0^{\theta-\Theta} \rho^{\theta-d} dxdt  
\leq C  \int_0^T  \int_{\cD}  \left(\vert \rho\,f\vert ^p  +   \vert \tbf\vert ^p  \right)\rho_0^{\theta-\Theta}  \rho^{\theta-d}  dxdt .
\end{align}
Furthermore, when $\tbf=0$  estimate \eqref{main est 2.1} is already proved in \cite[Lemma 3.1]{ConicPDE}, the deterministic counterpart of this article.
Hence, we may assume $f^0=0$. Finally, for simplicity we further assume $f^2=\cdots=f^d=0$.

{\bf 1}. We denote $\mu:=(\theta-d)/p$ and $\alpha:=(\Theta-d)/p$ and set
$$
h(t,x)=\rho_{\circ}^{\mu-\alpha}(x)\rho^{\alpha}(x) f^1(t,x).
$$
Then \eqref{main est 2.1} turns into
\begin{equation}\label{main inequality 2-220830}
\Big\|\rho_0^{\mu-\alpha}\rho^{\alpha-1}u\Big\|_{L_p\left([0,T]\times \mathcal{D}\right)}\le C \|h\|_{L_p\left([0,T]\times \mathcal{D}\right)}.
\end{equation}

We prepare a few things as we did in Step 1 of the proof of Lemma \ref{main est3}. By the range of $\theta$ given in the statement, we can find $\lambda^+\in(0,\lambda^+_{c,L_0})$ and $\lambda^-\in(0,\lambda^-_{c,L_0})$ satisfying
\begin{align*}
1-\frac{d}{p}-\lambda^+<\mu<d-1-\frac{d}{p}+\lambda^-.
\end{align*}
Also, by the range of $\Theta$ given we have
\begin{align*}
-\frac{1}{p}<\alpha<1-\frac{1}{p}.
\end{align*}
Then we can choose and fix the constants $\gamma_1$, $\gamma_2$, $\beta_1$ and $\beta_2$ satisfying
\begin{align}
-\frac{d-1}{p} < \gamma_1 < \lambda^+ - 1 + \mu+\frac{1}{p},&\qquad 0<\gamma_2<\lambda^-+d-1-\frac{d}{p}-\mu\nonumber\\
0<\beta_1<\alpha+\frac{1}{p},\qquad\qquad&\qquad 0<\beta_2<1-\frac{1}{p}-\alpha.\label{inequality gamma beta2}
\end{align}
Moreover, since $\lambda^+\in (0,\lambda^+_c)$ and $\lambda^-\in (0,\lambda^-_c)$, by Lemma~\ref{Green  estimate} there exist constants $C=C(\cM,L_0,\nu_1, \nu_2, \lambda^{\pm}),\,\sigma=\sigma(\nu_1,\nu_2)>0$ such that
\begin{align}
\vert \nabla_y G(t,s,x,y)\vert &\leq \frac{C}{(t-s)^{(d+1)/2}}e^{-\sigma\frac{\vert x-y\vert ^2}{t-s}}\,J_{t-s,x}R^{\lambda^+-1}_{t-s,x}R^{\lambda^--1}_{t-s,y} \nonumber\\
&=\frac{C}{(t-s)^{(d+1)/2}}e^{-\sigma\frac{\vert x-y\vert ^2}{t-s}}R^{\gamma_1}_{t-s,x}\left(\frac{J_{t-s,x}}{R_{t-s,x}}\right)^{\beta_1}R^{\gamma_2}_{t-s,y}\left(\frac{J_{t-s,y}}{R_{t-s,y}}\right)^{\beta_2} \nonumber\\
&\qquad \times R^{\lambda^+-\gamma_1}_{t-s,x}\left(\frac{J_{t-s,x}}{R_{t-s,x}}\right)^{1-\beta_1}R^{\lambda^--1-\gamma_2}_{t-s,y}
\left(\frac{J_{t-s,y}}{R_{t-s,y}}\right)^{-\beta_2}  \label{deriv}
\end{align}
holds for all $t>s$ and $x,y\in\cD$. Now, we start proving \eqref{main inequality 2-220830}.

{\bf 2}.
By  H\"older inequality and \eqref{deriv}, we have
\begin{align}
\vert u(t,x)\vert &=\Big\vert \int^t_0\int_{\mathcal{D}}G_{y^1}(t,s,x,y)f^1(s,y)dyds\Big\vert  \nonumber\\
&\leq \int^t_0\int_{\mathcal{D}}\vert \nabla_y G(t,s,x,y)\vert \cdot\vert y\vert ^{-\mu+\alpha}\rho(y)^{-\alpha}\vert h(s,y)\vert dyds \nonumber\\
&\leq C\, I_1(t,x)\cdot I_2(t,x),  \label{I12}
\end{align}
where $q=p/(p-1)$,
\begin{align*}
&I_1(t,x)\\
&=\left(\int^t_0\int_{\mathcal{D}}\frac{1}{(t-s)^{(d+1)/2}}\,e^{-\sigma\frac{\vert x-y\vert ^2}{t-s}}K_{1,1}(t-s,x)K_{1,2}(t-s,y)\vert h(s,y)\vert ^pdyds\right)^{1/p}
\end{align*}
and
\begin{align*}
&I_2(t,x)\\
&=\left(\int^t_0\int_{\mathcal{D}}\frac{1}{(t-s)^{(d+1)/2}}\,e^{-\sigma\frac{\vert x-y\vert ^2}{t-s}}K_{2,1}(t-s,x)K_{2,2}(t-s,y)\vert y\vert ^{(-\mu+\alpha)q}\rho^{\alpha q}(y)dyds\right)^{1/q}
\end{align*}
with
\begin{align*}
K_{1,1}(t,x)=R^{\gamma_1p}_{t,x}\left(\frac{J_{t,x}}{R_{t,x}}\right)^{\beta_1p},\quad &
 K_{1,2}(t,y)=R^{\gamma_2 p}_{t,y}\left(\frac{J_{t,y}}{R_{t,y}}\right)^{\beta_2 p},\\
K_{2,1}(t,x)=R^{(\lambda^+-\gamma_1)q}_{t,x}\left(\frac{J_{t,x}}{R_{t,x}}\right)^{(1-\beta_1)q},\quad 
&K_{2,2}(t,y)=R^{(\lambda^--1-\gamma_2)q}_{t,y}\left(\frac{J_{t,y}}{R_{t,y}}\right)^{-\beta_2q}.
\end{align*}
 
{\bf 3}.  We   show that there exists a constant $C=C(\cM,d,p,\theta,\Theta,\nu_1,\nu_2)>0$  such that
\begin{equation}\label{I2}
I_2(t,x)\leq C \vert x\vert ^{-\mu+\alpha}\rho(x)^{-\alpha+\frac{1}{q}}
\end{equation}
holds for all $t,x$; we note that the right hand side is independent of $t$.

First, by change of variables $y/\sqrt{t-s}\to y$ and Lemma~\ref{lemma3.1} (ii), which we can apply since \eqref{inequality gamma beta2} gives $(\lambda^--1-\mu-\gamma_2)q>-d$ and $(-\alpha-\beta_2)q>-1$,  we have
\begin{align*}
&\quad\frac{1}{(t-s)^{(d+1)/2}}\int_{\mathcal{D}}e^{-\sigma\frac{\vert x-y\vert ^2}{t-s}}K_{2,2}(t-s,y)\vert y\vert ^{(-\mu+\alpha)q}\vert \rho(y)\vert ^{-\alpha q}dy\\
&=(t-s)^{-(\mu q+1)/2}\int_{\mathcal{D}}e^{-\sigma\vert \frac{x}{\sqrt{t-s}}-y\vert ^2}\frac{\vert y\vert ^{(\lambda^--\mu-1-\gamma_2+\alpha+\beta_2)q}}{(\vert y\vert +1)^{(\lambda^--1-\gamma_2+
\beta_2)q}}\cdot\frac{\rho(y)^{(-\alpha-\beta_2)q}}{(\rho(y)+1)^{-\beta_2 q}}dy\\
&\leq  C (t-s)^{-1/2}\left(\vert x\vert +\sqrt{t-s}\right)^{(-\mu+\alpha)q}\left(\rho(x)+\sqrt{t-s}\right)^{-\alpha q},
\end{align*}
where $C=C(\cM,d,p,\theta,\Theta,\nu_1,\nu_2)$. Using this, we have
\begin{align*}
&I_{2}^q(t,x)\\
&\leq C\int^t_0K_{2,1}(t-s,x)\cdot(t-s)^{-1/2}\left(\vert x\vert +\sqrt{t-s}\right)^{(-\mu+\alpha)q}\left(\rho(x)+\sqrt{t-s}\right)^{-\alpha q}ds\\
&\le C\int^t_{-\infty}\frac{\vert x\vert ^{(\lambda^+-1-\gamma_1+\beta_1)q}}{(\vert x\vert +\sqrt{t-s})^{(\lambda^+-1+\mu-\alpha-\gamma_1+\beta_1)q}}\cdot\frac{\rho(x)^{(1-\beta_1)q}}{(\rho(x)+\sqrt{t-s})^{(1+\alpha-\beta_1)q}}\cdot \frac{1}{(t-s)^{1/2}}ds.
\end{align*}
Then, the change of variable $t-s\to s$ and  Lemma~\ref{lemma3.1}  (i), which we can apply since we have $(\lambda^++\mu-\gamma_1)q >1 $ and $(1+\alpha-\beta_1)q> 1$ from \eqref{inequality gamma beta2},  we further obtain
\begin{align*}
I_2^q(t,x)&\leq C \vert x\vert ^{(-\mu+\alpha)q}\rho(x)^{-\alpha q+1},
\end{align*}
which is equivalent to \eqref{I2}.

{\bf 4}. Now, by \eqref{I2} and   \eqref{I12}  we have
\begin{align*}
\vert u(t,x)\vert \leq C\, I_1(t,x)\cdot I_2(t,x)\leq C\, \vert x\vert ^{-\mu+\alpha}\rho(x)^{-\alpha+\frac{1}{q}}\, I_1(t,x)
\end{align*}
and hence
\begin{align*}
\|\rho_0^{\mu-\alpha}\rho^{\alpha-1}u\|^p_{L_p([0,T]\times\mathcal{D})}&\leq C \int_0^T\int_{\mathcal{D}}\vert \rho(x)\vert ^{-1}I_1^p(t,x)\,dxdt\\
&=C\int_0^T\int_{\mathcal{D}}I_3(s,y)\cdot \vert h(s,y)\vert ^pdyds,
\end{align*}
where
\begin{align*}
I_3(s,y)=\int^T_s\int_{\mathcal{D}}\frac{1}{(t-s)^{(d+1)/2}}e^{-\sigma\frac{\vert x-y\vert ^2}{t-s}}K_{1,1}(t-s,x)K_{1,2}(t-s,y)\rho(x)^{-1}\,dxdt.
\end{align*}
 By change of variables $t-s\to t$ followed by $x/\sqrt{t}\to x$ and Lemma~\ref{lemma3.1} (ii) with $\gamma_1p-1>-d$ and $\beta_1p-1>-1$ from \eqref{inequality gamma beta2}, we have
\begin{align*}
I_3(s,y)&=\int^T_s\frac{1}{(t-s)^{(d+1)/2}}K_{1,2}(t-s,y)\left(\int_{\mathcal{D}}e^{-\sigma\frac{\vert x-y\vert ^2}{t-s}}K_{1,1}(t-s,x)\rho(x)^{-1}\,dx\right)dt\\
&\leq \int^{\infty}_0\frac{1}{t}K_{1,2}(t,y)\left(\int_{\mathcal{D}}\frac{\vert x\vert ^{(\gamma_1-\beta_1)p}}{(\vert x\vert +1)^{(\gamma_1-\beta_1)p}}\frac{\rho(x)^{\beta_1p-1}}{(\rho(x)+1)^{\beta_1p}}e^{-\sigma\vert x-\frac{y}{\sqrt{t}}\vert ^2}\,dx\right)dt\\
&\leq C\int_0^{\infty}K_{1,2}(t,y)\left(\rho(y)+\sqrt{t}\right)^{-1}t^{-1/2}dt\\
&= C\int_0^{\infty}\frac{\vert y\vert ^{(\gamma_2-\beta_2)p}}{\left(\vert y\vert +\sqrt{t}\right)^{(\gamma_2-\beta_2)p}}\cdot\frac{\rho(y)^{\beta_2p}}{\left(\rho(y)+\sqrt{t}\right)^{\beta_2p+1}}\cdot\frac{1}{t^{1/2}}dt.
\end{align*}
Lastly, due to $\gamma_2p>0$ and $\nu_2p> 0$, Lemma~\ref{lemma3.1} (i)  yields
\begin{align*}
I_3(s,y)\leq C(\cM,d,p,\theta,\Theta, \nu_1,\nu_2).
\end{align*}
Hence, there exists a constant $C$ having the dependency described in the lemma  such that 
\begin{align*}
\left\|\rho_{\circ}^{\mu-\alpha}\rho^{\alpha-1}u\right\|^p_{L_p([0,T]\times\mathcal{D})}\leq C \|h\|^p_{L_p([0,T]\times\mathcal{D})}.
\end{align*} 
\eqref{main inequality 2-220830} and the lemma are proved. 

{\bf 5}.  The last part of the claim related to the range of $\theta$ holds by the same reason explained in Step 5 of the proof of Lemma \ref{main est3}.
\end{proof}
Now, we move on to the stochastic part, the most important and involved one.
\begin{lemma}\label{main est2}
Let $p\in [2,\infty)$ and let $\theta\in\bR$, $\Theta\in\bR$ satisfy 
$$
p(1-\lambda^+_{c,L_0})<\theta<p(d-1+\lambda^-_{c,L_0})\quad\text{and}\quad d-1<\Theta<d-1+p.
$$
If $g\in\bL_{p,\theta,\Theta}(\cD,T, \ell_2)$, then $u:=\cR(0,0,0,g)$ belongs to  $\bL_{p,\theta-p,\Theta-p}(\cD,T)$ and  the estimate
 \begin{eqnarray}\label{main inequality2}
\|u\|_{\bL_{p,\theta-p,\Theta-p}(\cD,T)}\leq C \|g\|_{\bL_{p,\theta,\Theta}(\cD,T,\ell_2)}
\end{eqnarray}
holds, where  $C=C(\cM,d,p,\theta,\Theta,L_0)$.   Moreover, if
\begin{equation*}
p\big(1-\lambda_c(\nu_1,\nu_2)\big)<\theta< p\big(d-1+ \lambda_c(\nu_1,\nu_2)\big),
\end{equation*}
then the constant $C$ depends only on $\cM,d,p,\theta,\Theta, \nu_1$, and $\nu_2$.
\end{lemma}

\begin{proof}

{\bf 1}. Again, we denote $\mu:=(\theta-d)/p$ and $\alpha:=(\Theta-d)/p$.
We put $h(\omega,t,x)=\rho_{\circ}^{\mu-\alpha}(x) \rho(x)^\alpha g(\omega,t,x)$ and recall
$$\Omega_T=\Omega \times (0,T], \quad L_p(\Omega_T \times \cD):=L_p(\Omega_T\times \cD, d\bP dt dx).
$$
   Then \eqref{main inequality2} is the same as
\begin{equation}\label{main inequality 2}
\big\|\rho_{\circ}^{\mu-\alpha}\rho^{\alpha-1}u\big\|^p_{L_p(\Omega_T\times\mathcal{D})}\le C\,\big\|\vert h\vert_{\ell_2}\big\|^p_{L_p\left(\Omega_T\times \mathcal{D}\right)}.
\end{equation}

 As we did in the proof of Lemma \ref{main est1}, we prepare few things.  By the range of $\theta$ given, we can find constants $\lambda^+\in(0,\lambda^+_{c,L_0})$ and $\lambda^-\in(0,\lambda^-_{c,L_0})$ satisfying
\begin{align*}
1-\frac{d}{p}-\lambda^+<\mu<d-\frac{d}{p}+\lambda^-.
\end{align*}
Also, by the range of $\Theta$ we have
\begin{align*}
-\frac{1}{p}<\alpha<1-\frac{1}{p}.
\end{align*}
Then we can choose and fix the constants $\gamma_1$, $\gamma_2$, $\beta_1$, and $\beta_2$ satisfying
\begin{align}
-\frac{d-2}{p} < \gamma_1 < \lambda^+ - 1 + \mu + \frac{2}{p},&\qquad 0<\gamma_2<\lambda^-+d-\frac{d}{p}-\mu\nonumber\\
\frac{1}{p}<\beta_1<\alpha+\frac{2}{p},\qquad\qquad & \qquad 0<\beta_2<2-\frac{1}{p}-\alpha.\label{inequality gamma beta3}
\end{align}

 Further, by Lemma \ref{Green estimate} there exist constants $C=C(\cM,L_0,\nu_1, \nu_2, \lambda^{\pm}),\,\sigma=\sigma(\nu_1,\nu_2)>0>0$ such that for any $t>s$ and $x,y\in\cD$,
\begin{align}
   \nonumber
G(t,s,x,y)\leq& \frac{C}{(t-s)^{d/2}}e^{-\sigma\frac{\vert x-y\vert^2}{t-s}} \,J_{t-s,x}\, J_{t-s,y}\,R^{\lambda^+-1}_{t-s,x}\,R^{\lambda^--1}_{t-s,y}\\   \nonumber
 =&C\, (t-s)^{-d/2}e^{-\sigma\frac{\vert x-y\vert ^2}{t-s}} R^{\gamma_1}_{t-s,x}\left(\frac{J_{t-s,x}}{R_{t-s,x}}\right)^{\beta_1}
R^{\gamma_2}_{t-s,y}\left(\frac{J_{t-s,y}}{R_{t-s,y}}\right)^{\beta_2}\\
& \times   R^{\lambda^+-\gamma_1}_{t-s,x}\left(\frac{J_{t-s,x}}{R_{t-s,x}}\right)^{1-\beta_1}
R^{\lambda^--\gamma_2}_{t-s,y}\left(\frac{J_{t-s,y}}{R_{t-s,y}}\right)^{1-\beta_2} \label{eqn 8.7.3}
\end{align}
holds.

{\bf 2}.  We first estimate the $p$-th moment  $\bE|u(t,x)|^p$ for any given $t$ and $x$. Using Burkholder-Davis-Gundy inequality and Minkowski's integral inequality, we have
\begin{align*}
\bE\vert u(t,x)\vert ^p&=\bE\Big\vert \sum_{k\in\bN}\int^t_0\int_{\mathcal{D}}G(t,s,x,y)g^k(s,y)dydw^k_s\Big\vert ^p\\
&\leq C\bE\left(\int_0^t\sum_{k\in\bN}\left(\int_{\cD}G(t,s,x,y)g^k(s,y)dy\right)^2 ds \right)^{p/2}\\
&\leq C\bE\left(\int_0^t\left(\int_{\cD}G(t,s,x,y)\vert g(s,y)\vert _{\ell_2}dy\right)^2 ds \right)^{p/2}\\
&=C\bE\left(\int_0^t\left(\int_{\cD}G(t,s,x,y)\vert y\vert ^{-\mu+\alpha}\rho(y)^{-\alpha}\vert h(s,y)\vert _{\ell_2}dy\right)^2 ds \right)^{p/2}.
\end{align*}
We denote
$$
I(\omega,t,x):=\left(\int_0^t\left(\int_{\cD}G(t,s,x,y)\vert y\vert ^{-\mu+\alpha}\rho(y)^{-\alpha}\vert h(\omega,s,y)\vert _{\ell_2}dy\right)^2 ds \right)^{1/2}.
$$
Then,  using  \eqref{eqn 8.7.3} and applying H\"older inequality twice for $x$ and then $t$, we get

\begin{align} 
I(\omega,t,x)&\leq C\left(\int_0^t\Big(\int_{\cD}I_1\cdot I_2\;dy\Big)^2ds\right)^{1/2}  \label{eqn 8.7.7}\\
&\leq C\|I_1(\omega,t,\cdot,x,\cdot)\|_{L_p((0,t)\times \cD, ds\,dy)}\Big\|\|I_2(t,\cdot,x,\cdot)\|_{L_{q}(\cD,dy)}\Big\|_{L_{r}((0,t),ds)}  \nonumber
\end{align}
where $q=\frac{p}{p-1}$, $r=\frac{2p}{p-2}\,(=\infty\text{ if }p=2)$,
\begin{align}
&I_1^p(\omega,t,s,x,y)    \label{I_1^p}\\
&= (t-s)^{-d/2}e^{-\sigma\frac{\vert x-y\vert ^2}{t-s}} \left(R^{\gamma_1}_{t-s,x}\left(\frac{J_{t-s,x}}{R_{t-s,x}}\right)^{\beta_1}R^{\gamma_2}_{t-s,y}\left(\frac{J_{t-s,y}}{R_{t-s,y}}\right)^{\beta_2}\right)^p\vert h(\omega,s,y)\vert _{\ell_2}^p   \nonumber \\
&= (t-s)^{-d/2}e^{-\sigma\frac{\vert x-y\vert ^2}{t-s}} \,\,K_{1,1}(t-s,x)\,K_{1,2}(t-s,y)\,\,\vert h(\omega,s,y)\vert _{\ell_2}^p,\label{I_1^p} \nonumber
\end{align}
and
\begin{align*}
&I_2^q(t,s,x,y)\\
&= (t-s)^{-d/2}e^{-\sigma\frac{\vert x-y\vert ^2}{t-s}}\\
&\quad  \times \left(R^{\lambda_1-\gamma_1}_{t-s,x}\left(\frac{J_{t-s,x}}{R_{t-s,x}}\right)^{1-\beta_1}R^{\lambda_2-\gamma_2}_{t-s,y}\left(\frac{J_{t-s,y}}{R_{t-s,y}}\right)^{1-\beta_2}\right)^q\vert y\vert ^{(-\mu+\alpha)q}\rho(y)^{-\alpha q}\\
&= (t-s)^{-d/2}e^{-\sigma\frac{\vert x-y\vert ^2}{t-s}} \,\,K_{2,1}(t-s,x)\,K_{2,2}(t-s,y)\,\,\vert y\vert ^{(-\mu+\alpha)q}\rho(y)^{-\alpha q},
\end{align*}
with
\begin{align*}
&K_{1,1}(t,x)=R^{\gamma_1p}_{t,x}\left(\frac{J_{t,x}}{R_{t,x}}\right)^{\beta_1p},\quad 
K_{1,2}(t,y)=R^{\gamma_2 p}_{t,y}\left(\frac{J_{t,y}}{R_{t,y}}\right)^{\beta_2 p},\\
&K_{2,1}(t,x)=R^{(\lambda^+-\gamma_1)q}_{t,x}\left(\frac{J_{t,x}}{R_{t,x}}\right)^{(1-\beta_1)q},\quad
K_{2,2}(t,y)=R^{(\lambda^--\gamma_2)q}_{t,y}\left(\frac{J_{t,y}}{R_{t,y}}\right)^{(1-\beta_2)q}.
\end{align*}

Note, by \eqref{eqn 8.7.7} we have
\begin{eqnarray}
      \label{eqn 8.7.8}
\bE \vert u(t,x)\vert ^p &\leq& C \bE I^p(t,x)  \\
&\leq& C \Big\|\|I_2(t,\cdot,x,\cdot)\|_{L_{q}(\cD,dy)}\Big\|^p_{L_{r}((0,t),ds)}
\bE\|I_1(\omega,t,\cdot,x,\cdot)\|^p_{L_p((0,t)\times \cD, ds\,dy)}. \nonumber
\end{eqnarray}

{\bf 3}. In this step we will show that there exists a constant $C=C(\cM,d,p,\theta,\Theta,\nu_1,\nu_2)>0$ such that
\begin{equation}\label{I_2}
\big\|\|I_2(t,\cdot,x,\cdot)\|_{L_{q}(\cD,dy)}\big\|_{L_{r}((0,t),ds)}\leq C \vert x\vert ^{-\mu+\alpha}\rho(x)^{-\alpha+1-2/p}.
\end{equation}
In particular, the right hand side  is independent of $t$.

{\bf Case 1.} Assume $p=2$ (hence, $q=2$ and $r=\infty$). First, we consider
\begin{align*}
&\int_{\cD}I_2^2(t,s,x,y) \,dy\\
&=K_{2,1}(t-s,x)\cdot\frac{1}{(t-s)^{d/2}}\int_{\mathcal{D}}e^{-\sigma\frac{\vert x-y\vert ^2}{t-s}}K_{2,2}(t-s,y)\vert y\vert ^{2(-\mu+\alpha)}\vert \rho(y)\vert ^{-2\alpha}\,dy.
\end{align*}
Since $2(\lambda^--\mu-\gamma_2)>-d$ and $2(1-\alpha-\beta_2)>-1$ from \eqref{inequality gamma beta3}, by change of variables $y/\sqrt{t-s}\to y$ and Lemma~\ref{lemma3.1} (ii), we have
\begin{align*}
&\frac{1}{(t-s)^{d/2}}\int_{\mathcal{D}}e^{-\sigma\frac{\vert x-y\vert ^2}{t-s}}K_{2,2}(t-s,y)\vert y\vert ^{2(-\mu+\alpha)}\vert \rho(y)\vert ^{-2\alpha}\,dy\\
&=(t-s)^{-\mu}\int_{\mathcal{D}}e^{-\sigma\vert \frac{x}{\sqrt{t-s}}-y\vert ^2}\frac{\vert y\vert ^{2(\lambda^--\mu-\gamma_2-1+\alpha+\beta_2)}}{(\vert y\vert +1)^{2(\lambda^--\gamma_2-1+\beta_2)}}\cdot\frac{\rho(y)^{2(1-\alpha-\beta_2)}}{(\rho(y)+1)^{2(1-\beta_2)}}dy\\
&\leq C\left(\vert x\vert +\sqrt{t-s}\right)^{2(-\mu+\alpha)}\left(\rho(x)+\sqrt{t-s}\right)^{-2\alpha}.
\end{align*}
Hence, we have
\begin{align*}
&\sup_{s\in[0,t]}\left(\int_{\cD}I_{2}^2\,dy\right)^{1/2}\\
&\leq C\sup_{s\in[0,t]}\left(K_{2,1}(t-s,x)\cdot\left(\vert x\vert +\sqrt{t-s}\right)^{2(-\mu+\alpha)}\left(\rho(x)+\sqrt{t-s}\right)^{-2\alpha}\right)^{1/2}\\
&= C\sup_{s\in[0,t]}\left(\frac{\vert x\vert ^{\lambda^+-1-\gamma_1+\beta_1}}{(\vert x\vert +\sqrt{t-s})^{\lambda^+-1+\mu-\gamma_1-\alpha+\beta_1}}\cdot\frac{\rho(x)^{1-\beta_1}}{(\rho(x)+\sqrt{t-s})^{\alpha+1-\beta_1}}\right)\\
&=C\,\vert x\vert ^{-\mu+\alpha}\rho(x)^{-\alpha}\sup_{s\in[0,t]}\left(R_{t-s,x}^{\lambda^+-1+\mu-\gamma_1}\left(\frac{J_{t-s,x}}{R_{t-s,x}}\right)^{\alpha+1-\beta_1}\right)\\
&\leq C \,\vert x\vert ^{-\mu+\alpha}\rho(x)^{-\alpha}
\end{align*}
due to $\lambda^+-1+\mu-\gamma_1>0$, $\alpha+1-\beta_1>0$ and 
$0\leq J_{t-s,x}\leq R_{t-s,x}\leq 1$. Thus \eqref{I_2} holds.

{\bf Case 2.} Let $p>2$.
Again, since $(\lambda^--\mu-\gamma_2)q>-d$ and $(1-\alpha-\beta_2)q>-1$, by change of variables and Lemma~\ref{lemma3.1} (ii), we observe
\begin{align*}
&\frac{1}{(t-s)^{d/2}}\int_{\mathcal{D}}e^{-\sigma\frac{\vert x-y\vert ^2}{t-s}}K_{2,2}(t-s,y)\vert y\vert ^{(-\mu+\alpha)q}\rho(y)^{-\alpha q}dy\\
=&\,(t-s)^{-\mu q/2}\int_{\mathcal{D}}e^{-\sigma\vert \frac{x}{\sqrt{t-s}}-y\vert ^2}\frac{\vert y\vert ^{(\lambda^--\mu-\gamma_2-1+\alpha+\beta_2)q}}{(\vert y\vert +1)^{(\lambda^--\gamma_2-1+\beta_2)q}}\cdot\frac{\rho(y)^{(1-\alpha-\beta_2)q}}{(\rho(y)+1)^{(1-\beta_2)q}}dy\\
\leq&\,C \left(\vert x\vert +\sqrt{t-s}\right)^{(-\mu+\alpha)q}\left(\rho(x)+\sqrt{t-s}\right)^{-\alpha q}.
\end{align*}
Hence, we have
\begin{align*}
&\int_0^t\|I_2(t,s,x,\cdot)\|_{L_{q}(\cD,dy)}^r ds\\
\leq &\,C \int^t_0\left\{K_{2,1}(t-s,x)\cdot\left(\vert x\vert +\sqrt{t-s}\right)^{(-\mu+\alpha)q}\left(\rho(x)+\sqrt{t-s}\right)^{-\alpha q}\right\}^{r/q}ds\\
=&\,C \int^t_0\frac{\vert x\vert ^{(\lambda^+-1-\gamma_1+\beta_1)r}}{(\vert x\vert +\sqrt{t-s})^{(\lambda^+-1+\mu-\gamma_1-\alpha+\beta_1)r}}\cdot\frac{\rho(x)^{(1-\beta_1)r}}{(\rho(x)+\sqrt{t-s})^{(\alpha+1-\beta_1)r}}ds\,.
\end{align*}
Moreover, since \eqref{inequality gamma beta3} also gives $\left(\lambda^++\mu-\gamma_1\right)r>2$ and $\left(\alpha+1-\beta_1\right)r>2$, 
using Lemma~\ref{lemma3.1} we again obtain
\begin{align*}
\big\|\|I_2(t,\cdot,x,\cdot)\|_{L_{q}(dy;\cD)}\big\|_{L_{r}((0,t),ds)}&=\left(\int_0^t\|I_2\|_{L_{q}(\cD,dy)}^r ds\right)^{1/r}\\
&\leq C\vert x\vert ^{-\mu+\alpha}\rho(x)^{-\alpha+1-2/p}.
\end{align*}

{\bf 4}.  Now,  by \eqref{eqn 8.7.8} and \eqref{I_2} we have
\begin{align*}
&\quad\bE\,\big\vert \vert x\vert ^{\mu-\alpha}\rho(x)^{\alpha-1}u(t,x)\big\vert ^p\\
&\leq C\big(\vert x\vert ^{\mu-\alpha}\rho(x)^{\alpha-1}\big)^p \cdot \bE\,\int_0^t\int_{\cD}I_1^p(t,s,x,y) dy\,ds\cdot \big(\vert x\vert ^{-\mu+\alpha}\rho(x)^{-\alpha+1-2/p}\big)^p\\
&= C\,\rho(x)^{-2}\;\bE\,\int_0^t\int_{\cD}I_1^p(t,s,x,y) dy\,ds.
\end{align*}
Therefore, taking integrations with respect to $x$ and $t$, using Fubini theorem and recalling \eqref{I_1^p}, we have
\begin{align}
\nonumber
\bE\,\|\rho_0^{\mu-\alpha}\rho^{\alpha-1}u\|^p_{L_p(\Omega_T\times\mathcal{D})}&\leq C \,\bE\int_0^T\int_{\mathcal{D}}\int^t_0\int_{\cD}\vert \rho(x)\vert ^{-2}I_1^p\,dyds\;dxdt\\
&=C\,\bE\,\int_0^T\int_{\mathcal{D}}I_3(s,y)\cdot \vert h(s,y)\vert _{\ell_2}^pdyds,   \label{eqn 8.11.31}
\end{align}
where
\begin{align*}
I_3(s,y):=\int^T_s\int_{\mathcal{D}}\frac{1}{(t-s)^{d/2}}e^{-\sigma\frac{\vert x-y\vert ^2}{t-s}}K_{1,1}(t,s,x,y)K_{1,2}(t,s,x,y)\rho(x)^{-2}\,dxdt.
\end{align*}
Since \eqref{inequality gamma beta3} also implies $\gamma_1p-2>-d$ and $\beta_1p-2>-1$, by change of variables $T-t\to t$ followed by $x/\sqrt{t}\to t$ and Lemma~\ref{lemma3.1} (ii),  we have
\begin{align*}
I_3(s,y)&=\int^T_s\frac{1}{(t-s)^{d/2}}K_{1,2}(t-s,y)\left(\int_{\mathcal{D}}e^{-\sigma\frac{\vert x-y\vert ^2}{t-s}}K_{1,1}(t-s,x)\rho(x)^{-2}\,dx\right)dt\\
&\leq \int^{\infty}_0\frac{1}{t}K_{1,2}(t,y)\left(\int_{\mathcal{D}}\frac{\vert x\vert ^{(\gamma_1-\beta_1)p}}{(\vert x\vert +1)^{(\gamma_1-\beta_1)p}}\frac{\rho(x)^{\beta_1p-2}}{(\rho(x)+1)^{\beta_1p}}e^{-\sigma'\vert x-\frac{y}{\sqrt{t}}\vert ^2}\,dx\right)dt\\
&\leq C\int_0^{\infty}K_{1,2}(t,y)\left(\rho(y)+\sqrt{t}\right)^{-2}dt\\
&= C\int_0^{\infty}\frac{\vert y\vert ^{(\gamma_2-\beta_2)p}}{\left(\vert y\vert +\sqrt{t}\right)^{(\gamma_2-\beta_2)p}}\cdot\frac{\rho(y)^{\beta_2p}}{\left(\rho(y)+\sqrt{t}\right)^{\beta_2p+2}}\,dt.
\end{align*}
Hence, by Lemma~\ref{lemma3.1} (i) with the conditions $\gamma_2p>0$ and $\beta_2p> 0$, we finally get 
\begin{align*}
I_3(s,y)\leq C(\cM,d,p,\theta,\Theta, \nu_1,\nu_2).
\end{align*}
This and \eqref{eqn 8.11.31} lead to  \eqref{main inequality 2} and the lemma is proved.

{\bf 5}.  Again, the last part of the claim related to the range of $\theta$ holds by the same reason explained in Step 5 of the proof of Lemma \ref{main est3}.
\end{proof}

\mysection{Proof of Theorems \ref{main result} and \ref{main result-random}}\label{sec:main proofs}
In this section we prove Theorems \ref{main result} and \ref{main result-random}, following the strategy below: 

\vspace{2mm}

{\bf 1}. \emph{ A priori estimate and the uniqueness}:

\begin{itemize}

\item[-] 
  In Lemma \ref{lemma entire} below, we first prove that for any   solution $u\in \cK^{\gamma+2}_{p,\theta,\Theta}(\cD,T)$ to equation \eqref{stochastic parabolic equation} equipped with the general operator $\cL=\sum_{i,j=1}^d a^{ij}(\omega,t)D_{ij}$, we have
\begin{eqnarray}
& \|u\|_{\cK^{\gamma+2}_{p,\theta,\Theta}(\cD,T)}
\leq C \big( \|u\|_{\bL_{p,\theta-p,\Theta-p}(\cD,T)} +\text{norms of the free terms} \big). \label{eqn 8.18.21}
\end{eqnarray}

\item[-]    If  $\cL$ is non-random, we   estimate $\|u\|_{\bL_{p,\theta-p,\Theta-p}(\cD,T)}$ based on  Lemma \ref{main est}. 

\item[-] To treat the SPDE with random coefficients, we introduce a SPDE having non-random coefficients and the same free terms $f^0$, $\tbf$, $g$, $u_0$. Then we prove a priori estimate for the original SPDE based on the fact that the difference between the new SPDE and the original SPDE becomes a PDE (with random coefficients).  

\item[-]  The uniqueness of solution to the original SPDE follows from the uniqueness result of PDEs. 

\end{itemize}

\vspace{2mm}

{\bf 2}. \emph{The existence}:

\begin{itemize}

\item[-]  If the coefficients of $\cL$ are non-random, we use the representation formula.

\item[-] For general case, we use the method of continuity with the help of the a priori estimate.

\end{itemize}

\vspace{3mm}
Now we start our proofs. The following lemma is what we meant in \eqref{eqn 8.18.21}. 
We emphasize that the lemma holds for any $\theta,\Theta\in \bR$ and the condition $\partial \cM\in C^2$ is not  needed in the proof.

\begin{lemma}\label{regularity.induction}
Let $p\in [2,\infty)$,    $\gamma, \mu, \theta, \Theta \in\bR$,  $\mu<\gamma$, and the diffusion coefficients $a^{ij}=a^{ij}(\omega,t)$ satisfy   Assumption  \ref{ass coeff}.  Assume that   $f^0\in\bK^{\gamma}_{p,\theta+p,\Theta+p}(\cD,T)$, $\tbf=(f^1,\cdots,f^d) \in\bK^{\gamma+1}_{p,\theta,\Theta}(\cD,T,d)$,  $g\in\bK^{\gamma+1}_{p,\theta,\Theta}(\cD,T, \ell_2)$, $u(0,\cdot)\in \bU^{\gamma+2}_{p,\theta,\Theta}(\cD)$, and  $u\in\bK^{\mu+2}_{p,\theta-p,\Theta-p}(\cD,T)$ satisfies 
\begin{align}
du=(\cL u+f^0+\sum_{i=1}^d f^i_{x^i})\,dt+\sum_{k=1}^{\infty} g^kdw^k_t,\quad t\in(0,T]   \label{eqn 8.19.1}
\end{align}
 in the sense of distributions on $\cD$.  Then  $u\in\bK^{\gamma+2}_{p,\theta-p,\Theta-p}(\cD,T)$, hence $ u\in\cK^{\gamma+2}_{p,\theta,\Theta}(\cD,T)$, and  the estimate
\begin{align*}
\|u\|_{\bK^{\gamma+2}_{p,\theta-p,\Theta-p}(\cD,T)}
&\leq  C\Big(\|u\|_{\bK^{\mu+2}_{p,\theta-p,\Theta-p}(\cD,T)}+\|f^0\|_{\bK^{\gamma}_{p,\theta+p,\Theta+p}(\cD,T)}\nonumber\\
&\quad \quad \quad +\|\tbf\|_{\bK^{\gamma+1}_{p,\theta,\Theta}(\cD,T,d)}+\|g\|_{\bK^{\gamma+1}_{p,\theta,\Theta}(\cD,T,\ell_2)}+\|u(0,\cdot)\|_{\bU^{\gamma+2}_{p,\theta,\Theta}(\cD)} \Big) 
\end{align*}
holds with $C=C(\cM,p,n,\theta,\Theta,\nu_1,\nu_2)$.
\end{lemma}

The proof of Lemma \ref{regularity.induction} is based on the following result on $\bR^d$.

\begin{lemma}
 \label{lemma entire}
 Let $p\in[2,\infty)$,  $\gamma \in \bR$, and  Assumption  \ref{ass coeff} hold. Assume  $f\in \bH^{\gamma}_p(T)$,  $g\in \bH^{\gamma+1}_p(T,\ell_2)$, $u(0,\cdot)\in L_p(\Omega;H^{\gamma+2-2/p}_p)$, and $u\in \bH^{\gamma+1}_p(T)$ satisfies
\begin{align}
du=(\cL u+f)\,dt+\sum^{\infty}_{k=1}g^kdw^k_t,\quad t\in(0,T]\nonumber
\end{align}
 in the sense of distributions on the whole space $\bR^d$.  Then  $u\in \bH^{\gamma+2}_p(T)$ and 
 \begin{eqnarray}
 \|u\|_{\bH^{\gamma+2}_p(T)} &\leq& C\Big(\|u\|_{\bH^{\gamma+1}_p(T)}\nonumber\\
 &&\quad+\|f\|_{\bH^{\gamma}_p(T)}
+\|g\|_{\bH^{\gamma+1}_p(T,\ell_2)}+\|u(0,\cdot)\|_{L_p(\Omega;H^{\gamma+2-2/p}_p)}\Big),\label{whole space estimate}
 \end{eqnarray}
 where $C=C(d,p,\nu_1,\nu_2)$ is independent of $T$.
 \end{lemma} 
 
 \begin{proof}
{\bf 1}. First, we consider the case $u(0,\cdot) \equiv 0$. Then, by e.g. \cite[Theorem 4.10]{Krylov 1999-4},  $u\in \bH^{\gamma+2}_p(T)$ and
 $$
 \|u_{xx}\|_{\bH^{\gamma}_p(T)}\leq C(d,p,\nu_1,\nu_2) (\|f\|_{\bH^{\gamma}_p(T)}+\|g\|_{\bH^{\gamma+1}_p(T,\ell_2)}).
 $$
 This and the inequality
 $$
 \|u\|_{\bH^{\gamma+2}_p(T)}\leq  (\|u_{xx}\|_{\bH^{\gamma}_p(T)}+\|u\|_{\bH^{\gamma}_p(T)} )
 $$
 together with  the  inequality $\|u\|_{\bH^{\gamma}_p(T)} \le \|u\|_{\bH^{\gamma+1}_p(T)} $, which due to a basic property of the space of Bessel potentials,
 yield the claim of the lemma. 
 
 {\bf 2}. For the case of general $u(0,\cdot)\not\equiv 0$, we use the solution $v=v(\omega,t,x)$ to the equation
 $$
 dv=\cL v \,dt, \quad t\in(0,T]
 $$
 with $v(\omega,0,\cdot)=u(\omega,0,\cdot)$ for all $\omega\in\Omega$ (see \cite[Theorem 5.2]{Krylov 1999-4}).
 From a classical  theory of PDE, which we apply for each $\omega$, we have
 $$
  \|v\|_{\bH^{\gamma+2}_p(T)}\leq C  \|u_0\|_{L_p(\Omega;H^{\gamma+2-2/p}_p)}.
 $$
 Then for the function $u-v$, which has zero initial condition, we can apply Step 1 and  we obtain estimate \eqref{whole space estimate} for $u$ simply by triangle inequality.
  \end{proof}

\begin{proof}[\textbf{Proof of Lemma \ref{regularity.induction}}]
We first note that we only need to consider the case $\mu=\gamma-1$.  Indeed, suppose that the lemma holds true if $\mu=\gamma-1$. Now let $\mu=\gamma-n$, $n\in \bN$. 
Then  applying the result  for $\mu'=\gamma-k$ and $\gamma'=\mu'+1$ with $k=n,n-1,\cdots, 1$ in order, we get the claim when $\mu=\gamma-n$. 
Now suppose that the difference between $\gamma$ and $\mu$ is not an integer, i.e.  $\gamma-\mu=n+\delta$, $n=0,1,2,\cdots$ and $\delta\in (0,1)$.  Then, since $\mu>\gamma-(n+1)=:\mu'$ and $\|\cdot \|_{\bK^{\mu'+2}_{p,\theta-p,\Theta-p}(\cD,T)}\le \|\cdot \|_{\bK^{\mu+2}_{p,\theta-p,\Theta-p}(\cD,T)}$, we conclude that our assumption holds for $\mu'$, that is,  $u\in \bK^{\mu'+2}_{p,\theta-p,\Theta-p}(\cD,T)$.  Therefore, the case $\gamma-\mu \not\in \bN$  is also covered by what we just discussed.

\vspace{1mm}

 Now we prove the lemma when $\mu=\gamma-1$, i.e. $u\in\bK^{\gamma+1}_{p,\theta-p,\Theta-p}(\cD,T)$.
 As usual,  we omit the argument $\omega$ for the simplicity of presentation. 
 
{\bf 1}. For $u\in\bK^{\gamma+1}_{p,\theta-p,\Theta-p}(\cD,T)$, put 
$$\xi(x)=\vert x\vert^{(\theta-\Theta)/p},\quad  v:=\xi u, \quad  f:=f^0+\sum_{i=1}^{d}f^i_{x^i},\quad v_0:=\xi u_0.
$$
Using  Definition  \ref{defn 8.19}, Definition \ref{defn 8.28}, and the change of variables $t\to e^{2n}t$, we  have
\begin{eqnarray}
 &&\|u\|^p_{\bK^{\gamma+2}_{p,\Theta-p,\Theta-p}(\cD,T)}=\|v\|^p_{\bH^{\gamma+2}_{p,\Theta-p}(\cD,T)} \nonumber\\
&&= \sum_{n\in \bZ} e^{n(\Theta-p)}\|\zeta(e^{-n}\psi(e^n\cdot))v(\cdot,e^n\cdot)\|^p_{\bH^{\gamma+2}_p(T)} \nonumber\\
  &&= \sum_{n\in \bZ} e^{n(\Theta-p+2)}\|\zeta(e^{-n}\psi(e^n\cdot))v(e^{2n}\cdot,e^n\cdot)\|^p_{\bH^{\gamma+2}_p(e^{-2n}T)}.  \label{eqn 4.23.1}  
   \end{eqnarray}

For each $n\in \bZ$, we denote 
$$
v_n(t,x):=\zeta(e^{-n}\psi(e^nx))v(e^{2n}t,e^nx),
\quad
 v_{0,n}(x)=\zeta(e^{-n}\psi(e^nx))v_0(e^nx).
 $$
Then using equation \eqref{eqn 8.19.1} and the product rule of differentiation, one can easily check that $v_n$ satisfies 
 $$
 d v_n=(\cL_n v_n +f_n)dt+ \sum_{k=1}^{\infty} g^k_n dw^{n,k}_t \quad  t\in(0,e^{-2n}T]
 $$
 in the sense of distributions on $\bR^d$ with the initial condition $v_{n}(0,\cdot)=v_{0,n}(\cdot)$,
 where
 $$
 \cL_n:=\sum_{i,j} a^{ij}_n(t)D_{ij},\quad a^{ij}_n(t):=a^{ij}(e^{2n}t), 
 $$
 $$
 g^k_n(t,x):=e^{n} \zeta(e^{-n}\psi(e^nx))\xi(e^nx) g^k(e^{2n}t,e^nx), \qquad w^{n,k}_t:=e^{-n}w^k_{e^{2n}t},
 $$
 and, with Einstein's summation convention with respect to $i,j$,
 \begin{eqnarray*}
 f_n(t,x)&:=&\quad e^{2n}\zeta(e^{-n}\psi(e^nx))\xi(e^nx) f(e^{2n}t,e^nx) \\
 &&+e^{n} a^{ij}_n(t)D_iu(e^{2n}t,e^nx) \zeta'(e^{-n}\psi(e^nx)) D_j\psi(e^nx) \xi(e^nx) \\
 &&+e^{2n}a^{ij}_n(t)D_iu(e^{2n}t,e^nx) \zeta(e^{-n}\psi(e^nx)) D_j \xi(e^nx) \\
 &&+e^na^{ij}_n(t)u(e^{2n}t,e^nx)\zeta'(e^{-n}\psi(e^nx))  D_i\psi(e^nx)D_j\xi(e^nx)\\
 &&+e^{2n}a^{ij}_n(t)u(e^{2n}t,e^nx)\zeta(e^{-n}\psi(e^nx))D_{ij}\xi(e^nx) \\
 &&+a^{ij}_n(t)u(e^{2n}t,e^nx)\zeta''(e^{-n}\psi(e^nx))D_i\psi(e^nx) D_j\psi(e^nx) \xi(e^nx)\\
 &&+ e^na^{ij}_n(t)u(e^{2n}t,e^nx)\zeta'(e^{-n}\psi(e^nx))  D_{ij}\psi(e^nx)\\
 &=:& \sum_{l=1}^7 f^{l}_n(t,x).
\end{eqnarray*}
Here, $\zeta'$ and $\zeta''$ denote the first and second derivative of $\zeta$, respectively.  We note that for each $n\in \bZ$, the operator $\cL_n$  still satisfies  the uniform parabolicity condition \eqref{uniform parabolicity} and $\{w^{n,k}_t: k\in \bN\}$ is a sequence of independent Brownian motions. Hence, 
 we can apply Lemma \ref{lemma entire} and from  \eqref{eqn 4.23.1} we get
  \begin{eqnarray}
  \nonumber
  \|u\|^p_{\bK^{\gamma+2}_{p,\theta-p,\Theta-p}(\cD,T)} 
  &\leq& C   \sum_{n\in \bZ} e^{n(\Theta-p+2)}\|v_n\|^p_{\bH^{\gamma+1}_p(e^{-2n}T)}\\
  &&+C \sum_{l=1}^7  \sum_{n\in \bZ} e^{n(\Theta-p+2)}\|f^l_n\|^p_{\bH^{\gamma}_p(e^{-2n}T)} \nonumber\\
  &&+C\sum_{n\in \bZ} e^{n(\Theta-p+2)}\|g_n\|^p_{\bH^{\gamma+1}_p(e^{-2n}T, \ell_2)}\nonumber\\
  &&+ C  \sum_{n\in \bZ} e^{n(\Theta-p+2)}\|v_{0,n}\|^p_{L_p(\Omega;H^{\gamma+2-2/p}_p)}\label{eqn 8.19.21}
  \end{eqnarray}
provided that 
\begin{equation}\label{in fact true}
 v_n\in \bH^{\gamma+1}_p(e^{-2n}T), \quad  f^l_n \in \bH^{\gamma}_p(e^{-2n}T),\quad
 g_n \in \bH^{\gamma+1}_p(e^{-2n}T,\ell_2), \quad (l=1,\ldots,7).
 \end{equation}
It turns out that the claims in \eqref{in fact true} hold true. 
Indeed, the change of variable $e^{2n}t \to t$ and Definition \ref{defn 8.19} yield
	\begin{eqnarray}
&&   \sum_{k\in \bZ} e^{n(\Theta-p+2)}\|v_n\|^p_{\bH^{\gamma+1}_p(e^{-2n}T)} \nonumber\\
&&= \sum_{n\in \bZ} e^{n(\Theta-p)}\| \zeta(e^n\psi(e^n\cdot))v(\cdot,e^n\cdot) \|^p_{\bH^{\gamma+1}_p(T)}
=  \|u\|^p_{\bK^{\gamma+1}_{p,\theta-p,\Theta-p}(\cD,T)} \label{eqn 8.19.31}
	\end{eqnarray}
and
	\begin{eqnarray}
&&   \sum_{n\in \bZ} e^{n(\Theta-p+2)}\|g_n\|^p_{\bH^{\gamma+1}_p(e^{-2n}T,\ell_2)} \nonumber \\
&&= \sum_{n\in \bZ} e^{n\Theta}\|\zeta(e^n\psi(e^n\cdot))\xi(e^n\cdot)g(\cdot,e^n\cdot) \|^p_{\bH^{\gamma+1}_p(T,\ell_2)}=  \|g\|^p_{\bK^{\gamma+1}_{p,\theta,\Theta}(\cD,T,\ell_2)}. \label{eqn 8.19.41}
   \end{eqnarray}
In particular,     
$$
 v_n \in \bH^{\gamma+1}_p(e^{-2n}T), \quad
 g_n \in \bH^{\gamma+1}_p(e^{-2n}T, \ell_2), \quad \forall\, n\in \bZ.
 $$ 
   Next, we show that $f^l_n$ belong to $\bH^{\gamma}_p(e^{-2n}T)$ in the following manner.
For $l=1$,  by Definition \ref{defn 8.19} and the change of variables $e^{2n}t \to t$, we have
   \begin{eqnarray*}
 &&   \quad\sum_{n\in \bZ} e^{n(\Theta-p+2)}\|f^1_n\|^p_{\bH^{\gamma}_p(e^{-2n}T)}\\
 && =  
    \sum_{n\in \bZ} e^{n(\Theta+p)}\|\zeta(e^{-n}\psi(e^n\cdot))\xi(e^n\cdot)f(e^{2n}\cdot,e^n\cdot) \|^p_{\bH^{\gamma}_p(T)} 
    = \|f\|^p_{\bK^{\gamma}_{p,\theta+p,\Theta+p}(\cD,T)}.
    \end{eqnarray*}
For $l=2$,  by Definition \ref{defn 8.19} and  \eqref{eqn 4.24.5},
we get
\begin{eqnarray*}
 &&  \quad\sum_{n\in \bZ} e^{n(\Theta-p+2)}\|f^2_n\|^p_{\bH^{\gamma}_p(e^{-2n}T)}\\
 &&\leq C 
    \sum_{n\in \bZ} \sum_{i,j}e^{n\Theta}\|D_iu(\cdot,e^n\cdot)\zeta'(e^{-n}\psi(e^n\cdot))\xi(e^n \cdot) D_j\psi(e^n\cdot)\|^p_{\bH^{\gamma}_p(T)} \\
    &&\leq C \|\psi_x \xi u_x  \|^p_{\bH^{\gamma}_{p,\Theta}(\cD,T)}=N\|\psi_x  u_x  \|^p_{\bK^{\gamma}_{p,\theta,\Theta}(\cD,T)} \\
    &&\leq 
    C \|u_x\|^p_{\bK^{\gamma}_{p,\theta, \Theta}(\cD,T)}\leq 
    C \|u\|^p_{\bK^{\gamma+1}_{p,\theta-p, \Theta-p}(\cD,T)},
   \end{eqnarray*}
where the   last two inequalities are  due to    \eqref{eqn 8.9.1}, \eqref{eqn 8.19.11}, and \eqref{eqn 4.16.1}.  For $l=3$, by definitions of norms, we have
      \begin{eqnarray}
  && \sum_{n\in \bZ} e^{n(\Theta-p+2)}\|f^2_n\|^p_{\bH^{\gamma}_p(e^{-2n}T)}  \nonumber \\
  & \leq& C  \sum_{n\in \bZ} \sum_{i,j}e^{n(\Theta+p)}\|D_iu(\cdot,e^n\cdot) \zeta(e^{-n}\psi(e^n\cdot)) D_j\xi(e^n\cdot)\|^p_{\bH^{\gamma}_p(T)}  \nonumber  \\
  &=&C
   \|u_x \xi_x\|^p_{\bH^{\gamma}_{p,\Theta+p}(\cD,T)}
   = C 
    \|\xi \xi^{-1}\xi_xu_x\|^p_{\bH^{\gamma}_{p,\Theta+p}(\cD,T)} \nonumber \\
    &=&C 
    \|\xi^{-1}\xi_x u_x \|^p_{\bK^{\gamma}_{p,\theta+p,\Theta+p}(\cD,T)} \leq C\|\psi \xi^{-1}\xi_x u_x\|^p_{\bK^{\gamma}_{p,\theta,\Theta}(\cD,T)},  \label{eqn 8.20.1}
   \end{eqnarray}
   where the  last  inequality is due to \eqref{eqn 8.19.81}.  Now we note that for any $n\in \bN$,
$$
\vert\psi \xi^{-1}\xi_x\vert ^{(0)}_n+\vert \psi^2 \xi^{-1}\xi_{xx}\vert ^{(0)}_n\leq C(n,\xi)<\infty.
$$   
  Thus, by \eqref{eqn 8.19.11} the last term in \eqref{eqn 8.20.1} is bounded by
  $$
  C\|u_x\|^p_{\bK^{\gamma}_{p,\theta,\Theta}(\cD,T)}  \leq C \|u\|^p_{\bK^{\gamma+1}_{p,\theta-p,\Theta-p}(\cD,T)}.
  $$
  
For other $l$s one can argue similarly and we gather the results:
  \begin{eqnarray}
&&  \sum_{l=1}^7  \sum_{n\in \bZ} e^{n(\Theta-p+2)}\|f^l_n\|^p_{\bH^{\gamma}_p(e^{-2n}T)} \nonumber \\
&\leq& C \|u\|^p_{\bK^{\gamma+1}_{p,\theta-p,\Theta-p}(\cD,T)}+
 C\|f\|^p_{\bK^{\gamma}_{p,\theta+p,\Theta+p}(\cD,T)}.  \label{eqn 8.19.51}
  \end{eqnarray}
   
Consequently, coming back to \eqref{eqn 8.19.21} and using \eqref{eqn 8.19.31}, \eqref{eqn 8.19.41}, and  \eqref{eqn 8.19.51}, we get
\begin{align*}
&\quad\quad\quad\|u\|^p_{\bH^{\gamma+2}_{p,\theta-p,\Theta-p}(\cD,T)}\\
& \leq C\Big(\|u\|^p_{\bK^{\gamma+1}_{p,\theta-p,\Theta-p}(\cD,T)}
+\|f\|^p_{\bK^{\gamma}_{p,\theta+p,\Theta+p}(\cD,T)}
+\|g\|^p_{\bH^{\gamma+1}_{p,\theta,\Theta}(\cD,T,\ell_2)}
+\|u_0\|^p_{\bU^{\gamma+2}_{p,\theta,\Theta}(\cD)} \Big).
\end{align*}
This yields what we want to have since $\|f^i_{x^i}\|_{K^{\gamma}_{p,\theta+p,\Theta+p}(\cD)}\leq C \|f^i\|_{K^{\gamma+1}_{p,\theta,\Theta}(\cD)}$. The lemma is proved.
\end{proof}

Now, we take  the deterministic operator $L_0$ introduced in  \eqref{8.29.1} and the Green function $G$ related to $L_0$.  Also, recall the representation $\cR (u_0,f^0,\tbf,g)$ defined in \eqref{eqn 8.21.11} in connection with $L_0$.

\begin{lemma}
\label{lemma rep}
 If  $f^0\in \bK^{\infty}_c(\cD,T)$, $\tbf\in \bK^{\infty}_c(\cD,T,d)$, $g\in \bK^{\infty}_c(\cD,T,\ell_2)$, and $u_0\in\bK^{\infty}_c(\cD)$, then  $u=\cR (u_0,f^0,\tbf,g)$  belongs to $\cK^0_{p,\theta,\Theta}(\cD,T)$ and satisfies
\begin{equation}
   \label{eqn 8.21.13}
du=\left(L_0u+f^0+\sum_{i=1}^d f^i_{x^i}\right)dt+\sum_{k=1}^{\infty}g^k dw^k_t, \quad t\in (0,T]
\end{equation}
in the sense of distributions on $\cD$ with $u(0,\cdot)=u_0$.
 \end{lemma}
\begin{proof}  First, we note that 
\begin{eqnarray*}\cR (u_0,f^0,\tbf,g)&=&\cR (u_0,0,0,0)+\cR (0,f^0,\tbf,0)+\cR (0,0,0,g)\\
&=:&v_1+v_2+v_3.
\end{eqnarray*}

By considering $v_1$ for each $\omega$ and by the definition of Green's function with the condition $u_0\in\bK^{\infty}_c(\cD)$,  we note that $v_1$ satisfies 
$$
dv_1=L_0v_1dt,\quad t>0\,; \quad v_1(0,\cdot)=u_0(\cdot)
$$ 
in the sense of distributions on $\cD$. Then Lemma \ref{main est3} and the facts that $\bK^{\infty}_c(\cD)$ is dense in $L_p(\Omega;K^0_{p,\theta+2-p,\Theta+2-p}(\cD))$ and $\|u_0\|_{\bU^0_{p,\theta,\Theta}(\cD)}\le \|u_0\|_{L_p(\Omega;K^0_{p,\theta+2-p,\Theta+2-p}(\cD))}$ confirm $v_1\in \cK^0_{p,\theta,\Theta}(\cD,T)$. Similarly, $v_2$ satisfies
$$
dv_2=(L_0v_2+f^0+\sum_{i=1}^d f^i_{x^i})dt,\quad t>0
$$
in the sense of distributions on $\cD$ with zero initial condition and Lemma \ref{main est1} leads us to have $v_2\in \cK^0_{p,\theta,\Theta}(\cD,T)$.
 The fact that $v_3$ satisfies
 $$
 dv_3=L_0v_3dt+\sum_{k=1}^{\infty}g^k dw^k_t, \quad t>0
 $$
 in the sense of distributions on $\cD$ with zero initial condition can be proved by the same way in the proof of \cite[Lemma 3.11]{CKLL 2018}, which deals with the case $d=2$.   Then  Lemma \ref{main est2} gives $v_2\in \cK^0_{p,\theta,\Theta}(\cD,T)$.
 Hence,  $u=v_1+v_2+v_3$ satisfies the assertions and the lemma is proved.
\end{proof}

\vspace{0.3cm}
\begin{proof}[\textbf{Proof of Theorem  \ref{main result}}]\quad

Note  that,  since $\cL$ is non-random, we can take $L_0=\cL$ (see \eqref{8.29.1}).

\vspace{1mm}

{\bf 1}. \emph{ Existence and    estimate \eqref{main estimate}} :

First, we assume that  $f^0\in \bK^{\infty}_c(\cD,T)$, $\tbf\in \bK^{\infty}_c(\cD,T,d)$, $g\in \bK^{\infty}_c(\cD,T,\ell_2)$, and $u_0\in\bK^{\infty}_c(\cD)$.  Then by Lemma \ref{lemma rep}, $u=\cR(u_0,f^0,\tbf,g) \in \cK^0_{p,\theta,\Theta}(\cD,T)$ satisfies equation \eqref{eqn 8.21.13}  in the sense of distributions on $\cD$ with initial condition $u_0$. Then, we use   Lemma \ref{regularity.induction} with $\mu=-2$.  As $\gamma+2\ge 1$,  Lemma \ref{main est} and Remark \ref{initial.p ge 2.}  imply
 $u\in \cK^{\gamma+2}_{p,\theta,\Theta}(\cD,T)$ and  \eqref{main estimate}.

The general case can be easily  handled by standard approximation argument. Indeed, take $f^0_n\in \bK^{\infty}_c(\cD,T)$, $\tbf_n\in \bK^{\infty}_c(\cD,T,d)$, $g_n\in \bK^{\infty}_c(\cD,T,\ell_2)$, and $u_{0,n}\in\bK^{\infty}_c(\cD)$
such that $f^0_n \to f^0$, $\tbf_n \to \tbf$, $g_n \to g$, and $u_{0,n}\to u_{0}$,   as $n\to \infty$,  in the corresponding spaces.  Now let $u_n:=\cR(u_0,f^0_n, \tbf_n,g_n)$.  Then, estimate  \eqref{main estimate} applied for $u_n-u_m$ shows that $\{u_n\}$ is a Cauchy sequence in $\cK^{\gamma+2}_{p,\theta,\Theta}(\cD,T)$. Taking $u$ as the limit of $u_n$ in $\cK^{\gamma+2}_{p,\theta,\Theta}(\cD,T)$, we find that $u$ is a solution to equation \eqref{eqn 8.21.13}.  Estimate \eqref{main estimate}  for $u$ also follows from those of $u_n$.

\vspace{1mm}
  
{\bf 2}. \emph{Uniqueness} :

Let  $u \in \cK^{\gamma+2}_{p,\theta,\Theta}(\cD,T)$ be a solution to  equation \eqref{eqn 8.21.13} with $f^0\equiv 0$, $\tbf \equiv 0$, $g\equiv 0$, and $u_0\equiv 0$.  Due to $\gamma+2\geq 1$,  $u$ at least belongs to $\bL_{p,\theta-p,\Theta-p}(\cD,T)$, and therefore by Lemma \ref{regularity.induction} we have $u\in \cK^2_{p,\theta,\Theta}(\cD,T)$ as all the inputs are zeros. Hence, for almost all $\omega\in \Omega$, $u^{\omega}:=u(\omega,\cdot,\cdot)\in L_p((0,T]; K^{2}_{p,\theta-p,\Theta-p}(\cD))$, and satisfies
$$
u^{\omega}_t=\cL u^{\omega}, \quad t\in (0,T]\quad ; \quad u^{\omega}(0,\cdot)=0.
$$
Hence, from  the uniqueness result for the deterministic parabolic equation  (see \cite[Theorem 2.12]{ConicPDE}),  we conclude $u^{\omega}=0$ for almost all $\omega$. This handles the uniqueness.
\end{proof}

\begin{remark}\label{sol.representation}
The approximation argument and uniqueness result in the above proof  show that if $\cL$ is non-random, then the  solution in Theorem \ref{main result} is given by the formula 
$$
u=\cR(u_0,f^0,\tbf,g), \quad \text{where}\quad  \tbf=(f^1,\cdots,f^d).
$$
\end{remark}

\begin{proof}[\textbf{Proof of Theorem  \ref{main result-random}}]\quad
\vspace{1mm}

{\bf 1}. \emph{The a priori estimate} :

Having the method of continuity in mind, we consider the following operators.
Denote $L_0=\nu_1 \Delta$, and  for $\lambda \in [0,1]$ denote
\begin{eqnarray*}
\cL_{\lambda}=(1-\lambda)L_0+\lambda \cL
\end{eqnarray*}
Obviously,
\begin{equation*}
\cL_{\lambda}(\omega,\cdot)\in \cT_{\nu_1,\nu_2}, \quad \forall \, \lambda\in [0,1], \, \omega\in \Omega.
\end{equation*}

 Now we prove that  the a priori estimate 
\begin{eqnarray}
 \|v\|_{\cK^{\gamma+2}_{p,\theta,\Theta}(\cD,T)}&\leq& C\Big(\|f^0\|_{\bK^{\gamma \vee 0}_{p,\theta+p,\Theta+p}(\cD,T)}+ \|\tbf\|_{\bK^{\gamma+1}_{p,\theta,\Theta}(\cD,T,d)} \nonumber \\
&&\quad\quad +\|g\|_{\bK^{\gamma+1}_{p,\theta,\Theta}(\cD,T,l_2)}+\|u_0\|_{\bU^{\gamma+2}_{p,\theta,\Theta}(\cD)}\Big)   \label{the a priori}
\end{eqnarray}
 holds with $C=C(\cM,d,p,\gamma,\theta,\Theta,\nu_1,\nu_2)$,  provided that 
  $v\in \cK^{\gamma+2}_{p,\theta,\Theta}(\cD,T)$ is a solution to the equation
 \begin{equation}
  \label{method}
 dv=\left(\cL_{\lambda} v+f^0+\sum_{i=1}^d f^i_{x^i}\right)dt+\sum_{k=1}^{\infty}g^k dw^k_t, \quad t\in(0,T]\,\,\,; \quad v(0,\cdot)=u_0(\cdot).
 \end{equation}

To prove \eqref{the a priori}, we take  $u\in  \cK^{\gamma+2}_{p,\theta,\Theta}(\cD,T)$ from Theorem \ref{main result}, which is the solution to equation \eqref{eqn 8.21.13}  with the operator $L_0=\nu_1 \Delta$ and the initial condition $u(0,\cdot)=u_0$.  Then $\bar{v}:=v-u\in \cK^{\gamma+2}_{p,\theta,\Theta}(\cD,T)$ satisfies
 $$
 \bar{v}_t=\cL_{\lambda} \bar{v}+\bar{f}=\cL_{\lambda}\bar{v}+\sum_{i=1}^d\bar{f}^i_{x^i}, \quad t\in(0,T]\quad;\quad \bar{v}(0,\cdot)=0
 $$
 where 
 $$\bar{f}:=(L_0-\cL_{\lambda})u=\sum_{i=1}^d \left(\sum_{j=1}^d [\nu_1 \delta^{ij}-a^{ij}(\omega,t)]u_{x^j}\right)_{x^i}=:\sum_{i=1}^d \bar{f}^i_{x^i}.
 $$
 Note that for each fixed $\omega$, $\bar{v}(\omega,\cdot)$ satisfies a deterministic PDE with non-random operator $\cL_{\lambda}(\omega,\cdot)$ and non-random free terms $\bar{f}^i(\omega,\cdot)$.  Hence, using the deterministic counterpart of Theorem \ref{main result} for each $\omega$, and then taking the expectation, we get
 $$
 \|v-u\|_{\cK^{\gamma+2}_{p,\theta,\Theta}(\cD,T)}= \|\bar{v}\|_{\cK^{\gamma+2}_{p,\theta,\Theta}(\cD,T)} \leq C \sum_{i=1}^d \|\bar{f}^i\|_{\bK^{\gamma+1}_{p,\theta,\Theta}(\cD,T)}\leq C\|u\|_{\bK^{\gamma+2}_{p,\theta-p,\Theta-p}(\cD,T)}.
 $$
For the last inequality above we used \eqref{eqn 4.16.1}. This with estimate \eqref{main estimate} obtained for $u$ finally gives \eqref{the a priori}. 
 
   \vspace{2mm}

{\bf 2}. \emph{Existence, uniqueness and the estimate} :

 Estimate \eqref{main estimate} and  uniqueness result of solution are direct consequences of   a priori estimate \eqref{the a priori}, for which  the constant $C$ is independent of $\cL$ and $\lambda$.  Thus we only need to prove the existence result.
 
Let $J$ denote the set of $\lambda\in [0,1]$ such that   for any given  $f^0,\tbf, g,u_0$ in their corresponding spaces,   equation \eqref{method} with given $\lambda$ has a solution $v$ in $\cK^{\gamma+2}_{p,\theta,\Theta}(\cD,T)$. Then by Theorem  \ref{main result}, $0\in J$.  Hence,  the method of continuity 
(see e.g. proof of  \cite[Theorem 5.1]{Krylov 1999-4}) and  a priori estimate \eqref{the a priori} together yield  $J=[0,1]$, and in particular $1\in J$. This proves the existence result.  The theorem is proved. 
\end{proof}

In the next section, we use the result of Theorem \ref{main result-random} to study the regularity of SPDEs on  polygonal domains in $\bR^2$.  We  also use the following result which helps us prove the existence of a solution on   polygonal domains.

\begin{lemma}\label{global.uniqueness}
For $j=1,\,2$, let $p_j\geq 2$ and $\theta_j,\Theta_j\in\bR$, and $d-1<\Theta_j<d-1+p_j$. Also let $\theta_j$ ($j=1,2$)  satisfy
\begin{equation*}
p_j(1-\lambda^+_c)<\theta_j<p_j(d-1+\lambda^-_c) \quad \text{if $\cL$ is non-random},
\end{equation*}
and
$$
p_j(1-\lambda_{c}(\nu_1,\nu_2))<\theta_j<p_j(d-1+\lambda_{c}(\nu_1,\nu_2))  \quad \text{if $\cL$ is random}.
$$
 Then, if $u\in\cK^1_{p_1,\theta_1,\Theta_1}(\cD,T)$ is a solution to equation \eqref{stochastic parabolic equation} with the initial condition $u(0,\cdot)=u_0(\cdot)$ and $f^0, \tbf=(f^1,\cdots,f^d)$, $g$, $u_0$ satisfying
\begin{align*}
&f^0\in\bL_{p_j,\theta_j+p_j,\Theta_j+p_j}(\cD,T), \quad \tbf \in\bL_{p_j,\theta_j,\Theta_j}(\cD,T,d ),  
\end{align*}
$$
g\in \bL_{p_j,\theta_j,\Theta_j}(\cD,T,\ell_2), \quad u_0\in\bU^1_{p,\theta,\Theta}(\cD)
$$
for both  $j=1$ and $j=2$,  then  $u\in\cK^1_{p_2,\theta_2,\Theta_2}(\cD,T)$.
\end{lemma}

\begin{proof}
If $\cL$ is non-random, the lemma follows from  Remark \ref{sol.representation}.  In general, as before we fix  a deterministic operator  $L_0(t)=\sum_{i,j}\alpha^{ij}(t) \in \cT_{\nu_1,\nu_2}$ and  $v=\cR(u_0,f^0,\tbf,g)$.  Then, since $L_0$ is non-random, by Remark \ref{sol.representation}
\begin{equation}
 \label{eqn 8.31.5}
v\in \cK^1_{p_1,\theta_1,\Theta_1}(\cD,T) \cap \cK^1_{p_2,\theta_2,\Theta_2}(\cD,T).
\end{equation}
Put $\bar{u}_1:=u-v$. Then $\bar{u}=\bar{u}_1$ satisfies 
\begin{equation}
  \label{eqn 8.23.1}
d\bar{u}=\left[\cL \bar{u}+\sum_{i=1}^d \Big(\sum_{j=1}^d [\alpha^{ij}(t)-a^{ij}(\omega,t)]v_{x^j}\Big)_{x^i}\right]dt, \quad t\in(0,T].
\end{equation}
Also, due to \eqref{eqn 8.31.5}, equation \eqref{eqn 8.23.1} has a solution $\bar{u}_2\in \cK^1_{p_2,\theta_2,\Theta_2}(\cD,T)$.   Now note that for each fixed 
$\omega$, both $\bar{u}_1(\omega,\cdot,\cdot)$ and  $\bar{u}_2(\omega,\cdot,\cdot)$ satisfy equation  \eqref{eqn 8.23.1}, which we can consider as a deterministic equation with non-random operator.  By the above result for non-random operator we conclude 
$$\bar{u}_1(\omega,\cdot,\cdot)=\bar{u}_2(\omega,\cdot,\cdot)
$$  
for almost all $\omega$. From this we conclude that  both $v$ and $u-v$ are in $\cK^1_{p_2,\theta_2,\Theta_2}(\cD,T)$, and therefore the lemma is proved.
 \end{proof}

\mysection{SPDE on polygonal domains}\label{sec:polygonal domains}

In this section,  based on Theorem~\ref{main result-random}, we develop a regularity theory of the stochastic parabolic equations on polygonal domains in $\bR^2$.  This development is an enhanced version of the corresponding result  in \cite{CKL 2019+} in which 
$\cL=\Delta_x$ and  $\Theta=d$. Our generalization is as follows:
\begin{itemize}

\item{}
$\Delta \quad \rightarrow \quad \cL=\sum_{i,j}a^{ij}(\omega,t)D_{ij}$; operator with (random) predictable  coefficients 

\item{} 
$\Theta=2  \quad \rightarrow  \quad 1<\Theta<1+p$

\item{}
The restriction on  $\theta$ is weakened

\item{}
Sobolev regularity  with $\gamma\in \{-1,0,\cdots\}$\\
$ \quad  \rightarrow \quad$ Sobolev  and H\"older regularities with  real number $ \gamma \geq -1$

\end{itemize}

Let $\cO\subset \bR^2$ be a bounded polygonal domain with a finite number of vertices $\{p_1,\ldots,p_M\}\subset \partial \cO$.
For any $x\in\cO$, we denote
\begin{align*}
\rho(x):=\rho_{\cO}(x):=d(x,\partial\cO).
\end{align*}
In the polygonal domain,  the function  of $x$ defined by
$$
 \min_{1\leq m\leq M}\vert x-p_m\vert 
 $$
 will play the role of $\rho_{\circ, \cD}$, which is the distance to the vertex in an angular domain $\cD$. We  first construct a smooth version of the function $\min_{1\leq m\leq M}\vert x-p_m\vert g$ as follows.  Consider the domain $V:=\bR^2\setminus \{p_1,\cdots,p_M\}$ and note that  
$$\rho_V(x):=d(x,\partial V)=\min_{1\leq m\leq M}\vert x-p_m\vert .
$$
 Then, applying  \eqref{eqn 8.25.1} and \eqref{eqn 8.25.2} for $\rho_V$ and the domain $V$, we define $\psi_{V}$ and set
$$
\rho_{\circ}=\rho_{\circ,\cO}:=\psi_{V}.
$$
We can check that for any multi-index $\alpha$ and $\mu\in \bR$,
$$
\rho_{\circ} \sim \min_{1\leq m\leq M}\vert x-p_m\vert , \quad 
\sup_{\cO}\big\vert  \rho^{\vert \alpha\vert -\mu}_{\circ}D^{\alpha}\rho^{\mu}_{\circ}\big\vert <\infty.
$$
On the other hand, we also choose a smooth function $\psi=\psi_{\cO}$ such that $\psi\sim \rho_{\cO}$ and satisfies \eqref{eqn 8.9.1} with $\rho_{\cO}$ in place of $\rho_{\cD}$. 

Then, we recall the norms of the spaces $H^{\gamma}_{p,\Theta}(\cO)$ and $H^{\gamma}_{p,\Theta}(\cO; \ell_2)$ introduced in Definition \ref{defn 8.28};
 \begin{equation*}
  \|f\|^p_{H^{\gamma}_{p,\Theta}(\cO)}:= \sum_{n\in \bZ} e^{n\Theta} \|\zeta(e^{-n}\psi(e^n\cdot))f(e^{n}\cdot)\|^p_{H^{\gamma}_p(\bR^d)},
  \end{equation*}
  \begin{equation*}
  \|g\|^p_{H^{\gamma}_{p,\Theta}(\cO;\ell_2)}:= \sum_{n\in \bZ} e^{n\Theta} \|\zeta(e^{-n}\psi(e^n\cdot))g(e^{n}\cdot)\|^p_{H^{\gamma}_p(\bR^d;\ell_2)},
  \end{equation*}
where $\psi=\psi_{\cO}$. Using $\rho_{\circ,\cO}$  in place of 
 $\rho_{\circ,\cD}$,  and following Definition \ref{defn 8.19},  we define the function spaces
$$
K^{\gamma}_{p,\theta,\Theta}(\cO),  \quad K^{\gamma}_{p,\theta,\Theta}(\cO;\bR^d), \quad K^{\gamma}_{p,\theta,\Theta}(\cO;\ell_2),
$$ 
as well as the stochastic  spaces
$$\bK^{\gamma}_{p,\theta,\Theta}(\cO,T), \quad \bK^{\gamma}_{p,\theta,\Theta}(\cO,T,d ),\quad \bK^{\gamma}_{p,\theta,\Theta}(\cO,T,\ell_2),
$$
$$
\cK^{\gamma+2}_{p,\theta,\Theta}(\cO,T), \quad \bK^{\infty}_c(\cO,T), \quad \bK^{\infty}_c(\cO,T,\ell_2), \quad \bK^{\infty}_c(\cO).
$$
More specifically,  we write $f\in K^{\gamma}_{p,\theta,\Theta}(\cO)$ if and only if $\rho^{(\theta-\Theta)/p}_{\circ} f\in H^{\gamma}_{p,\Theta}(\cO)$, and define
  $$
  \|f\|_{ K^{\gamma}_{p,\theta,\Theta}(\cO)} := \|\rho^{(\theta-\Theta)/p}_{\circ} f\|_{H^{\gamma}_{p,\Theta}(\cO)}.
  $$
As in Section 2, if $\gamma\in \bN_0$, then we have
\begin{equation}
\label{eqn 8.25.8}
\|f\|^p_{ K^{\gamma}_{p,\theta,\Theta}(\cO)}  \sim \sum_{\vert \alpha\vert \leq \gamma}\int_{\cO} \vert \rho^{\vert \alpha\vert }D^{\alpha}f\vert ^p \rho^{\theta-\Theta}_{\circ}\rho^{\Theta-d} dx.
\end{equation}
\begin{defn}\label{definition solution polygon}
We  write $u\in\cK^{\gamma+2}_{p,\theta,\Theta}(\cO,T)$ if
 $u \in \bK^{\gamma+2}_{p,\theta-p,\Theta-p}(\cO,T)$ and there exist 
 $(\tilde{f}, \tilde{g}) \in\bK^{\gamma}_{p,\theta+p,\Theta+p}(\cO,T)\times \bK^{\gamma+1}_{p,\theta,\Theta}(\cO,T, \ell_2)$ and $u(0,\cdot)\in \bU^{\gamma+2}_{p,\theta,\Theta}(\cO)$ satisfying  
$$
du=\tilde{f}\,dt+\sum_k \tilde{g}^kdw^k_t,\quad t\in(0,T] 
$$
in the sense of distributions on $\cO$. The norm is defined by 
\begin{eqnarray*}
\|u\|_{\cK^{\gamma+2}_{p,\theta,\Theta}(\cO,T)}&:=&\|u\|_{\bK^{\gamma+2}_{p,\theta-p,\Theta-p}(\cO,T)}+\|\tilde{f}\|_{\bK^{\gamma}_{p,\theta+p,\Theta+p}(\cO,T)}
+\|\tilde{g}\|_{\bK^{\gamma+1}_{p,\theta,\Theta}(\cO,T,\ell_2)}\\
&&+\|u(0,\cdot)\|_{\bU^{\gamma+2}_{p,\theta,\Theta}(\cD)}.
\end{eqnarray*}
\end{defn}

  \begin{thm}
  \label{thm all}
  With $\cD$ replaced by $\cO$, all the claims of Lemma \ref{property1}, Remark \ref{dense space}, Theorem   \ref{banach}, Theorem \ref{embedding}, and  Lemma \ref{regularity.induction}  hold.  \end{thm}
  \begin{proof}
  All of these claims in Section 2 are proved based  on   \eqref{eqn 8.10.14},   \eqref{eqn 8.10.1}, and some properties of weighted Sobolev spaces $H^{\gamma}_{p,\Theta}(\cD)$ taken e.g from \cite{Lo1}.  Since these properties in \cite{Lo1} hold true on arbitrary domains, the exactly same proofs of Section 2 work with $\cD$ replaced by $\cO$. 
  \end{proof}
  
\begin{remark}
\label{remark 8.29}
For the analog of Theorem \ref{embedding} in the case of polygonal domain we do not need the additional condition for the initial condition. This is because since $\psi$ is bounded and $\beta>2/p$, by Lemma \ref{property1} (iv),  we have
$$
 \|\psi^{\beta-1}u(0,\cdot)\|_{L_p(\Omega;K^{\gamma+2-\beta}_{p,\theta,\Theta}(\cD))}\leq C \|\psi^{2/p-1}u(0,\cdot)\|_{L_p(\Omega;K^{\gamma+2-2/p}_{p,\theta,\Theta}(\cD))} \leq C \|u\|_{\cK_{p,\theta,\Theta}^{\gamma+2}}.
 $$

\end{remark}
  
For $m=1,\ldots,M$, let $\kappa_m$ denote the interior angle at the vertex $p_m$, and denote
\begin{equation*}
\kappa_0:=\max_{1\leq m\leq M}\kappa_m.
\end{equation*}
Also, for each $m$,  let $\cD_m$ denote the conic domain in $\bR^2$ such that 
$$
\cO \cap B_{\varepsilon}(p_m) \cap \{p_m+x: x\in \cD_m\} = \cO \cap B_{\varepsilon}(p_m)
$$
for all sufficiently small $\varepsilon>0$. Denote 
$$\lambda^{\pm}_{c,\cL,\cO}:=\min_{m}\lambda^{\pm}_{c,\cL, \cD_m} \quad \text{if $\cL$ is non-random}
$$
and 
$$\lambda_{c,\cO}(\nu_1,\nu_2):=\min_{m}\lambda_c(\nu_1,\nu_2,\cD_m)\quad \text{ if $\cL$ is random}.
$$
In Theorem \ref{main result polygon} below, we pose  the condition
\begin{equation}
\label{theta poly}
p(1-\lambda^+_{c,\cL,\cO})<\theta< p(1+\lambda^-_{c,\cL,\cO})
\end{equation}
if $\cL$ is non-random, and 
\begin{equation}
  \label{theta poly2}
  p(1-\lambda_{c,\cO}(\nu_1,\nu_2))<\theta <p(1+\lambda_{c,\cO}(\nu_1,\nu_2))
  \end{equation}
  if $\cL$ is random.  

  Here are our main results on polygonal domains. 
  
\begin{thm}[SPDE on polygonal domains with random or non-random coefficients]
\label{main result polygon}
Let $p\in[2,\infty)$, $\gamma \geq -1$,  and Assumption \ref{ass coeff} hold.  Also assume that
\begin{equation}
  \label{theta application}
1<\Theta<p+1,
\end{equation} 
and condition \eqref{theta poly} holds if $\cL$ is non-random, condition \eqref{theta poly2} holds if $\cL$ is random.  
Then for given $f^0\in\bK^{\gamma \vee 0}_{p,\theta+p,\Theta+p}(\domain,T)$, $\tbf=(f^1,\cdots,f^d) \in\bK^{\gamma +1}_{p,\theta,\Theta}(\domain,T,d)$,  $g\in\bK^{\gamma+1}_{p,\theta,\Theta}(\domain,T,\ell_2)$, and $u_0\in\bU^{\gamma+2}_{p,\theta,\Theta}(\cO)$, the equation
\begin{equation}\label{stochastic parabolic equation polygon}
d u =\left(\cL u+f^0+\sum_{i=1}^d f^i_{x^i}\right)dt +\sum^{\infty}_{k=1} g^kdw_t^k,\quad   t\in(0,T]\quad\,; \quad u(0,\cdot)=u_0
\end{equation}
admits  a unique solution $u$  in the class $\cK^{\gamma+2}_{p,\theta,\Theta}(\domain,T)$.
Moreover, the estimate
\begin{eqnarray*}
\|u\|_{\cK^{\gamma+2}_{p,\theta,\Theta}(\domain,T)}
&\leq& C\big(\|f^0\|_{\bK^{\gamma\vee 0}_{p,\theta+p,\Theta+p}(\domain,T)}
+\|\tbf\|_{\bK^{\gamma+1}_{p,\theta,\Theta}(\domain,T,d)}+\|g\|_{\bK^{\gamma+1}_{p,\theta,\Theta}(\domain,T,\ell_2)}\nonumber\\
&&\quad\quad+\|u_0\|_{\bU^{\gamma+2}_{p,\theta,\Theta}(\cO)}\big)
\end{eqnarray*}
holds with a constant $C=C(\domain,p,\gamma,\nu_1,\nu_2,\theta,\Theta,T)$.
\end{thm}

\begin{remark}
Since $d=2$ in this section, the range of $\Theta$ in \eqref{theta application} coincides with $(d-1,d-1+p)$ which we have kept throughout this article.
\end{remark}

\begin{thm}[H\"older estimates on polygonal domains]
\label{cor 8.23}
Let $p\geq 2$, $\theta, \Theta\in \bR$ and $u\in \cK^{\gamma+2}_{p,\theta,\Theta}(\cO,T)$.

(i) If $\gamma+2-\frac{d}{p}\geq n+\delta$, where $n\in \bN_0$ and $\delta\in (0,1)$, then for any $k\leq n$, 
$$
\vert \rho^{k-1+\frac{\Theta}{p}} \rho^{(\theta-\Theta)/p}_{\circ} D^{k}u(\omega,t,\cdot)\vert _{\cC(\cO)}+
 [\rho^{n-1+\delta+\frac{\Theta}{p}} \rho^{(\theta-\Theta)/p}_{\circ} D^{k} u(\omega,t,\cdot)]_{\cC^{\delta}(\cO)}<\infty 
$$
holds for almost all $(\omega,t)$.  In particular, 
\begin{equation*}
 \vert u(\omega,t,x)\vert \leq C(\omega,t) \rho^{1-\frac{\Theta}{p}}(x) \rho^{(-\theta+\Theta)/p}_{\circ}(x).
\end{equation*}

(ii) Let
$$
2/p<\alpha<\beta\leq 1, \quad \gamma+2-\beta-d/p \geq m+\varepsilon,
$$
where $m\in \bN_0$ and $\varepsilon\in (0,1]$.  Put $\eta=\beta-1+\Theta/p$. Then for any $k\leq m$,
\begin{eqnarray*}
&&\bE \sup_{t,s\leq T}  \frac {\big\vert \rho^{\eta+k}  \rho^{(\theta-\Theta)/p}_{\circ} \left(D^ku(t)-D^ku(s)\right)\big\vert ^p_{\cC(\cO)}}
{\vert t-s\vert ^{p\alpha/2-1}}<\infty, \\
&& \bE \sup_{t,s\leq T}  \frac {\left[\rho^{\eta+m+\varepsilon}  \rho^{(\theta-\Theta)/p}_{\circ} \left(D^mu(t)-D^mu(s)\right)\right]^p_{\cC^{\varepsilon}(\cO)}}
{\vert t-s\vert ^{p\alpha/2-1}} <\infty. 
\end{eqnarray*}

\end{thm}
\begin{proof}
The claims follow from the corresponding results of \eqref{eqn 8.21.1} and \eqref{eqn 8.10.10} mentioned in Theorem \ref{thm all}.
\end{proof}

For the proof of  Theorem \ref{main result polygon}, we first prove the following estimate.

\begin{lemma}[A priori estimate]
\label{a priori p}
Let Assumptions in Theorem \ref{main result polygon} hold. Then there exists a constant $C=C(d,p,\theta,\Theta,\nu_1,\nu_2,\cO,T)$ such that  the a priori estimate
\begin{eqnarray}
\|u\|_{\cK^{\gamma+2}_{p,\theta,\Theta}(\domain,T)}&\leq& C\big(\|f^0\|_{\bK^{\gamma\vee 0}_{p,\theta+p,\Theta+p}(\domain,T)}
+\|\tbf\|_{\bK^{\gamma+1}_{p,\theta,\Theta}(\domain,T,d)}+\|g\|_{\bK^{\gamma+1}_{p,\theta,\Theta}(\domain,T,\ell_2)}\nonumber\\
&&\quad\quad+ \|u_0\|_{\bU^{\gamma+2}_{p,\theta,\Theta}(\cO)}\big)\label{polygon a priori}
\end{eqnarray}
holds provided  that a solution $u\in \cK^{\gamma+2}_{p,\theta,\Theta}(\cD,T)$ to equation \eqref{stochastic parabolic equation polygon}  exists. 
\end{lemma}

\begin{proof}
First, choose a sufficiently small constant $r>0$ such that each $B_{3r}(p_m)$ contains only one vertex $p_m$ and intersects with only two edges for each $m=1,\,\ldots,\,M$. Then we choose a function $\xi\in\cC_c^{\infty}(\bR^2)$ satisfying
\begin{align*}
1_{B_r(0)}(x)\leq \xi(x)\leq 1_{B_{2r}(0)}(x)\quad\text{for all }x\in\bR^2.
\end{align*}

Let $\xi_m(x):=\xi(x-p_m)$ and $\xi_0:=1-\sum_{m=1}^M\xi_m$.
By the choice of $r$ and $\xi$, $supp(\xi_m)$s are disjoint and hence $0\leq \xi_0\leq 1$.
Moreover, $\xi_0(x)=1$ if $\rho_{V}(x)>2r$.

For $m=1,\ldots,M$, let $\cD_m$ be the angular (conic) domain centered at $p_m$ with interior angle $\kappa_m$ such that $\cD_m\cap B_{3r}(p_m)=\cO\cap B_{3r}(p_m)$.

Now let $G$ be a $C^1$-domain in $\cO$ such that
$$
\xi_0(x)=0\quad\text{for }x\in\cO\setminus G\quad\text{and}\quad \inf_{x\in G} \rho_{\circ}(x)\geq c>0\text{ with a constant}\; c.
$$
Then, due to the choices of  $\xi_m$ and  $\cD_m$ $(m=1,\ldots,M)$,    \eqref{space K norm} and \eqref{eqn 8.25.8}  together easily yield
\begin{align*}
\|\xi_m v\|^p_{K^n_{p,\theta,\Theta}(\cO)} \sim\|\xi_m v\|_{K^n_{p,\theta,\Theta}(\cD_m)},\quad m=1,\ldots,M.
\end{align*}
for any $\theta,\Theta\in\bR$, $n\in\{0,1,2,\ldots\}$, and $v\in K^n_{p,\theta,\Theta}(\cO)$.  Similarly,
$$
\|\xi_0 v\|^p_{K^n_{p,\theta,\Theta}(\cO)}\sim \int_{G} \vert \xi_0v\vert ^p \rho^{\Theta-d}dx\sim \|\xi_0 v\|^p_{H^n_{p,\Theta}(G)},
$$
and the same relations hold for $\ell_2$-valued functions. Denote
$$
\bH^{\gamma}_{p,\Theta}(G,T):=L_p(\Omega\times (0,T], \cP; H^{\gamma}_{p,\Theta}(G)),
$$
$$
 \bH^{\gamma}_{p,\Theta}(G,T,\ell_2):=L_p(\Omega\times (0,T], \cP; H^{\gamma}_{p,\Theta}(G;\ell_2)).
$$
Then, the above observations  in particular imply
\begin{align}\label{partition-of-unity.eq}
\|v\|_{\bK^n_{p,\theta,\Theta}(\cO,T)} \sim \Big(\|\xi_0 v\|_{\bH^n_{p,\Theta}(G,T)}+\sum_{m=1}^M\|\xi_m v\|_{\bK^n_{p,\theta,\Theta}(\cD_m,T)}\Big)
\end{align}
for any $v\in \bK^n_{p,\theta,\Theta}(\cO,T)$, where $n\in \{0,1,2,\cdots\}$. 

Now, for each $m=1,\ldots,M$ we define $u_m:=\xi_mu$. Then, since $\gamma+2\geq 1$,  $u_m$ belongs to $\bK^{1}_{p,\theta-p,\Theta-p}(\cD_m,T)$. Also, $\xi_0 u$ belongs to $\bH^{1}_{p,\Theta-p}(G,T)$.  Note that each $u_m$ satisfies
\begin{equation}\label{equation for m}
d(u_m)=\Big(\cL u_m+f^0_m+\sum_{i=1}^d (f^i_m)_{x^i}\Big)dt+\sum_k g^{k}_m dw_t^k,\quad t\in (0,T]
\end{equation}
in the sense of distributions on  $\cD_m$ with the initial condition $
u_m(0,\cdot)=\xi_m u_0$ and $\xi_0 u$ satisfies 
\begin{equation}\label{equation for m=0}
d(\xi_0 u)=\Big(\cL (\xi_0 u)+f^0_0+\sum_{i=1}^d (f^i_0)_{x^i}\Big)dt+\sum_k g^{k}_0 dw_t^k,\quad t\in (0,T]
\end{equation}
in the sense of distributions on $G$ with the initial condition $w(0,\cdot)=\xi_0 u_0$,  where
\begin{equation}\label{f g for m}
f^0_m=f^0\xi_m-\sum_{i=1}^d f^i (\xi_m)_{x^i}-u \cL(\xi_{m}), \quad f^i_m=2\sum_{j=1}^d a^{ij}u\,(\xi_{m})_{x^j},\quad
g_m=\, g\xi_m
\end{equation}
for $m=0,1,2,\ldots,M$.

Since $supp(\xi_m) \subset \overline{B_{2r}(p_m)}$ and $(\xi_m)_x=0$ on a neighborhood of $p_m$ for $m=1,\ldots,M$, we have
\begin{align*}
\|u (\xi_m)_x\|_{\bL_{p,\theta,\Theta}(\cO,t)}+\|u(\xi_m)_{xx}\|_{\bL_{p,\theta+p,\Theta+p}(\cO,t)}\leq C\|u\|_{\bL_{p,\theta,\Theta}(\cO,t)}
\end{align*}
for $t\leq T$, where $C$ depends only on $\cO,\,p,\,\theta$ and $\Theta$.

Hence, for $m=1,\,\ldots,\,M$, by Theorems \ref{main result} and \ref{main result-random}, which our range of $\theta$ allows us to use,   we have for any $t\leq T$,
\begin{align*}
&\quad \quad \|\xi_m u\|_{\bK^1_{p,\theta-p,\Theta-p}(\cD_m,t)}
\\
& \leq C\big(\|f^0_m\|_{\bL_{p,\theta+p,\Theta+p}(\cD_m,t)}+ \sum_{i=1}^d \|f^i_m\|_{\bL_{p,\theta,\Theta}(\cD_m,t)}+\|g_m\|_{\bL_{p,\theta,\Theta}(\cD_m,t,\ell_2)}\nonumber\\
&\hspace{1cm}+\|\xi_m u_0\|_{\bU^1_{p,\theta,\Theta}(\cD_m)}\big)
\\
& \leq C\big(\|u\|_{\bL_{p,\theta,\Theta}(\cO,T)}+\|f^0\|_{\bL_{p,\theta+p,\Theta+p}(\cO,T)}
+\|\tbf\|_{\bL_{p,\theta,\Theta}(\cO,T,d)} +\|g\|_{\bL_{p,\theta,\Theta}(\cO,T,\ell_2)}\\
&\hspace{1cm}+\|u_0\|_{\bU^1_{p,\theta,\Theta}(\cO)}\big).
\end{align*}
  
  For $m=0$, by  \cite[Theorem 2.7]{Kim2004-2} (or \cite[Theorem 2.9]{Kim2004}),  
 we  have
\begin{align*}
&\quad \quad\quad \|\xi_0 u\|_{\bH^1_{p,\Theta-p}(G,t)}\\
&\leq C\Big(\|f^0_0\|_{\bL_{p,\Theta+p}(G,t)}+ \sum_{i=0}^d \|f^i_0\|_{\bL_{p,\Theta}(G,t)}+\|g_0\|_{\bL_{p,\Theta}(G,t,\ell_2)}+\|\xi_0u_0\|_{L_p(\Omega;H^{1-2/p}_{p,\Theta+2-p}(G))}\Big)
\\
&\leq C\Big(\|u\|_{\bL_{p,\theta,\Theta}(\cO,T)}+\|f^0\|_{\bL_{p,\theta+p,\Theta+p}(\cO,T)}
+\|\tbf\|_{\bL_{p,\theta,\Theta}(\cO,T,d)} +\|g\|_{\bL_{p,\theta,\Theta}(\cO,T,\ell_2)}\\
&\hspace{1cm}+ \|u_0\|_{\bU^1_{p,\theta,\Theta}(\cO)}\Big).
\end{align*}
Summing up over all $m=0,\,\ldots,\,M$ and using \eqref{partition-of-unity.eq}, for each $t\leq T$, we have
\begin{align*}
&\quad\quad\|u\|_{\bK^1_{p,\theta-p,\Theta-p}(\cO,t)}
\\
&\leq C \Big(\|u\|_{\bL_{p,\theta,\Theta}(\cO,t)}+\|f^0\|_{\bL_{p,\theta+p,\Theta+p}(\cO,T)}
+\|\tbf\|_{\bL_{p,\theta,\Theta}(\cO,T)} +\|g\|_{\bL_{p,\theta,\Theta}(\cO,T,\ell_2)}\\
&\hspace{1cm}  \|u_0\|_{\bU^1_{p,\theta,\Theta}(\cO)}\Big).
\end{align*}
Using this and  the polygonal versions of \eqref{eqiv norm} and \eqref{eqn 8.25.31},  which mentioned in Theorem \ref{thm all},  we get, for each $t\leq T$,
\begin{align*}
&\quad\quad \|u\|^p_{\cK^1_{p,\theta,\Theta}(\cO,t)}\\
\leq &\,C\int^t_0\|u\|^p_{\cK^1_{p,\theta,\Theta}(\cO,s)}ds
\\
&+C\left(\|f^0\|^p_{\bL_{p,\theta+p,\Theta+p}(\cO,T)}+\|\tbf\|^p_{\bL_{p,\theta,\Theta}(\cO,T)}+\|g\|^p_{\bL_{p,\theta,\Theta}(\cO,T,\ell_2)}+\|u_0\|_{\bU^1_{p,\theta,\Theta}(\cO)}\right).
\end{align*}
Applying Gronwall's inequality, we further obtain
\begin{align*}
&\quad\quad \|u\|_{\cK^1_{p,\theta,\Theta}(\cO,T)}\\
&\leq C\left(\|f^0\|_{\bL_{p,\theta+p,\Theta+p}(\cO,T)}+\|\tbf\|_{\bL_{p,\theta,\Theta}(\cO,T)}+\|g\|_{\bL_{p,\theta,\Theta}(\cO,T,\ell_2)}+\|u_0\|_{\bU^1_{p,\theta,\Theta}(\cO)}\right).
\end{align*}
This and the polygonal version of Lemma \ref{regularity.induction}, which is mentioned in Theorem \ref{thm all},  yield a priori estimate \eqref{polygon a priori}.
The lemma is proved.
\end{proof}


The following is a $\cC^1$-domain version of Lemma \ref{global.uniqueness}. We use it in the  proof of Theorem \ref{main result polygon} below.

\begin{lemma}\label{lem for uniqueness2}
Let $G$ be a bounded $\cC^1$ domain in $\bR^d$ and let $p_j\in[2,\infty)$, $\Theta_j\in (d-1,d-1+p_j)$ for $j=1,2$. Assume that $u\in \bH^1_{p_1,\Theta_1-p_1}(G,T)$
satisfies
\begin{align*}
du=\Big(\cL u+f^0+\sum_{i=1}^d f^i_{x^i}\Big)\,dt+\sum_kg^kdw^k_t,\quad t\in(0,T]\quad 
\end{align*}
 in the sense of distributions on $G$ with the initial condition $u(0,\cdot)=u_0(\cdot)$ and $f^0$, $f^i$ ($i=1,2,\ldots,$), $g$, $u_0$ satisfying
$$
f^0\in \bL_{p_j,\Theta_j+p_j}(G,T)\cap \bL_{p_j,d+p_j}(G,T),\quad f^i \in \bL_{p_j,\Theta_j}(G,T)\cap \bL_{p_j,d}(G,T), \, i=1,\cdots,d,
$$
$$
g\in \bL_{p_j,\Theta_j}(G,T,\ell_2)\cap \bL_{p_j,d}(G,T,\ell_2),
$$
$$
\quad u_0\in L_p(\Omega,\rF_0;H^{1-2/p_j}_{p_j,\Theta_j+2-p_j}(G))\cap  L_p(\Omega,\rF_0;H^{1-2/p_j}_{p_j,d+2-p_j}(G))
$$
for both $j=1$ and $j=2$.    
Then $u$ belongs to $\bH^1_{p_2,\Theta_2-p_2}(G,T)$. 
    
\end{lemma}

\begin{proof}
See \cite[Lemma 3.8]{CKL 2019+}. We remark that  only $\Delta$ is considered in \cite{CKL 2019+}, however the proof of   \cite[Lemma 3.8]{CKL 2019+} works for general case without any changes since the proof depends only on   \cite[Theorem 2.7]{Kim2004-2} (or \cite[Theorem 2.9]{Kim2004}), 
which involves operators having coefficients measurable in $(\omega,t)$ and continuous in $x$.
\end{proof}

We recall $d=2$ in this section.
\begin{proof}[\textbf{Proof of Theorem  \ref{main result polygon}}]

Due to Lemma \ref{a priori p}, we only need to prove the existence result. Furthermore, relying on standard approximation argument, we may assume
$$
f^0\in\bK^{\infty}_c(\cO,T),\quad \tbf\in\bK^{\infty}_c(\cO,T,2), \quad  g\in\bK^{\infty}_c(\cO,T,\ell_2)\quad u_0\in \bK^{\infty}_c(\cO).
$$
Considering $u-u_0$ as usual, we may assume $u_0\equiv 0$.
Also, note that $g^k=0$ for all large $k$ (say, for all $k> N$), and each $g^k$ is of the type $\sum_{j=1}^{n(k)} 1_{(\tau^k_{j-1},\tau^k_j]}(t) h^{kj}(x)$, where $\tau^k_j$ are bounded stopping times and $h^{jk}\in \cC^{\infty}_c(\cO)$. 
Thus the function $v$ defined by
$$
v(t,x):=\sum_{k=1}^{\infty}\int^t_0 g^k dw^k_s=\sum_{k\leq N} \sum_{j\leq n(k)} \left(w^k_{\tau^k_j \wedge t}-w^k_{\tau^k_{j-1} \wedge t}\right) h^{kj}(x)
$$
is  infinitely differentiable in $x$ and vanishes near the boundary of $\cO$. Consequently $v$ belongs to $\cK^{\nu+2}_{p,\theta,\Theta}(\cO,T)$ for any $\nu, \theta, \Theta \in \bR$ as we consult with Definition \ref{definition solution polygon}.  Now, $u$ satisfies equation \eqref{stochastic parabolic equation polygon} if and only if $\bar{u}:=u-v$ satisfies
$$
d\bar{u}=\Big(\cL\bar{u}+\bar{f}^0+\sum_{i=1}^2 f^i_{x^i}\Big)dt, \quad t\in(0,T]\quad;\quad \bar{u}(0,\cdot)= 0,
$$
where $\bar{f}^0=f^0+\cL v$.  Hence, considering $\bar{f}^0$ in place of $f^0$, to prove the existence we may further assume $g=0$.

Then, by the  classical results without weights  for $p=2$,  (see, e.g.  \cite{Roz1990} or \cite[Theorem~2.12, Corollary~2.14]{Kim2014}),
there exists a  solution $u$ in $\cK^1_{2,2,2}(\cO,T)$  to equation \eqref{stochastic parabolic equation polygon}, which now is simplified as
$$
u_t=\cL u+f^0+\sum_{i=1}^2 f^i_{x^i}, \quad t\in(0,T]\quad; \quad  u(0,\cdot)=0.
$$
By Theorem A in \cite{Aronson} (or see  estimate (2.11) and proof of Theorem 2.4 in  \cite{Kim2004-3} for more detail), for any $r>4$, we have
\begin{equation}
 \label{bound}
\bE \sup_{t,x} \vert u(t,x)\vert ^p \leq C \bE \|\,\vert f^0\vert +\vert \tbf\vert \,\|^p_{L_r((0,T]\times \cO))} <\infty.
\end{equation}

Now we prove $u\in \cK^1_{p,\theta,\Theta}(\cO,T)$ using Lemma \ref{global.uniqueness} and  Lemma \ref{lem for uniqueness2} along with $u\in \cK^1_{2,2,2}(\cO,T)$.   Define $u_m:=\xi_m u$ in the same way we did in  the proof of Lemma \ref{a priori p}. Then $\xi_m u$ satisfies \eqref{equation for m} in the sense of distributions on  $\cD_m$ for $m=1,\ldots,M$ and $\xi_0u$ satisfies \eqref{equation for m=0}  
 on $G$ for $m=0$ with the same $f^0_m,\,f^i_m,\, \xi_mu_0$ as in \eqref{f g for m}. 
Note that since $f^0,\tbf$ are bounded and $f^0,\tbf, (\xi_m)_x, (\xi_m)_{xx}$ vanish near vertices,  we have for any $\theta\in \bR$, $q\geq 2$ and $1<\Theta<1+q$, 
\begin{eqnarray*}
&&\|f^0_m\|^q_{\bL_{q,\theta+q,\Theta+q}(\cO,T)} +\sum_{i=1}^2\|f^i_m\|^q_{\bL_{q,\theta,\Theta}(\cO,T)}\\
 &\leq& C \bE \int^T_0 \int_{\cO}(1+|u|^q) \rho^{\Theta-2}dx\leq C \left(\int_{\cO}\rho^{\Theta-2}dx\right) \bE \sup_{t,x}(1+|u|^q)<\infty.
\end{eqnarray*}
For the last inequality we used \eqref{bound} and the fact $\Theta-2>-1$. 
Hence, $f^0_m, \,f^i_m$ along with $\xi_mu_0$ satisfy assumptions in Lemma \ref{global.uniqueness} and  Lemma \ref{lem for uniqueness2}. Consequently  $\xi_m u \in \bK^1_{p,\theta-p,\Theta-p}(\cD_m,T)$ as $\xi_m u \in \cK^1_{p,\theta,\Theta}(\cD_m,T)$ for $m=1,2,\cdots, M$ and $\xi_0 u \in \bH^1_{p,\Theta-p}(G,T)$.  These and  \eqref{partition-of-unity.eq}  with $n=1$ yield $u\in \bK^1_{p,\theta-p,\Theta-p}(\cO,T)$ and in turn $u\in \cK^1_{p,\theta,\Theta}(\cO,T)$.  

Finally, the analogy of of Lemma \ref{regularity.induction} in case of polygonal domains (see Theorem \ref{thm all}) proves that  the solution $u$ found above actually belongs to the space $u\in \cH^{\gamma+2}_{p,\theta,\Theta}(\cO,T)$.  The theorem is proved.  \end{proof}

\providecommand{\bysame}{\leavevmode\hbox to3em{\hrulefill}\thinspace}
\providecommand{\MR}{\relax\ifhmode\unskip\space\fi MR }
\providecommand{\MRhref}[2]{%
  \href{http://www.ams.org/mathscinet-getitem?mr=#1}{#2}
}
\providecommand{\href}[2]{#2}

\end{document}